\documentclass[final,11pt,a4paper]{amsart} 

\usepackage[all]{xy}
\usepackage{placeins}
\usepackage{enumitem}
\usepackage{amssymb}
\usepackage{latexsym}
\usepackage{amsmath}
\usepackage{mathrsfs}
\usepackage{array,booktabs}
\usepackage{verbatim}
\usepackage{amsthm}
\usepackage[title]{appendix}
\usepackage{hyperref}	

\usepackage[normalem]{ulem}

\usepackage{graphicx}	

\usepackage[dvipsnames]{xcolor}

\usepackage{fullpage}

\theoremstyle{plain}
\newtheorem{theorem}{Theorem}[section]
\newtheorem*{theorem*}{Theorem}
\newtheorem*{mtheorem*}{Main Theorem}
\newtheorem{lemma}[theorem]{Lemma}
\newtheorem{corollary}[theorem]{Corollary}
\newtheorem{proposition}[theorem]{Proposition}
\newtheorem{observation}[theorem]{Observation}

\theoremstyle{definition}
\newtheorem{remark}[theorem]{Remark}

\newtheorem*{notation*}{Notation}
\newtheorem{example}[theorem]{Example}
\newtheorem*{example*}{Example}
\newtheorem{definition}[theorem]{Definition}

\newcommand{\Ext}{\mathrm{Ext}}
\newcommand{\Hom}{\mathrm{Hom}}
\newcommand{\End}{\mathrm{End}}
\newcommand{\Prod}[1]{\mathrm{Prod}(#1)}
\newcommand{\id}[2]{\mathrm{id}_{#1}(#2)}
\newcommand{\pd}[2]{\mathrm{pd}_{#1}(#2)}
\newcommand{\Mod}[1]{\mathrm{Mod}(#1)}

\renewcommand{\mod}[1]{\mathrm{mod}(#1)}

\newcommand{\PI}{\mathrm{PE}}

\newcommand{\Ab}{\mathrm{Ab}}

\newcommand{\FunR}{\mathcal{F}(R)}
\newcommand{\Pinj}[1]{\mathrm{Pinj}(#1)}
\newcommand{\Inj}[1]{\mathrm{Inj}(#1)}
\newcommand{\Cogen}[1]{\mathrm{Cogen}(#1)}

\renewcommand{\k}{\mathrm{K}}
\DeclareMathAlphabet{\mathpzc}{OT1}{pzc}{m}{it}

\newcommand{\ann}[1]{\mathrm{Ann}(#1)}

\renewcommand{\dim}[1]{\mathrm{dim}_k(#1)}

\newcommand{\Tor}{\mathrm{Tor}}

\begin{document} 

\title{Classification of cosilting modules in type $\tilde{A}$}
\author{Karin Baur and Rosanna Laking}
\address{Karin Baur, School of Mathematics, University of Leeds, Leeds, LS2 9JT, UK}
\email{pmtkb@leeds.ac.uk}
\urladdr{http://www1.maths.leeds.ac.uk/~pmtkb/}
\address{Rosanna Laking,  Dipartimento di Informatica - Settore di Matematica, Universit\`a degli Studi di Verona, Strada le Grazie 15 - Ca' Vignal, I-37134 Verona, Italy} 
\email{rosanna.laking@univr.it}
\urladdr{http://profs.scienze.univr.it/laking/}

\begin{abstract}
Torsion pairs in the category of finitely presented modules over a noetherian ring can be parametrised by the class of cosilting modules.  In this paper, we characterise such modules in terms of their indecomposable summands, providing a new approach to 
the classification of torsion pairs. In particular, we 
classify cosilting modules over cluster-tilted algebras of type $\tilde{A}$. We do this by 
using a geometric model for finite- and infinite-dimensional modules over such algebras.
\end{abstract}
\subjclass[2010]{16G10, 16G20, 16S90, 16E05}
\keywords{cosilting, pure-injective, cluster-tilted, torsion pair, geometric model, gentle algebra}
\thanks{
The first author was supported by FWF grants P30549-N26 and DK1230. She is supported by a Royal Society 
Wolfson Research Merit Award. Currently, she is on leave from the University of Graz. 
The second author was supported by the Max Planck Institute for Mathematics and also by the 
European Union's Horizon 2020 research and innovation programme under the 
Marie Sk{\l}odowska-Curie Grant Agreement No.~797281.
}

\date{}

\maketitle

Cosilting modules were first introduced by Breaz and Pop (\cite{BreazPop}) as a dual notion to silting modules (\cite{AMV}).  These modules are a common generalisation of (large) cotilting modules and support $\tau^{-1}$-tilting modules over finite-dimensional algebras.  The next theorem establishes a close relationship between cosilting modules in the module category $\Mod{R}$ and the torsion pairs in the subcategory category $\mod{R}$ of noetherian modules when $R$ is left noetherian.  It follows immediately from the combination of two known results in the literature.  The notation used in the bijection is defined in the notation section at the end of the introduction.

\begin{theorem*}[{\cite[Cor.~3.9]{abundance}, \cite[Lem.~4.4]{CBLFP}}]
For a left noetherian ring $R$, the assignment sending a cosilting module $C$ to the torsion pair $({}^{\perp_0}C\cap \mod{R}, \Cogen{C}\cap\mod{R})$ induces a bijection between the following sets.\begin{enumerate}
\item The set of equivalence classes of cosilting modules $C$ in $\Mod{R}$.
\item The set of torsion pairs $(\mathcal{X}, \mathcal{Y})$ in $\mod{R}$.
\end{enumerate} 
\end{theorem*}

Torsion pairs play a fundamental role in localisation theory (see, for example, \cite{BelRei}) and also provide a direct connection between the module category and t-structures in the derived category (\cite{HRS}).  Tilting theory has played a major role in the study of torsion pairs in module categories, but this often only yields a restricted class of torsion pairs.  For example, under the above bijection, the cotilting modules correspond to the torsion pairs $(\mathcal{X}, \mathcal{Y})$ such that $\mathcal{Y}$ generates $\mod{R}$ (\cite[Thm.~A]{BuanKrause}) and, if $R$ is a finite-dimensional algebra, then the support $\tau^{-1}$-tilting modules correspond to the torsion pairs $(\mathcal{X}, \mathcal{Y})$ such that $\mathcal{Y}$ is functorially finite (\cite[Thm.~2.15]{AIR}).  More recently the lattice of all torsion pairs in $\mod{R}$ has been investigated in the case where $R$ is a finite-dimensional algebra via the study of bricks in $\mod{R}$ (\cite{DIRRT}). 

The above theorem suggests a different approach to the classification of torsion pairs in $\mod{R}$: classify equivalence classes of cosilting modules in $\Mod{R}$. With a view to following this approach, we show in Theorem \ref{Thm: cotilt vs rigid} that equivalence classes of cotilting modules over left artinian rings are parametrised by maximal sets of pairwise $\Ext$-orthogonal indecomposable pure-injective modules of injective dimension less than or equal to $1$.  This yields Corollary \ref{cor: max rig cosilt}, which is a new characterisation of cosilting modules in the same setting.

From this starting point we consider the case of cosilting modules over cluster-tilted algebras of type $\tilde{A}$.  This family of algebras was introduced by Buan, Marsh and Reiten (\cite{BMR}) in the context of cluster theory and has been extensively studied since.  Such algebras can be characterised as surface algebras over the annulus (\cite{ABCJP}) and, in particular, all indecomposable finite-dimensional modules are known to be string or band modules.  Moreover, it is possible to express all finite-dimensional string modules as arcs between marked points on the surface.  

In the same way, we express the infinite-dimensional (indecomposable pure-injective) string modules as \emph{asymptotic arcs} in the surface, starting at a marked point and converging towards the unique non-contractible closed curve (see \cite{BBM, BD}).  The remaining infinite-dimensional indecomposable pure-injective modules, which were classified in \cite{PP}, consist of band modules parametrised by $\k^* \times \{\infty, -\infty\}$, as well as the unique generic module.  We represent this family of band modules, together with the finite-dimensional band modules, by the non-contractible closed curve in the surface.   Using this geometric model of the pure-injective modules, we establish the following classification of the cosilting modules up to equivalence.  We use the term \emph{arc} to refer to either a finite or asymptotic arc up to homotopy.

\begin{mtheorem*}[Theorem \ref{Thm: main theorem} and Remark \ref{Rem: parametrising sets}]
Let $A(\Gamma)$ be a cluster-tilted $\k$-algebra of type $\tilde{A}$ where $\Gamma$ is a triangulation of the annulus $\mathrm{S}$ with marked points $\mathrm{M}$.  There are the following bijective correspondences: \[ \mathcal{A}(\mathrm{S}, \mathrm{M}) \times \mathcal{P}(\k^*) \: \: \overset{1\text{-}1}{\longleftrightarrow} \: \: \mathrm{Cosilt}\text{-}A(\Gamma)\]
 \[\mathcal{T}(\mathrm{S}, \mathrm{M}) \: \: \overset{1\text{-}1}{\longleftrightarrow} \: \: \mathrm{cosilt}\text{-}A(\Gamma)\]
 \end{mtheorem*}

\noindent Where $\mathcal{A}(\mathrm{S}, \mathrm{M})$ is the set of maximal collections of non-crossing 
arcs in $(\mathrm{S}, \mathrm{M})$ that contain at least one asymptotic arc; $\mathcal{P}(\k^*)$ is the powerset of $\k^*$;  
$\mathcal{T}(\mathrm{S}, \mathrm{M})$ is the set of maximal collections of non-crossing arcs in 
$(\mathrm{S}, \mathrm{M})$ that contain no asymptotic arcs; $\mathrm{Cosilt}\text{-}A(\Gamma)$ is the set of infinite-dimensional cosilting $A(\Gamma)$-modules up to equivalence; and $\mathrm{cosilt}\text{-}A(\Gamma)$ is the set of finite-dimensional cosilting $A(\Gamma)$-modules up to equivalence.

The diagram below depicts an element of $\mathcal{A}(\mathrm{S}, \mathrm{M})$ where $(\mathrm{S}, \mathrm{M})$ is as in Figure~\ref{fig:triangulation}.  By the Main Theorem, each subset of $\k^*$ determines an infinite-dimensional cosilting module with the indecomposable pure-injective modules corresponding to the arcs $\alpha_1, \alpha_2, \alpha_3, \alpha_4, \alpha_5$ occurring as direct summands.
\[
\includegraphics[width=4cm]{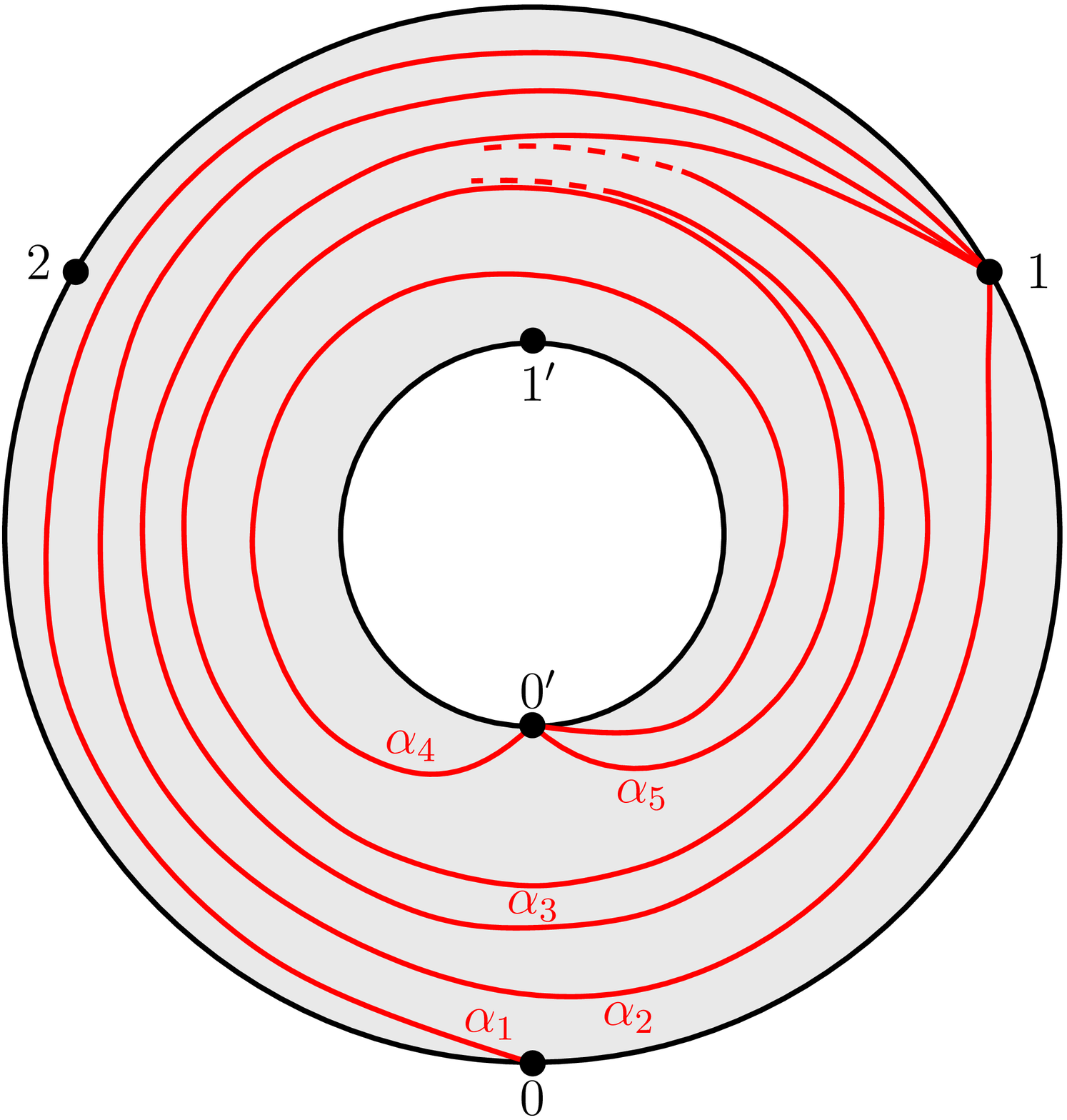}
\]

This classification encompasses some known results, including the classification of cotilting modules over the path algebra of a quiver of type $\tilde{A}$ (\cite{BuanKrause}) and the classification of torsion pairs in a tube category in terms of asymptotic arcs (\cite{BBM}); see Remark \ref{Rem: recovery}.   We also note that there is a significant body of existing work on the classification of torsion pairs in cluster categories via geometric models \cite{GHJ, Ng, HJR1, HJR2, HJR3, ZZZ}.

We end this introduction with a summary of the contents of the article.  The first section is dedicated to giving a characterisation of cosilting modules over a left artinian ring $R$ in terms of certain classes of indecomposable pure-injective modules (Corollary \ref{cor: max rig cosilt}).  To do this, we first introduce the fundamental notions from the theory of purity (Section \ref{Sec: purity}) and then provide a characterisation of cotilting $R$-modules (Section \ref{Sec: art cotilt}).  Section \ref{Sec: cluster-tilted} of the paper focuses on cluster-tilted algebras of type $\tilde{A}$.  In Section \ref{Sec: type A}, we introduce the algebras and the geometric model.  In Section \ref{sec: max rig} we provide several preliminary results that will lead us to the classification.  In particular, we describe the indecomposable pure-injective modules (Section \ref{sec: ind pinj}); determine their injective dimensions (Section \ref{Sec: inj dim}); and describe extensions between them (Section \ref{Sec: Extensions in A}).  The main result of Section \ref{Sec: Extensions in A} is Theorem \ref{Thm: extensions}, which establishes a correspondence between certain crossings of arcs in the surface and extensions between the corresponding modules.  A large portion of the proof of the theorem makes use of the combinatorial description of the derived category of $A(\Gamma)$; we postpone these arguments until Appendix \ref{App: Ext}.  The final section (Section \ref{Sec: Asymp_tri}) consists of the proof of Theorem \ref{Thm: main theorem}.

\begin{notation*}\label{Not: not}
Throughout $R$ denotes an associative unital ring.  The category of left $R$-modules is denoted by $\Mod{R}$ and the full subcategory of finitely presented left $R$-modules is denoted by $\mod{R}$.  

Let $N$ be a left $R$-module and $I$ a set.  We use the notation $N^I$ to denote the direct product of copies of $N$ indexed by $I$.  We use the notation $N^{(I)}$ to denote the direct sum of copies of $N$ indexed by $I$.

For a set $\mathcal{N}$ of modules we use the notation $\Prod{\mathcal{N}}$ to denote the set of direct summands of direct products of copies of objects in $\mathcal{N}$.  If $\mathcal{N}=\{N\}$ then we will write $\Prod{N}$ for $\Prod{\mathcal{N}}$.

We denote the set of submodules of modules in $\Prod{\mathcal{N}}$ (respectively in $\Prod{N}$) by $\Cogen{\mathcal{N}}$ (respectively $\Cogen{N}$) for a set $\mathcal{N}$ of modules or a module $N$.  For a module $N$, define the sets ${}^{\perp_1} N := \{ M \in \Mod{R} \mid \Ext_R^1(M,N) = 0 \}$ and $ {}^{\perp_0}N := \{ M \in \Mod{R} \mid \Hom_R(M,N) = 0 \}$.
\end{notation*}

%%%%%%%%%%%%%%%%%%%%%%%%%%%%%%%%%%%%%%%%%%%%%%%%%%%%

\section{Cosilting modules over artinian rings.}\label{Sec: art cotilt}

\subsection{Background on purity in module categories.}\label{Sec: purity}

In this section we briefly define some key concepts from the theory of purity in module categories.  Consider the category $\FunR$ of additive functors from the category $\mod{R^{\mathrm{op}}}$ of finitely presented right $R$-modules to the category $\Ab$ of abelian groups.  There exists a fully faithful functor $t \colon \Mod{R} \rightarrow \FunR$ where, for each left $R$-modules $M$, we define the functor $t(M) \colon \mod{R^{\mathrm{op}}} \to \Ab$ which takes $N \mapsto N \otimes_R M$ and $f \mapsto f\otimes_R 1_M$

\begin{definition}
An exact sequence $0 \rightarrow X \overset{f}{\rightarrow} Y \overset{g}{\rightarrow} Z \rightarrow 0$ in $\Mod{R}$ is called \textbf{pure-exact} if its image $0 \rightarrow t(X) \rightarrow t(Y) \rightarrow t(Z) \rightarrow 0$ is an exact sequence in $\FunR$.  In this case, the morphism $f$ is called a \textbf{pure monomorphism} and $g$ is called a \textbf{pure epimorphism}.  A module $N$ is called \textbf{pure-injective} if every pure monomorphism $N \rightarrow Y$ splits.  Let $\Pinj{R}$ denote the full subcategory of pure-injective modules.
\end{definition}

It turns out that the functor $t \colon \Mod{R} \rightarrow \FunR$ restricts to an equivalence of categories $t \colon \Pinj{R} \overset{\sim}{\rightarrow} \Inj{\FunR}$ where $\Inj{\FunR}$ is the full subcategory of injective objects in $\FunR$ (see, for example, \cite[Thm.~B.16]{JL}).

\begin{definition}
We say that a module $N$ in $\Pinj{R}$ has \textbf{no superdecomposable part} if every nonzero direct summand of $N$ has a nonzero indecomposable direct summand.
\end{definition}

If $N$ has no superdecomposable part, then there exists a set $\mathcal{N}$ of indecomposable pure-injective modules such that $N \cong \PI\left(\bigoplus_{M \in \mathcal{N}} M\right)$ where $\PI(-)$ denotes the \textbf{pure-injective envelope} of $\bigoplus_{M \in \mathcal{N}} M$ in $\Mod{R}$ i.e.~there exists a pure-monomorphism $\bigoplus_{M \in \mathcal{N}} M \to N$ such that $t\left(\bigoplus_{M \in \mathcal{N}} M\right) \to t(N)$ is an injective envelope of $t\left(\bigoplus_{M \in \mathcal{N}} M\right)$ in $\FunR$ (see \cite{Fisher}, cf.~\cite[Thm.~4.A14]{PrestBlue} for more details).

\begin{remark}\label{Rem: ess unique}
If \[\PI\left(\bigoplus_{M \in \mathcal{N}} M\right) \cong \PI\left(\bigoplus_{L \in \mathcal{L}} L\right)\] for sets $\mathcal{L}$ and $\mathcal{N}$ of indecomposable pure-injective modules, then there is a bijection $\sigma \colon \mathcal{N} \rightarrow \mathcal{L}$ such that $M \cong \sigma(M)$ for every $M$ in $\mathcal{N}$ (again, see \cite[Thm.~4.A14]{PrestBlue}).  It follows that (up to reordering and isomorphism) $\mathcal{N}$ is the set of indecomposable direct summands of $N$. For a pure-injective module $N$ with no superdecomposable part, we will use the notation $\mathcal{N}_N$ to denote this set.  Note that the elements of $\mathcal{N}_N$ are necessarily pure-injective.
\end{remark}

\begin{definition} An additive functor from $\Mod{R}$ to $\Ab$ is called a \textbf{coherent functor} if there exists a morphism $f \colon M \to N$ in $\mod{R}$ such that \[ \Hom_R(N,L) \overset{\Hom_R(f,L)}{\longrightarrow} \Hom_R(M,L) \to F(L) \to 0\] is an exact sequence of abelian groups for every $L \in \Mod{R}$. \end{definition}

\begin{definition}
A full subcategory $\mathcal{D}$ of $\Mod{R}$ is \textbf{definable} if there exists a set $\{ F_i \}_{i\in I}$ of coherent functors such that \[\mathcal{D} = \{ M \in \Mod{R} \mid F_i(M) =0 \text{ for all } i \in I\}.\]
\end{definition}

In \cite{CB}, Crawley-Boevey characterises definable subcategories of $\Mod{R}$ in terms of closure conditions.  He shows that a full subcategory $\mathcal{D}$ of $\Mod{R}$ is definable if and only if $\mathcal{D}$ is closed under direct limits, direct products and pure submodules i.e. if $X \rightarrow Y$ is a pure monomorphism with $Y\in\mathcal{D}$, then $X\in\mathcal{D}$.

\subsection{Cotilting modules over artinian rings.}\label{Sec: art cotilt}
In this section we characterise cotilting modules over artinian rings in terms of their indecomposable direct summands.  First we recall the definition of a cotilting module in $\Mod{R}$.  Note that we are only considering considering cotilting modules of injective dimension $\leq 1$ (so-called $1$-cotilting modules).

\begin{definition}\label{Def: cotilting}
A module $C$ in $\Mod{R}$ is called a \textbf{cotilting module} if $ \Cogen{C} = {}^{\perp_1} C$.  If $C_1$ and $C_2$ are cotilting modules, then we say that $C_1$ is \textbf{equivalent} to $C_2$ if $\Prod{C_1} = \Prod{C_2}$.  
\end{definition}

\noindent We will make use of the following properties of cotilting modules.

\begin{theorem}[{\cite[Thm.~2.8, Prop.~3.2]{Bazz}}]\label{Thm: cotilting are pure-inj}
Every cotilting module over $R$ is a pure-injective module and the full subcategory $\Cogen{C} = {}^{\perp_1} C$ is definable.
\end{theorem}

Next we prove a structure theorem for cotilting modules over an artinian ring.  We require the following definition that can be found in \cite[Sec.~3.4]{TrlifajHandbook} (referred to as a 1-rigid system). For tame hereditary artin algebras, similar ideas can be found in \cite{BuanKrause}.  

\begin{definition}\label{def: max rig}
Let $\mathcal{N}$ be a non-empty set of indecomposable pure-injective modules.  Then $\mathcal{N}$ is called a \textbf{rigid system} if the following conditions are satisfied. 
\begin{enumerate}
\item $\id{R}{M}\leq 1$ for all $M\in \mathcal{N}$.
\item $\Ext^1_R(M, N) = 0$ for all $M, N \in \mathcal{N}$.
\end{enumerate}
We refer to a rigid system $\mathcal{N}$ as \textbf{maximal} if the following property is satisfied:  if there exists a rigid system $\mathcal{L}$ such that $\Prod{\mathcal{N}} \subseteq \Prod{\mathcal{L}}$, then $\Prod{\mathcal{N}} = \Prod{\mathcal{L}}$.  Two maximal rigid systems $\mathcal{N}$ and $\mathcal{L}$ are \textbf{equivalent} if $\Prod{\mathcal{N}} = \Prod{\mathcal{L}}$.
\end{definition}

We will also need the following two lemmas.  The first is implicit in the proof of \cite[Thm.~3.7]{TrlifajHandbook} and the second can be found in \cite[Cor.~2.3]{BuanKrause}.  We include the proofs for completeness.  See Remark \ref{Rem: ess unique} for the definition of $\mathcal{N}_N$.

\begin{lemma}\label{Lem: id definable}
Let $R$ be a left noetherian ring.  For a pure-injective module $N$ with no superdecomposable part, the following statements are equivalent. \begin{enumerate}
\item $\id{R}{N} \leq 1$.
\item $\id{R}{M} \leq 1$ for all $M \in \mathcal{N}_N$.
\end{enumerate}
\end{lemma}
\begin{proof}
Since $R$ is left noetherian, we may apply Baer's criterion and conclude that the class of modules of injective dimension $\leq 1$ is a definable subcategory since the functor $\Ext^1_R(R/I,-)$ is a coherent functor for every ideal $I$.  Then, since definable subcategories are closed under direct summands, direct sums and pure-injective envelopes, the statement follows from the structure of $N$ described in Section \ref{Sec: purity}. 
\end{proof}

\begin{lemma}\label{Lem: self-orth}
Let $R$ be a left artinian ring.  For a pure-injective module $N$ with no superdecomposable part and $\id{R}{N} \leq 1$, the following statements are equivalent.
\begin{enumerate}
\item $\Ext^1_R(N, N) = 0$.
\item $\Ext^1_R(N^I, N) =0$ for all sets $I$.
\item $\Ext^1_R(M', M) = 0$ for all $M', M \in \mathcal{N}_N$.
\end{enumerate}
\end{lemma}
\begin{proof}
The argument for \textbf{(2)$\Rightarrow$(3)} is straight forward.  To show \textbf{(1)$\Rightarrow$(2)} we apply \cite[Cor.~1.10]{BuanKrause}, which tells us that ${}^{\perp_1} N$ is closed under products.  

Finally we prove \textbf{(3)$\Rightarrow$(1)}.  Since $N$ has no superdecomposable part, we may write $N \cong \PI\left(\bigoplus_{M \in \mathcal{N}_N}M\right)$.  Moreover, $N$ is a direct summand of $K := \prod_{M\in \mathcal{N}_N}M$ and it therefore suffices to show that $\Ext_R^1(K, K) = 0$.  By \cite[Cor.~1.10]{BuanKrause} and Lemma \ref{Lem: id definable}, we have that ${}^{\perp_1} M$ is closed under products for each $M\in \mathcal{N}_N$.  Applying our assumption, we therefore have that $\Ext_R^1(K, M) =0$.  But then we have that $\Ext_R^1(K, K) \cong \prod_{M\in \mathcal{N}_N}\Ext_R^1(K, M) = 0$.
\end{proof}

\begin{remark}\label{Rem: Prod = Prod}
Let $\mathcal{N}$ be a set of indecomposable pure-injective modules and consider the modules $N := \PI\left(\bigoplus_{M\in \mathcal{N}} M \right)$ and $L := \prod_{M\in \mathcal{N}}M$.  Then $N$ is a direct summand of $L$ and so we have the following containments $\Prod {L} = \Prod{\mathcal{N}} \subseteq \Prod{ N} \subseteq  \Prod{ L}$.  That is, we have $\Prod{ L} = \Prod{\mathcal{N}} = \Prod{ N}$.
\end{remark}

In Theorem \ref{Thm: cotilt vs rigid} below, we prove that cotilting modulevs over artinian rings are determined by their indecomposable summands.  A slightly weaker version of this theorem can be found in \cite[Thm.~3.7]{TrlifajHandbook}.  Trlifaj produces a similar correspondence between cotilting \emph{classes} and (not necessarily maximal) rigid systems.  Our stronger version requires the following lemma, which follows immediately from \cite[Thm.~3.3]{MehdiPrest}.  A cotilting module is of \textbf{cofinite type} if  there is a set $\mathcal{S}$ of finitely generated right $R$-modules such that \[\Cogen{C} = \{ X\in \Mod{R} \mid \Tor_1^R(S, X) = 0 \text{ for all } S\in \mathcal{S}\}.\]

\begin{lemma}\label{Lem: no sup dec}
Let $R$ be a ring and let $C$ be a cotilting module of cofinite type.  Then there is a direct summand $D$ of $C$ that has no superdecomposable part and is cotilting.  In particular, if $R$ is left noetherian, then every cotilting module is equivalent to one with no superdecomposable part.
\end{lemma} 
\begin{proof}
Let $C$ be a cotilting module of cofinite type.  By \cite[Thm.~15.18]{GT}, there exists a tilting right $R$-module such that $C \cong \Hom_\mathbb{Z}(T, \mathbb{Q}/\mathbb{Z})$.  It follows from \cite[Thm.~3.3]{MehdiPrest} that there exists a set of $\mathcal{N}$ of indecomposable direct summands of $C$ such that $\Prod{C} = \Prod{\mathcal{N}}$.  Then, by Remark \ref{Rem: Prod = Prod}, we have that $\PI\left(\bigoplus_{N\in \mathcal{N}} N \right)$ is a cotilting module equivalent to $C$.  The last statement follows from \cite[Thm.~2.10]{AHPST}.
\end{proof}

\begin{theorem}\label{Thm: cotilt vs rigid}
Let $R$ be a left artinian ring. The assignments \[ C \mapsto \mathcal{N}_{C} \: \: \: \text{ and } \: \: \: \mathcal{N} \mapsto \prod_{M\in \mathcal{N}} M \] are mutually inverse bijections between the set of cotilting modules $C$ in $\Mod{R}$ (up to equivalence) and the set of maximal rigid systems $\mathcal{N}$ in $\Mod{R}$ (up to equivalence). 
\end{theorem}
\begin{proof} Let $C$ be a cotilting module.  By Lemma \ref{Lem: no sup dec}, we may assume that $C$ has no superdecomposable part.  By Lemma \ref{Lem: self-orth} and Lemma \ref{Lem: id definable} we have that $\mathcal{N}_C$ is a rigid system so it remains to show that $\mathcal{N}_C$ is maximal.  Suppose $\mathcal{N}$ is a rigid system with $\Prod{\mathcal{N}_C} \subseteq \Prod{\mathcal{N}}$.  Then let $N := \PI\left(\bigoplus_{M\in \mathcal{N}}M\right)$ and note that, since $\bigoplus_{M\in \mathcal{N}_C}M$ is a direct summand of $\bigoplus_{M\in \mathcal{N}}M$, it follows that $C$ is a direct summand of $N$.  We may then argue as in the proof of (1) $\Rightarrow$ (2) in \cite[Prop.~3.1]{BuanKrause} to obtain that $\Prod{ C} = \Prod{ N}$ and hence that $\Prod{\mathcal{N}_C} = \Prod{\mathcal{N}}$ by Remark \ref{Rem: Prod = Prod}.

Let $\mathcal{N}$ be a maximal rigid system and consider $C := \PI\left(\bigoplus_{M\in \mathcal{N}} M \right)$.  Then $\mathcal{N} = \mathcal{N}_C$ by Remark \ref{Rem: ess unique} and so we have that $\Ext^1_R(C, C) = 0$ and $\id{R}{C}\leq1$ by Lemma \ref{Lem: self-orth} and Lemma \ref{Lem: id definable}.  By \cite[Cor.~1.12]{BuanKrause}, there exists a pure-injective module $N$ such that $C \oplus N$ is cotilting.  Note that $\mathcal{N}_{C\oplus N} = \mathcal{N}_C\cup \mathcal{N}_N = \mathcal{N}\cup \mathcal{N}_N$ and so, by applying Lemma \ref{Lem: no sup dec}, we can ensure that $N$ has no superdecomposable part.  As $\mathcal{N}\cup \mathcal{N}_N$ is a rigid system containing $\mathcal{N}$, we have $\Prod{\prod_{M\in \mathcal{N}} M} = \Prod{C} = \Prod{\mathcal{N}} = \Prod{\mathcal{N}\cup \mathcal{N}_N} = \Prod{N\oplus C}$ by the maximality of $\mathcal{N}$ and Remark \ref{Rem: Prod = Prod}.  This implies that \[\Cogen{\prod_{M\in \mathcal{N}}M} = \Cogen{N\oplus C} = {}^{\perp_1}(N\oplus C) = {}^{\perp_1}\left( \prod_{M\in \mathcal{N}}M\right)\] and so $\prod_{M\in \mathcal{N}}M$ is a cotilting module.

The fact that the assignments are mutually inverse bijections follows immediately from Remark \ref{Rem: Prod = Prod}.
\end{proof}

\subsection{Cosilting modules over artinian rings.}

Cosilting modules over a ring $R$, introduced by Breaz and Pop \cite{BreazPop}, are a generalisation of cotilting modules over $R$.  In fact, the cotilting modules are exactly the faithful cosilting modules.

\begin{definition}
A module $C$ over $R$ is a \textbf{cosilting module} if there exists an injective copresentation $0 \to C \to E_0 \overset{\sigma}{\to} E_1$ such that $\Cogen{C} = \{ M \in \Mod{R} \mid \Hom_R(M, \sigma) \text{ is surjective}\}$.  Two cosilting modules $C_1$ and $C_2$ are said to be \textbf{equivalent} if $\Prod{C_1} = \Prod{C_2}$.
\end{definition}

In the later sections we will make use of the following useful characterisation of cosilting modules.  For this precise formulation, see \cite[Thm.~3.6]{abundance}.

\begin{proposition}[\cite{ZW}]\label{Prop: Cosilt is unfaithful cotilt}
Let $R$ be a ring.  An $R$-module $C$ is a cosilting module if and only if $C$ satisfies the following two conditions:\begin{enumerate} \item $C$ is a cotilting module over $R / \ann{C}$ where $\ann{C} := \{ r \in R \mid r.C = 0\}$.
\item $\Cogen{C}$ is a torsion-free class in $\Mod{R}$.
\end{enumerate}
\end{proposition}

\noindent Combining this characterisation with Theorem \ref{Thm: cotilting are pure-inj}, we obtain the following corollary.

\begin{corollary}[\cite{ZW}, \cite{BreazPop}]\label{Cor: def cosilt}
Every cosilting module $C$ over $R$ is a pure-injective module and $\Cogen{C}$ is a definable subcategory.
\end{corollary}

If $R$ is a left artinian ring, then $R/I$ is left artinian for any ideal $I$ in $R$.  Thus we may combine the results of Section \ref{Sec: art cotilt} with Proposition \ref{Prop: Cosilt is unfaithful cotilt} to obtain the following characterisation of cosilting modules over an artinian ring.

\begin{corollary}\label{cor: max rig cosilt}
Let $R$ be a left artinian ring.  Then a pure-injective module $C$ is a cosilting module if and only if $C$ satisfies the following conditions: \begin{enumerate}
\item There exists a maximal rigid system $\mathcal{N}$ in $\Mod{R/\ann{C}}$ such that $\Prod{C} = \Prod{\mathcal{N}}$.
\item $\Cogen{C}$ is a torsion-free class in $\Mod{R}$.
\end{enumerate}
\end{corollary}

\section{Cosilting modules over cluster-tilted algebras.}\label{Sec: cluster-tilted}

Throughout this section we fix an algebraically closed field $\k$.  This assumption is needed in order to appeal to the classification of indecomposable pure-injective modules over a domestic string algebra used in Theorem \ref{Thm: ind pinj}.
%%%%%%%
%
\subsection{Cluster-tilted algebras of type $\tilde{A}$.}\label{Sec: type A}

Cluster-tilted algebras were introduced by Buan, Marsh and Reiten (\cite{BMR}) as endomorphism algebras of cluster-tilting objects in 
cluster categories: If $\mathcal C_Q$ is the cluster category of an acyclic quiver $Q$ and $T$ a cluster-tilting object in $\mathcal C_Q$, 
the endomorphism ring $\End_{\mathcal C_Q}(T)$ is a {\bf cluster-tilted algebra} of type $Q$. 
We will use the characterisation of cluster-tilted algebras through surface triangulations: 
Assem et al. show in~\cite[Theorem 1.1]{ABCJP} that the cluster-tilted algebras of type $\tilde{A}$ are the 
gentle algebras arising from triangulations 
of an annulus. 

Let $\mathrm{S}$ be an oriented surface and $\mathrm{M}$ a nonempty finite set of points in the boundary of $\mathrm{S}$ such that each boundary component 
has at least one marked point. 
An {\bf arc} in $\mathrm{S}$ is a non self-intersecting curve $\gamma \colon [0,1] \to \mathrm{S}$ with endpoints in $\mathrm{M}$ whose 
interior is disjoint from the boundary of $\mathrm{S}$ and which does not cut out a monogon. 
The arc is a \textbf{boundary arc} if it cuts out a digon, it is \textbf{internal} otherwise. Arcs are considered up to homotopy fixing 
endpoints. A \textbf{triangulation} of a $\mathrm{S}$ (and $\mathrm{M}$) 
is a maximal collection of pairwise non-crossing internal arcs. For details we refer to~\cite{FST}. 
It is known that in the case of an annulus, the number of internal arcs of any triangulation is equal to the 
number of marked points 
on the two boundary components, \cite[Proposition 2.10]{FST}. 
Let $\Gamma$ be a triangulation of an annulus. Then its quiver $Q=Q(\Gamma)$ has as vertices the set $Q_0 := \{1,2,\dots,m\}$ indexed by the set $\{\gamma_1, \dots,\gamma_m\}$ of arcs in $\Gamma$. The set of arrows $Q_1$ is defined to consist of arrows defined as follows: for any two internal arcs $\gamma_i,\gamma_j$ in $\Gamma$ bounding a common triangle we draw an arrow $a \colon i\to j$ if $j$ is clockwise from $i$ (using the orientation of $\mathrm{S}$).  As usual, the quiver $Q$ comes with assignments $s, t \colon Q_1 \to Q_0$ such that $s(a) = i$ and $t(a) = j$.

A triangle $\triangle$ in $\Gamma$ is \textbf{internal} if all its edges are internal arcs. Every internal 
triangle gives rise to an oriented cycle $a_{\triangle}b_{\triangle}c_{\triangle}$ in $Q$ 
(up to cyclic permutation). Then the algebra $A(\Gamma)$ is defined to be the quotient 
$\k Q/I(\Gamma)$ where $I(\Gamma)$ is the ideal generated by the paths 
$a_{\triangle}b_{\triangle},b_{\triangle}c_{\triangle},c_{\triangle}a_{\triangle}$ for 
all triangles $\triangle$ in $\Gamma$. This is the (noncompleted) Jacobian algebra of $Q$ 
with the potential arising from all the cycles of $\Gamma$ \cite{DWZ,Keller}. 

Since we are working with cluster-tilted algebras of type $\tilde{A}$, we will from now on we will 
fix the surface $(\mathrm{S},\mathrm{M})$ to be an annulus with $p$ respectively $q$ marked 
points on the two boundary components, $p,q>0$.  We will also fix a triangulation $\Gamma$ 
and the algebra $A(\Gamma)$.  Unless otherwise stated, the notation $Q$ will refer to the quiver $Q(\Gamma)$ i.e.~$A(\Gamma) \cong = \k Q/I(\Gamma)$.

%%%%%%%%
%
\subsubsection{Geometric model.}

As well as allowing us to define $A(\Gamma)$, the surface $(\mathrm{S},\mathrm{M})$ provides us with a geometric model of the module category.  Internal arcs will simply be called arcs. Arcs with both endpoints at the same boundary 
are called {\bf peripheral} arcs, arcs connecting the two boundaries are called {\bf bridging}. It is customary to draw an annulus and its triangulation on its universal cover $U=U_{p,q}$. This is a horizontal strip in the plane, with a projection map to the annulus. We will label marked points on the boundary components by integers, indicating the upper boundary with a prime. We arrange the marked points so that $0'$ is above $0$ in the universal cover. In every copy of the fundamental domain, there are $p$ marked points on the lower boundary and $q$ marked points on the upper boundary. 
See Figure~\ref{fig:annulus} for an example with $p=3$ and $q=2$.

\begin{figure}
\begin{center}
\includegraphics[width=8cm]{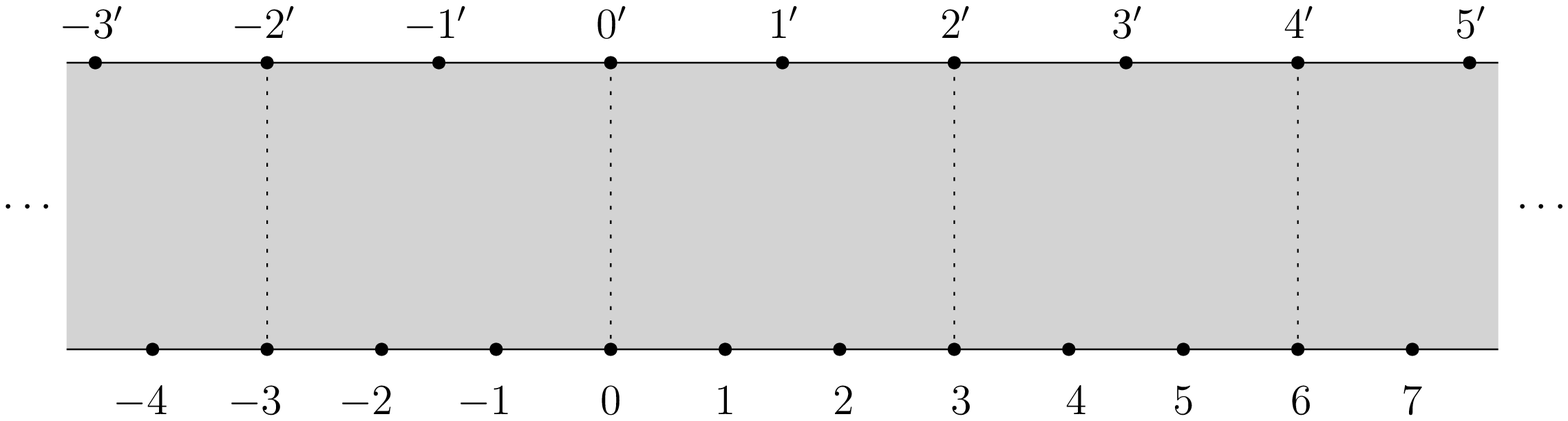} \hskip.5cm $\stackrel{\pi}{\longrightarrow}$ \hskip.5cm 
\includegraphics[width=3cm]{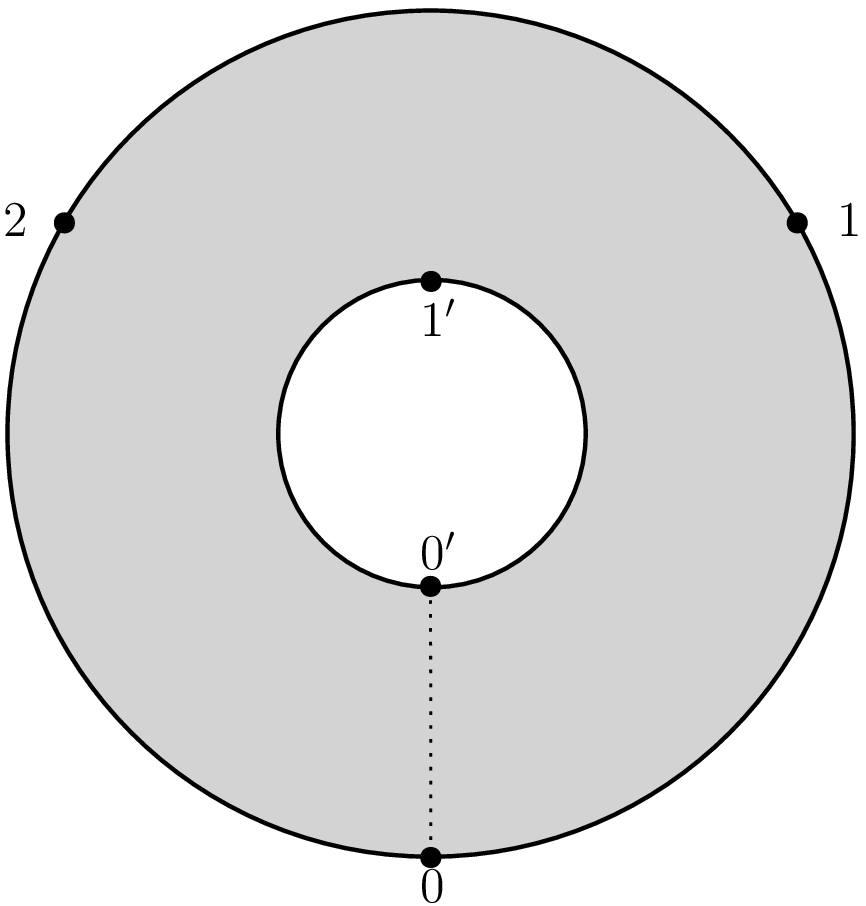} 
\end{center}
\caption{Annulus with universal cover $U_{3,2}$}\label{fig:annulus}
\end{figure}

As an example, we consider a triangulation of an annulus with three marked points 
on the outer boundary and 2 marked points 
on the inner boundary. See Figure~\ref{fig:triangulation}. The associated quiver is on the right. 

\begin{figure}
\begin{center}
\includegraphics[width=7cm]{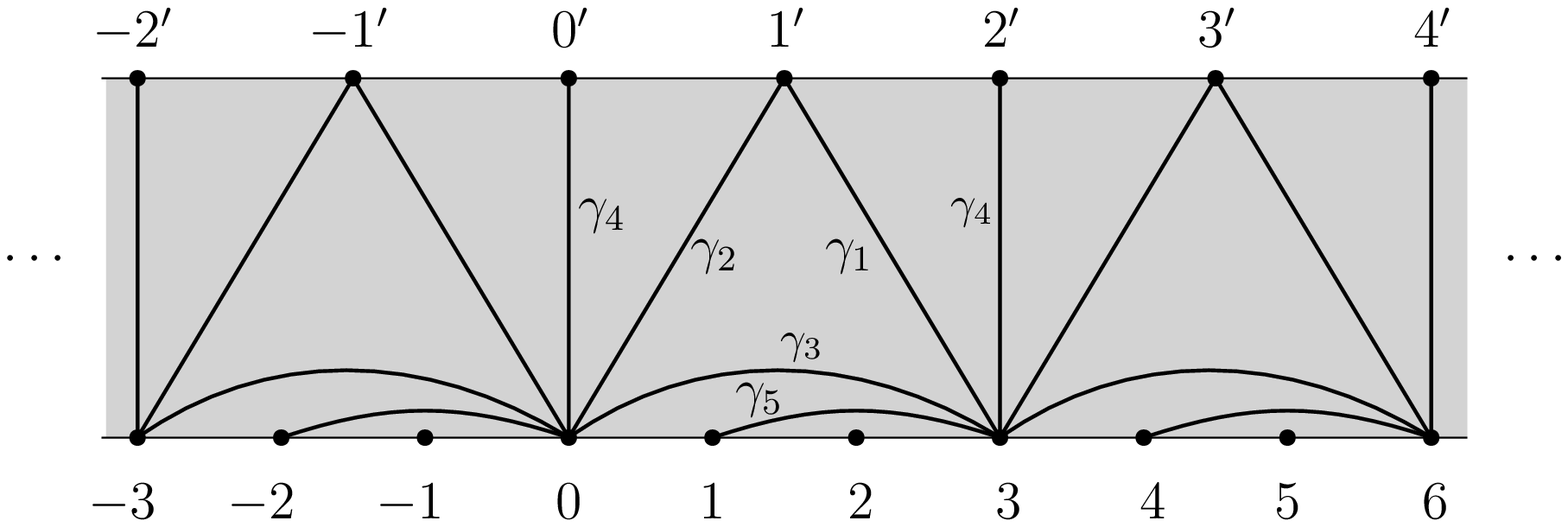} 
\hskip .1cm
\includegraphics[width=2.7cm]{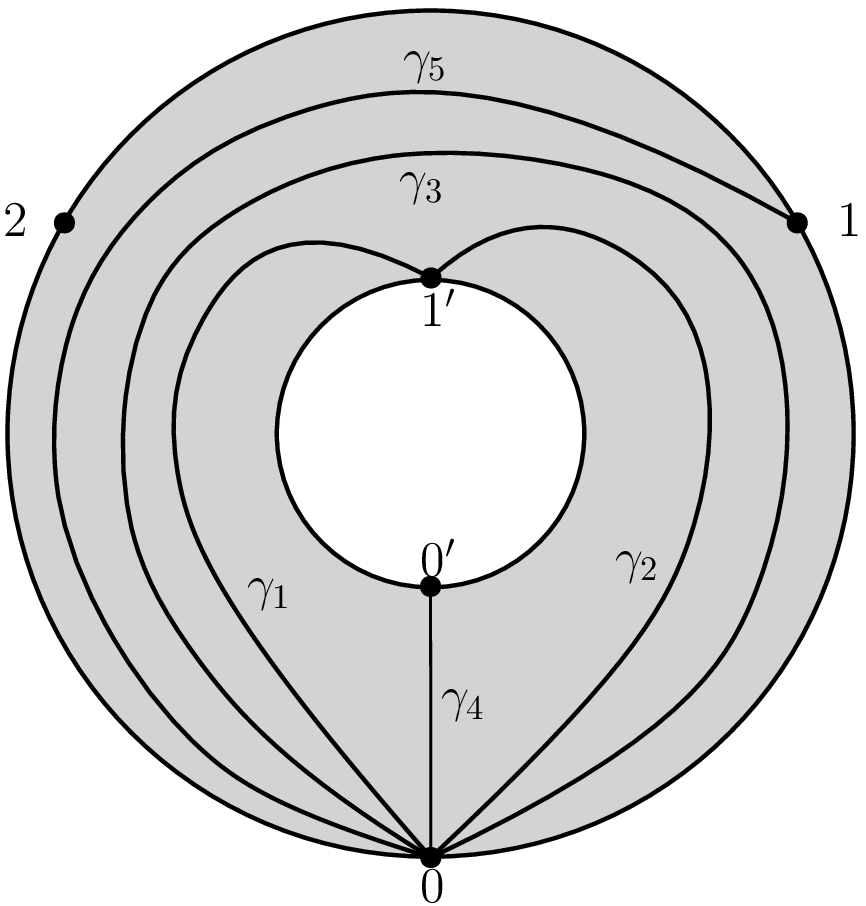} 
\hskip .2cm
\includegraphics[width=3.3cm]{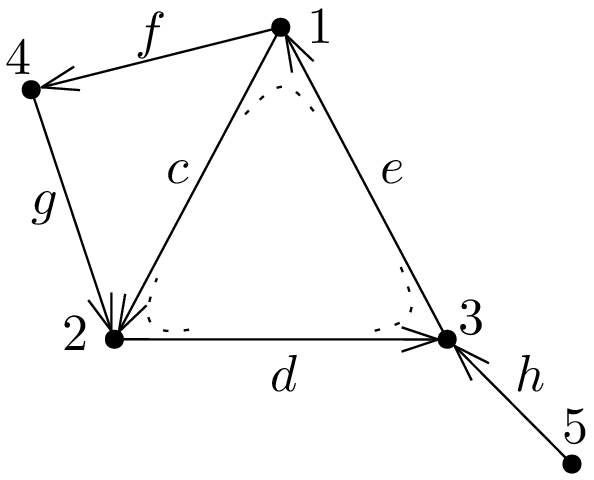} 
\end{center}
\caption{Triangulation in the cover, in the annulus and its quiver}\label{fig:triangulation}
\end{figure}

%%%%%%%%
%
\subsection{Towards maximal rigid systems.}\label{sec: max rig}

In Section \ref{Sec: Asymp_tri} we classify the cosilting $A(\Gamma)$-modules using Corollary \ref{cor: max rig cosilt} i.e.~we identify the pure-injective modules $C$ such that the indecomposable summands of $C$ form a maximal rigid system in $\Mod{A(\Gamma)/\ann{C}}$ and $\Cogen{C}$ is a torsion-free class in $\Mod{A(\Gamma)}$.    

Recall that the elements of a maximal rigid system are pairwise $\Ext$-orthogonal indecomposable pure-injective modules that have injective dimension less than or equal to $1$.  In this section, we set out some preliminary results that will allow us to identify maximal rigid systems in the subsequent section.  In particular, in Section \ref{sec: ind pinj} we consider indecomposable pure-injective modules; in Section \ref{Sec: inj dim} we consider conditions for injective dimension less than or equal to 1; and in Section \ref{Sec: Extensions in A} we consider extensions between indecomposable pure-injective modules.  

\subsubsection{Indecomposable pure-injective modules.}\label{sec: ind pinj}

Here we describe indecomposable finite-dimensional modules for $A(\Gamma)$ 
via arcs in the annulus as introduced in \cite{BM12}, see also~\cite{W,BBM}, where we  
allow marked points on both boundaries as in~\cite{BT} and arbitrary orientation of $Q$. 

We will extend this description of $A(\Gamma)$-modules to include the indecomposable infinite-dimensional pure-injective modules, which were classified in \cite{PP}.

\paragraph{Finite dimensional string modules.}

By a \textbf{finite arc} in $(\mathrm{S},\mathrm{M})$ we mean an arc 
$\alpha$ in $\mathrm{S}$ whose endpoints $\alpha(0), \alpha(1)$ are contained in 
$\mathrm{M}$ and all other points in $\alpha$ are contained in the interior of $\mathrm{S}$.  

Let $\alpha \colon [0,1] \to \mathrm{S}$ be a (finite) arc in $(\mathrm{S},\mathrm{M})$ that is 
not homotopic to any arc in $\Gamma$. 
We assume that $\alpha$ is chosen in its homotopy class so that, if $\alpha$ intersects an 
arc $\gamma$ in $\Gamma$, then $\alpha$ and $\gamma$ intersect transversally and the 
number of intersections is minimal.  
 
We will define a module $M(\alpha)$ for every finite arc $\alpha$ and in order to do this we must 
first define strings. For every arrow $a \in Q_1$, we introduce a formal inverse $a^{-1}$ and we 
extend the assignments $s$ and $t$ so that $s(a^{-1}) = t(a)$ and $t(a^{-1}) = s(a)$.  If $a^{-1}$ is 
the inverse of an arrow $a$, then its formal inverse $(a^{-1})^{-1}$ is defined to be $a$.  
An arrow $a$ will be referred to as a \textbf{direct letter} and its inverse $a^{-1}$ will be referred 
to as an \textbf{inverse letter}.  For each vertex $i \in Q_0$, we introduce a \textbf{trivial letter} 
or \textbf{trivial string}, which we denote $e_i$.  We define ${e_i}^{-1} = e_i$.  We will refer to 
the elements of the set \[ Q_1\cup \{a^{-1}\mid a\in Q_1\} \cup \{e_i\mid i\in Q_0\}\] as 
\textbf{letters}.  For any sequence of letters $p_n\dots p_1p_0$ we define the inverse 
sequence $(p_n\dots p_1p_0)^{-1}$ to be $p_0^{-1}p_1^{-1}\dots p_n^{-1}$.  

Orienting $\alpha$ from $\alpha(0)$ to $\alpha(1)$, suppose $\gamma_{i_0}$ is the first arc of $\Gamma$ that intersects $\alpha$, that $\gamma_{i_1}$ is the second arc of $\Gamma$ that intersects $\alpha$, and so on.  In this way we obtain a sequence $(i_0, i_1,\dots, i_n)$ of vertices in $Q$, which uniquely defines a sequence of letters $u_\alpha$.  This sequence is \[ u_\alpha := \begin{cases} a_n a_{n-1}\dots a_1 & \text{if } n>0\\ e_{i_0} & \text{if } n=0 \end{cases}\]  where $a_j \colon i_{j-1} \to i_j$ for all $1\leq j\leq n$.  We will call such a sequence a \textbf{finite string} of length $n$.  We also extend the definition of $s$ and $t$ to strings 
so that $s(u_\alpha) = i_0$ and $t(u_\alpha) = i_n$.

\begin{example}
As an example, consider algebra $A(\Gamma)$ where $\Gamma$ is the triangulation as in 
Figure~\ref{fig:triangulation}. Let $\alpha$ be the arc in Figure~\ref{fig:ex-arc}. 
Then the string of $\alpha$ is 
$h^{-1}dgfe$.
\begin{figure}
\begin{center}
\includegraphics[width=8cm]{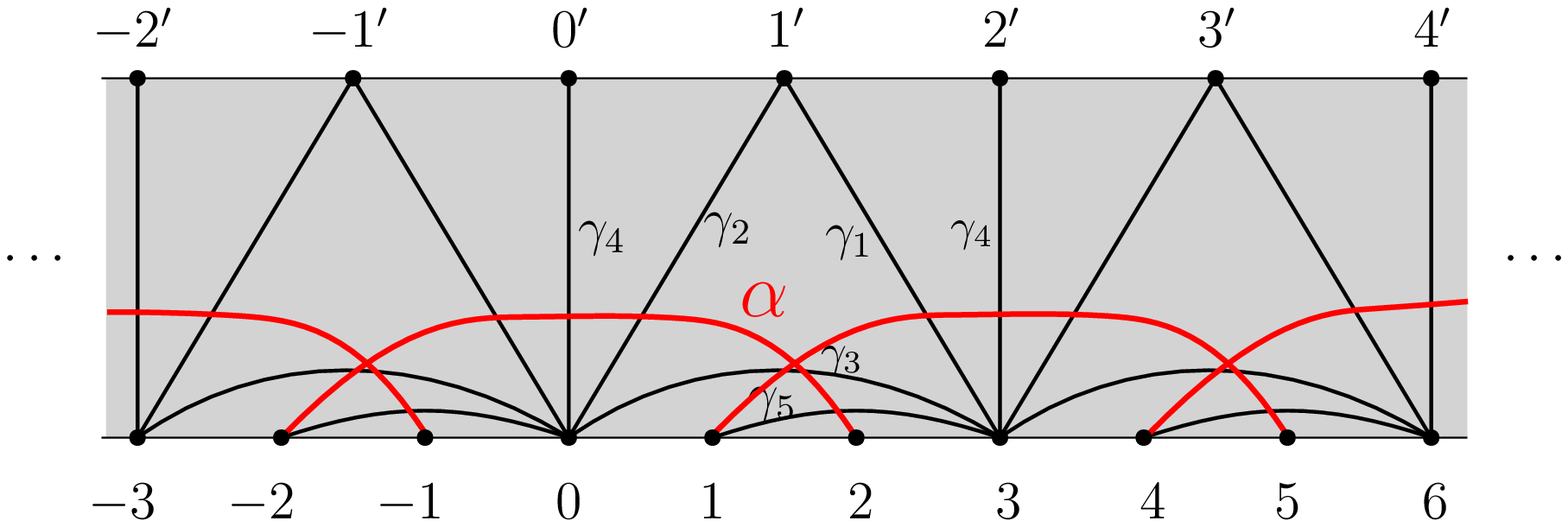} 
\hskip .3cm
\includegraphics[width=4.5cm]{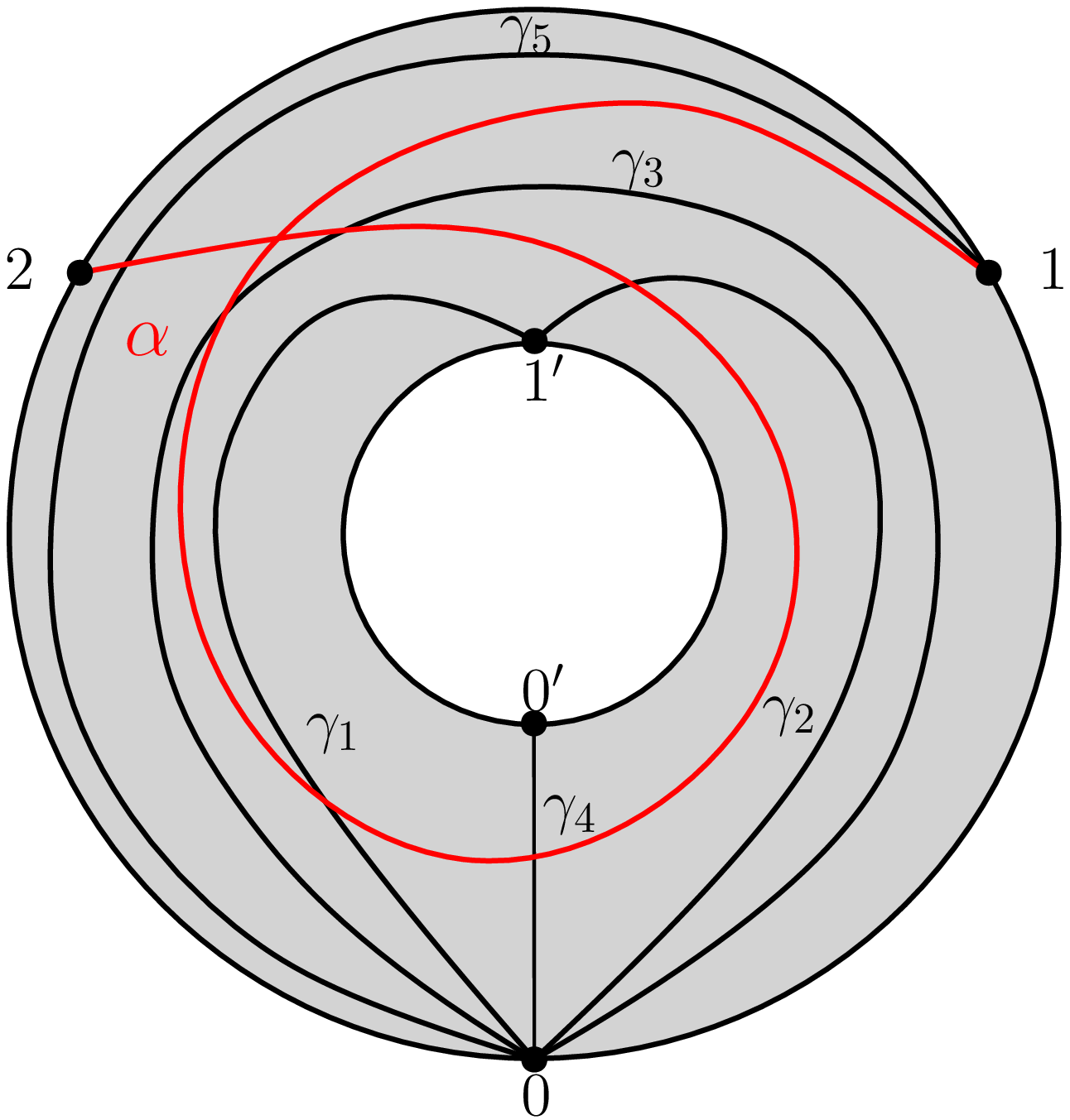} 
\end{center}
\caption{Finite arc in the triangulated cover and in the triangulated annulus}\label{fig:ex-arc}
\end{figure}
\end{example}

\begin{remark}\label{Rem: arcs are strings}
Strings were first introduced in \cite{BR} and used to parametrise the string modules over an arbitrary string algebra.  The above definition of string coincides with this original definition by \cite[Prop.~4.2]{ABCJP}.
\end{remark}

Let $\alpha$ be an arc in $(\mathrm{S},\mathrm{M})$ with associated string $u_\alpha$ as above.  We define the \textbf{string module} $M(\alpha)$ as a representation in the following way.

\noindent 
For each $i\in Q_0$, let $\mathcal{I}_i := \{k \mid  i_k = i \text{ for } 0\leq k\leq n \}$ and define 
\[ 
M(\alpha)_i := \k^{(\mathcal{I}_i)}.
\]

\noindent 
For each arrow $a \colon i \to j$ in $Q$, let 
\[ 
M(\alpha)_a \colon \k^{(\mathcal{I}_i)} \to \k^{(\mathcal{I}_j)}  
\] 
be the matrix with $(k,l)$th entry ($k\in \mathcal{I}_i$, $l \in \mathcal{I}_j$) given by 
$1_\k$ if either $l=k+1$ and $a_l = a$ or if $k=l+1$ and $a_k = a^{-1}$.  
All other entries of $M(\alpha)_a$ are defined to be zero. \\

\begin{remark}
It follows directly from the definitions that $u_\alpha = u_\delta^{-1}$ if and only if 
$M(\alpha) \cong M(\delta)$.  So for our purposes, it will not matter 
which orientation we pick for a finite arc.
\end{remark}

\paragraph{Infinite-dimensional string modules.}  In order to define infinite-dimensional string modules, we will consider `infinite' arcs with only one endpoint.  A curve $\alpha \colon [0, 1) \to \mathrm{S}$ is called an \textbf{asymptotic arc} in $(\mathrm{S},\mathrm{M})$ if the endpoint $\alpha(0)$ is in 
$\mathrm{M}$, the interior of $\alpha$ is contained in the interior of $\mathrm{S}$ and 
$\alpha$ spirals (clockwise or anticlockwise) around the inner boundary component of 
$\mathrm{S}$. 
For an example, see the asymptotic arc $\alpha$ in Figure~\ref{fig:asympt}. 

\begin{figure}
\begin{center}
\includegraphics[width=8cm]{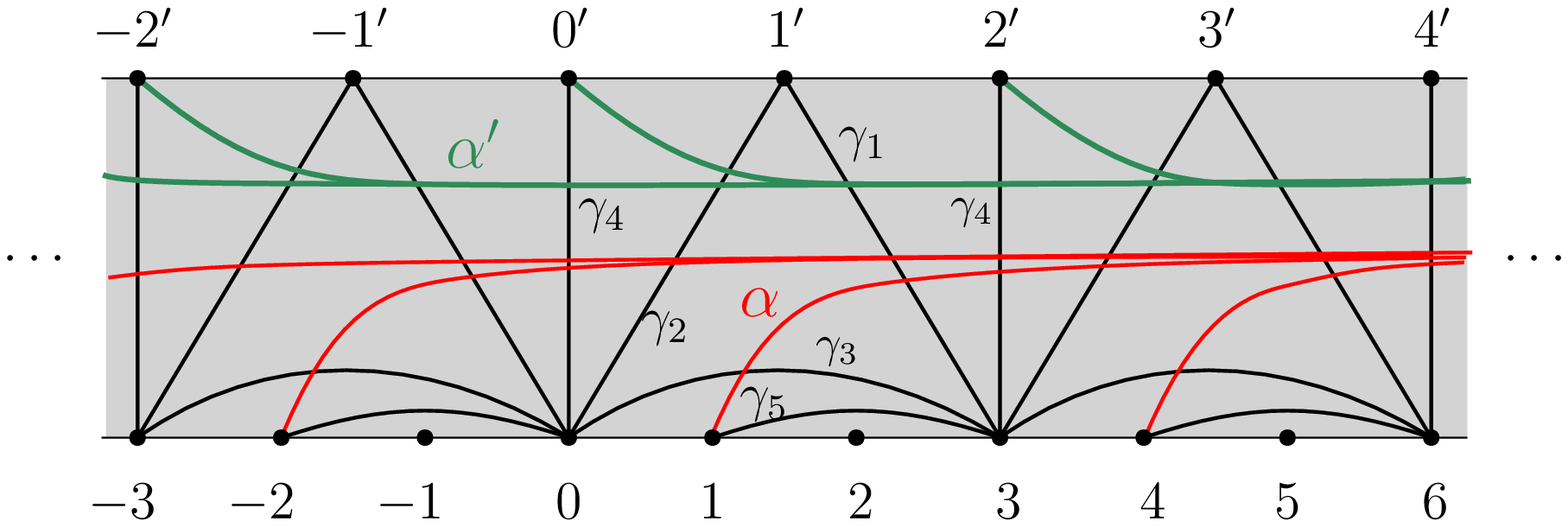} 
\hskip .3cm
\includegraphics[width=4.5cm]{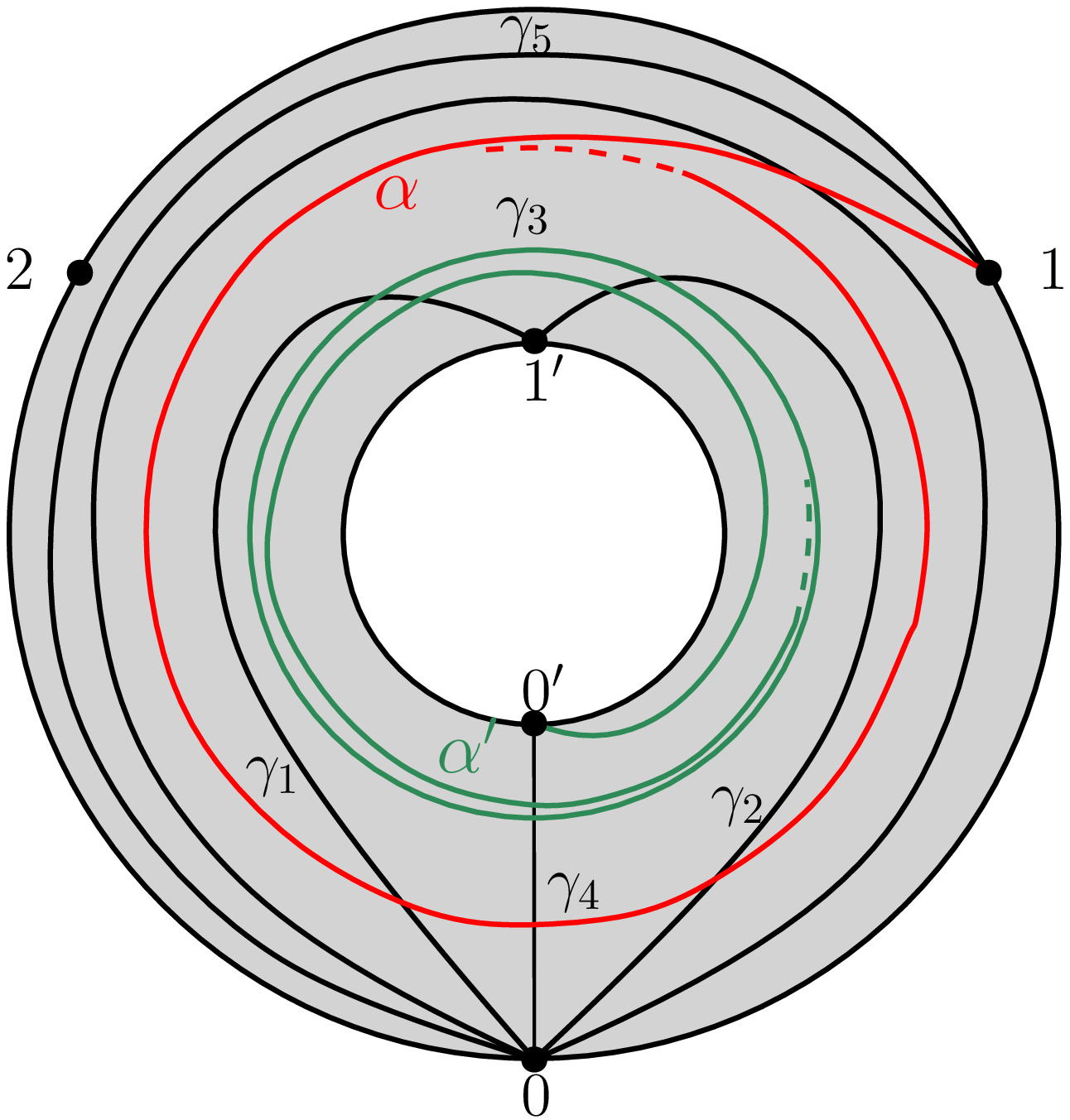} 
\end{center}
\caption{Asymptotic arcs in the triangulated cover and in the triangulated annulus}\label{fig:asympt}
\end{figure}

As with the finite arcs, we consider asymptotic arcs up to homotopy and we assume that $\alpha$ is chosen in its homotopy class so that, if $\alpha$ intersects an arc $\gamma$ in $\Gamma$, then $\alpha$ and $\gamma$ intersect transversally and the number of intersections is minimal.  

From $\alpha$ we define a sequence of vertices in $Q$ as we did in the finite case.  We obtain an infinite sequence $(i_0, i_1, i_2, \dots, i_n, \dots)$, which uniquely determines a sequence of letters $ u_\alpha = \dots a_n \dots a_2 a_1$ such that $a_j \colon i_{j-1} \to i_j$ for all $j >0$.  We call such a sequence an \textbf{$\mathbb{N}$-string}.  We also define $s(u_\alpha)$ to be $i_0$.

\begin{example} 
As an example, we consider the algebra $A(\Gamma)$ for the triangulated annulus 
from Figure~\ref{fig:triangulation}. 
Let $\alpha$ be the arc in Figure~\ref{fig:asympt}, starting at $1$ on the lower (outer) boundary. 
Its string is 
\[
\cdots \ c^{-1} gfc^{-1}gfe 
\] 
Let $\alpha'$ be the asymptotic arc starting at $0'$ on the upper (inner) boundary. 
Its string is 
\[
\cdots \ gfc^{-1} gfc^{-1}
\]
\end{example}

\begin{remark}\label{Rem: some alg compact}
Such $\mathbb{N}$-strings were first considered in the context of string algebras by Ringel 
in \cite{Ringel}.  Ringel defines an $\mathbb{N}$-word to be an infinite sequence of letters 
$\dots a_3a_2a_1$ such that every truncation $a_n\dots a_2a_1$ is a finite string.  In our 
setting $\mathbb{N}$-strings coincide with $\mathbb{N}$-words. 
Indeed, if we consider $u_\alpha = \dots a_3a_2a_1$, then, for every $n \in \mathbb{N}$, 
we can constuct an arc $\alpha_n$ such that $u_{\alpha_n} = a_n \dots a_2 a_1$ as follows. 
Let $\delta$ be the segment of $\alpha$ between $\alpha(0)$ and the point $p$ where 
$\alpha$ intersects $\gamma_{i_n}$.  We may then define $\alpha_n$ to be the concatenation of 
$\delta$ with an arc connecting $p$ with the marked point opposite $\gamma_{i_n}$. 
\end{remark}

Let $\alpha$ be an asymptotic arc in $(\mathrm{S},\mathrm{M})$ and $u_\alpha$ 
the corresponding $\mathbb{N}$-string.  We will define an infinite-dimensional indecomposable pure-injective module $M(\alpha)$.  Let $i_p$ be the first bridging arc crossed by $\alpha$, then the $\mathbb{N}$-string $\dots a_{p+3}a_{p+2}a_{p+1}$ is the \textbf{maximal periodic part of $u_\alpha$}. Let $n \in \mathbb{N}$ be minimal such that $i_{p+n} = i_p$.  If $a_{p+n}$ is direct, then $\alpha$ is said to be \textbf{expanding}.  If $a_{p+n}$ is inverse, then $\alpha$ is said to be \textbf{contracting}.

\begin{example}
We consider the algebra $A(\Gamma)$ for the triangulated annulus 
from Figure~\ref{fig:triangulation}. 
The $\mathbb{N}$-string of the arc $\alpha$ in Figure~\ref{fig:asympt} is contracting. 
The $\mathbb{N}$-string of the arc $\alpha'$ is expanding. 
\end{example}

As before we may define a \textbf{string module} $M(\alpha)$.  For each 
$i\in Q_0$, let $\mathcal{I}_i := \{k \mid  i_k = i \text{ for }  k \geq 0 \}$.  From this data, we  define two modules.  
Firstly let  
\[ 
M^{\Pi}(\alpha)_i := \k^{\mathcal{I}_i}\]  and secondly let \[ M^\oplus(\alpha)_i := \k^{(\mathcal{I}_i)}.
\]

The action of $\Lambda$ is defined in the same way in both cases. For each arrow $a \colon i \to j$ in $Q$, let 
\[ 
M^*(\alpha)_a \colon M^*(\alpha)_i \to M^*(\alpha)_j 
\] 
(with $* = \prod, \bigoplus$) be the morphism determined by the following components. For $k\in \mathcal{I}_i$ and 
$l \in \mathcal{I}_j$, take the component map to be $1_\k$ if either $l=k+1$ and $a_l = a$ or if $k=l+1$ and 
$a_k = a^{-1}$.  Otherwise take the $(k,l)$th component of $M^*(\alpha)_a$ to be zero. 

If $\alpha$ is expanding, then define $M(\alpha) := M^{\Pi}(\alpha)$ and if $\alpha$ is contracting, then  define $M(\alpha) := M^\oplus(\alpha)$.

\paragraph{Band modules.}

Let $\beta$ be the unique simple closed arc lying in the interior of $(\mathrm{S},\mathrm{M})$ 
(considered up to homotopy). 
As before, we choose a representative in the homotopy class such that 
$\beta$ crosses all arcs $\gamma \in \Gamma$ transversally and the number of crossings 
between each $\gamma$ and $\beta$ is minimal.   
We orient $\beta$ in the same way as the boundary: 
in the universal cover, $\beta$ is oriented from left to right, in the annulus 
anticlockwise.

Choose an arc $\gamma_{i_0}$ in $\Gamma$ that is crossed by $\beta$ such that, if 
$\gamma_{i_1}$ is the next arc in $\Gamma$ crossed by $\beta$ in a clockwise direction 
and $\gamma_{i_{-1}}$ is the next arc crossed by $\beta$ in a anticlockwise direction, 
then the letter $a_0 \colon i_{-1} \to i_0$ is direct and the letter $a_1 \colon i_0 \to i_1$ is 
inverse.\footnote{This can always be done, as the bridging arcs in the triangulation 
change orientation at least twice, i.e.\ they form both a figure N and a \reflectbox{N}  
where the 
first and last of the three arcs may coincide.}  
Following $\beta$ in a clockwise direction, let $\gamma_{i_{2}}$ be the next arc 
in $\Gamma$ crossed by $\beta$ after $\gamma_{i_1}$, let $\gamma_{i_{3}}$ be the next 
arc in $\Gamma$ crossed by $\beta$, and so on.  After carrying out a similar procedure in 
the anticlockwise direction, the arc $\beta$ uniquely determines a sequence 
$(\dots i_{-1}, i_0, i_1, \dots)$ of vertices in $Q$.  This sequence then uniquely determines a 
sequence of letters $u_\beta =\dots a_1a_0a_{-1}\dots$ such that $a_j \colon i_{j-1} \to i_j$ 
for all $j \in \mathbb{Z}$.  We will refer to $u_\beta$ as the \textbf{periodic $\mathbb{Z}$-string} 
and, for formal reasons, we will say that the \textbf{maximal periodic part of $u_\beta$} is $u_\beta$.  
Let $s>0$ be minimal such that $i_s=i_0$. Then we call the 
sequence $b := a_s\dots a_1$ of letters the 
\textbf{band} corresponding to $\beta$.\begin{remark}\label{Rem: idem-ideal} 
Consider the band $b := a_s\cdots a_1$.  Let $J$ be the ideal in in $A(\Gamma)$ generated 
by the idempotent elements corresponding to the set $V$ consisting of the peripheral arcs 
in $\Gamma$ and let $B : = A(\Gamma) / J$. The algebra $B$ is the surface algebra given 
by the surface obtained by cutting along all of the arcs $\gamma_j\in V$.  So $B$ is given 
by a triangulation of the annulus consisting of only bridging arcs, i.e.~$B\cong \k \tilde{A}_{s-1}$. 
\end{remark}

\begin{example}
Consider the algebra $A(\Gamma)$ for the triangulated annulus 
from Figure~\ref{fig:triangulation}. 
The sequence for $\beta$ as in Figure~\ref{fig:band} is $(\dots, 4,2,1,4,2,1,4,\dots)$ with 
periodic $\mathbb{Z}$-string $\cdots fc^{-1} gf c^{-1}g \cdots$. We have $i_0=1$, $s=3$ and 
$b=c^{-1}gf$. 

\begin{figure}
\begin{center}
\includegraphics[width=8cm]{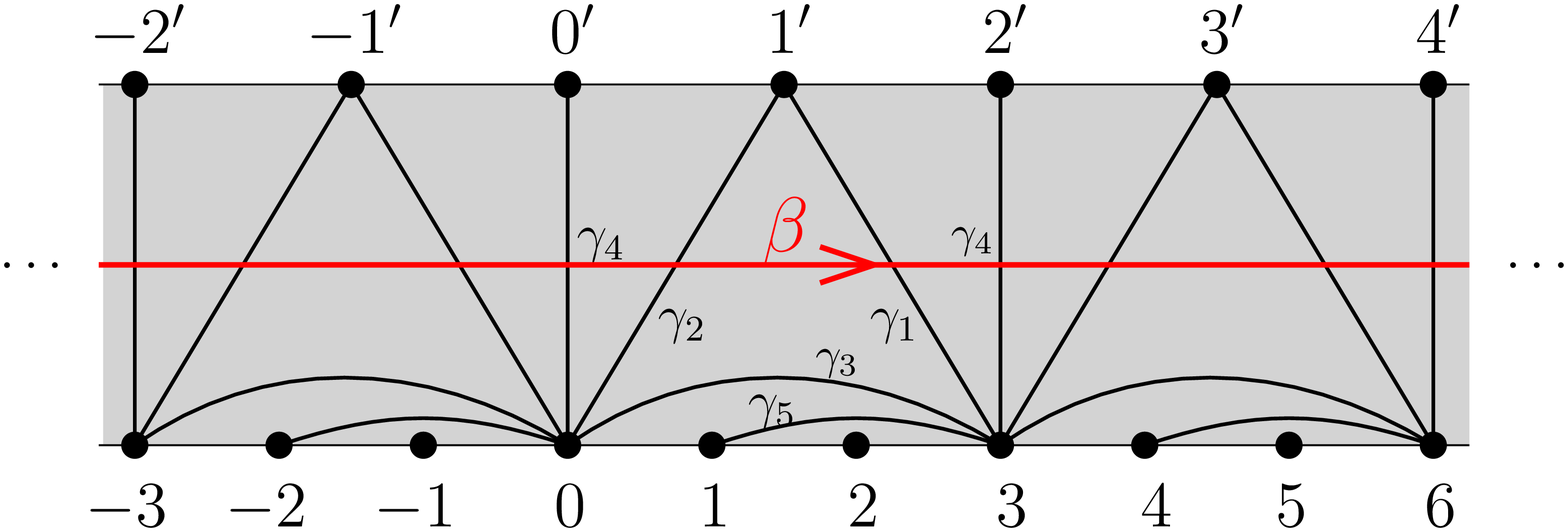} 
\hskip .3cm
\includegraphics[width=4.5cm]{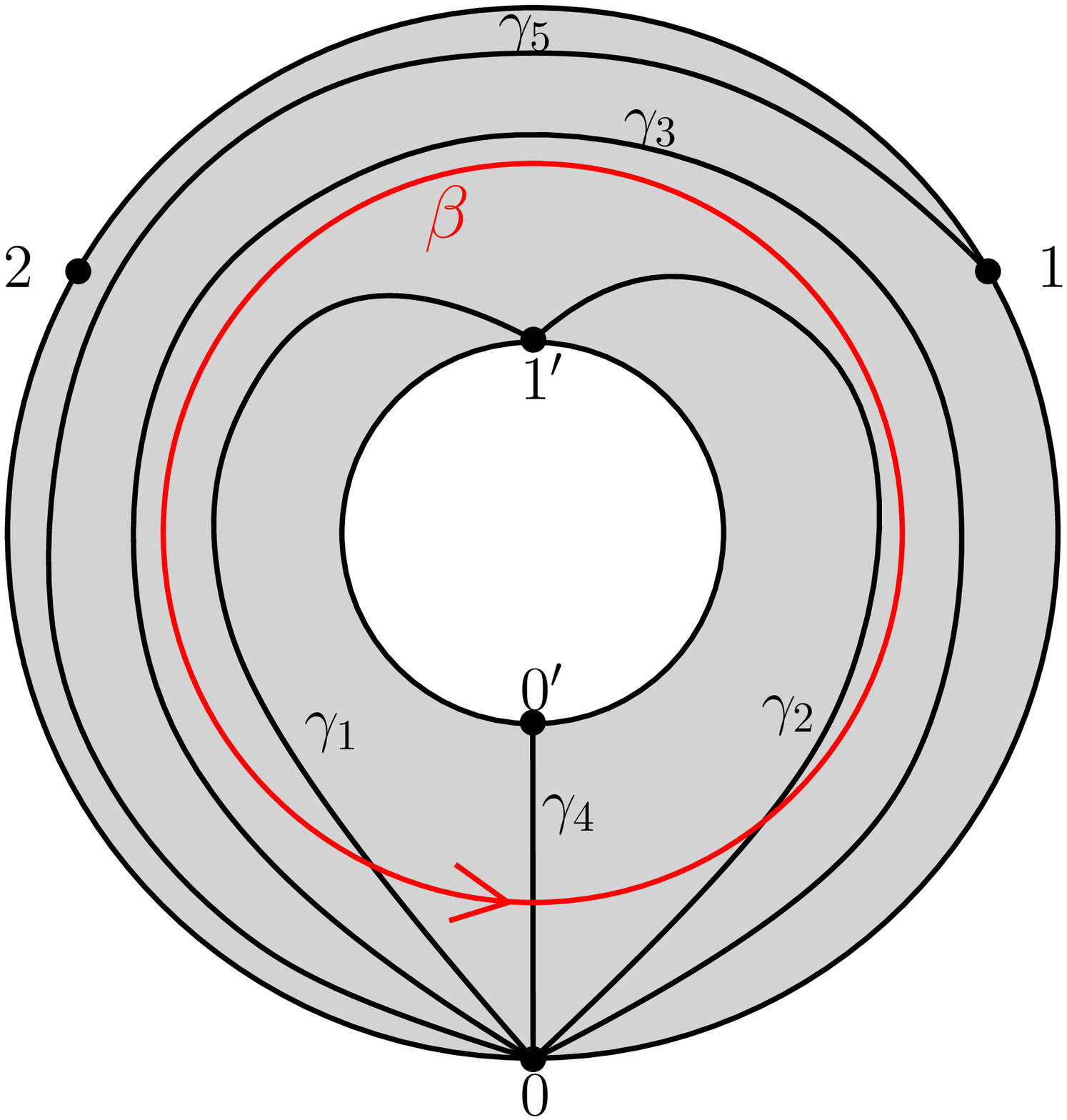} 
\end{center}
\caption{Arc $\beta$ in the cover and in the annulus}\label{fig:band}
\end{figure}
\end{example}

As before, we define a module $M^\oplus(\beta)$ associated to $u_\beta$.  
For each $i\in Q_0$, let $\mathcal{I}_i := \{k \mid  i_k = i \text{ for }  k \in \mathbb{Z} \}$ 
and define $M^\oplus(\beta)_i := \k^{(\mathcal{I}_i)}$.  For each arrow $a \colon i \to j$ 
in $Q$, let $ M^\oplus(\beta)_a \colon M^\oplus(\beta)_i \to M^\oplus(\beta)_j  $ be the 
morphism defined exactly as in the definition of the infinite-dimensional string modules.

It is clear from the construction of $M^\oplus(\beta)$ that the $\k$-basis vectors of 
$M^\oplus(\beta)$ are canonically parametrised with the integers: $\{\dots v_{-1}, v_0, v_1, \dots\}$. 
Let $\Phi \colon M^\oplus(\beta) \to M^\oplus(\beta)$ be such that $v_i \mapsto v_{i+s}$ for all 
$i\in \mathbb{Z}$. It can easily be checked that $\Phi \in \End_A(M^\oplus(\beta))$ and that 
$\Phi^{-1}$ is the endomorphism such that $v_i \mapsto   v_{i-s}$ for all $i\in \mathbb{Z}$. 
This endomorphism equips $M^\oplus(\beta)$ with a $(A(\Gamma), R)$-bimodule structure 
where $R := \k[\Phi, \Phi^{-1}]$.  We
consider the functor 
\[
M^\oplus(\beta) \otimes_{R}(-) \colon \Mod{R} \to \Mod{A(\Gamma)}.
\] 

An $A(\Gamma)$-module of the form $N \cong M^\oplus(\beta) \otimes_{R}L$ is an 
indecomposable pure-injective module if $L$ is an indecomposable 
pure-injective $R$-module (see, for example, \cite{Ringel}).

\begin{remark}\label{Rem: K[T,T^{1}]}
The following is a complete list of the indecomposable pure-injective modules over $R$.

\begin{enumerate} \item Finitely-generated module $N(\lambda, n)$ i.e. 
$\k[\Phi]/{\mathbf{p}_\lambda}^n$ where $\mathbf{p}_\lambda = (\Phi - \lambda)$, 
for each $n\geq1$ and $\lambda \in \k^*$.

\item Pr\"ufer module $N(\lambda, \infty)$ for each $\lambda \in \k^*$ i.e. the direct 
limit $\varinjlim_{n\geq1} N(\lambda, n)$.

\item $\mathbf{p}_\lambda$-adic completion $N(\lambda, -\infty)$ of $\k[\Phi]$ for each 
$\lambda \in \k^*$ i.e. the inverse limit $\varprojlim_{n\geq 1} N(\lambda, n)$.

\item Field of fractions $\k(\Phi)$.
\end{enumerate}

The above list is obtained from the classification of indecomposable pure-injective modules 
over a Dedekind domain in~\cite{Ziegler} by observing that the category of $R$-modules can 
be realised as a definable subcategory of the category of $\k[\Phi]$-modules.  
We also use that, since $\k$ is algebraically closed, all the maximal ideals of $\k[\Phi]$ are of 
the form $\mathbf{p}_\lambda$ for $\lambda\in\k$.
\end{remark}

From the list of Remark~\ref{Rem: K[T,T^{1}]}, we obtain 
the following indecomposable pure-injective $A(\Gamma)$-modules. We will refer to them 
as {\bf band modules}.

\begin{enumerate}
\item Finite-dimensional band module $M(\lambda, n) := M^\oplus(\beta) \otimes_{R}N(\lambda, n)$ 
for each $\lambda \in \k^*$ and $n\in \mathbb{N}$.
\item Pr\"ufer module $M(\lambda, \infty) := M^\oplus(\beta) \otimes_{R}N(\lambda, \infty)$ 
for each $\lambda \in \k^*$.
\item Adic module $M(\lambda, -\infty) := M^\oplus(\beta) \otimes_{R}N(\lambda, -\infty)$ 
for each $\lambda \in \k^*$.
\item Generic module $G := M^\oplus(\beta) \otimes_{R}\k(\Phi)$.
\end{enumerate}

\begin{remark}
The definition of the band modules depends on the choice of arc $\gamma_{i_0}$ in the 
construction.  However, the set of modules defined by different choices of $\gamma_{i_0}$ 
are the same up to isomorphism.  See \cite{Ringel} for details.
\end{remark}

\paragraph{Classification of indecomposable pure-injective $A(\Gamma)$-modules.}
The algebras $A(\Gamma)$ are domestic string algebras, therefore we may apply the following classification result.

\begin{theorem}[\cite{PP}]\label{Thm: ind pinj}
Let $A(\Gamma)$ be the algebra associated to a triangulated $\Gamma$ of the annulus $(\mathrm{S},\mathrm{M})$.  Then the following is a complete list of indecomposable pure-injective $A(\Gamma)$-modules. \begin{enumerate}
\item The string module $M(\alpha)$ for each $\alpha$ (up to homotopy) such that $\alpha$ is a finite or asymptotic arc in $(\mathrm{S},\mathrm{M})$ that is not homotopic to any arc in $\Gamma$.
\item The band module $M(\lambda, n)$ for every $\lambda \in \k^*$ and every $n\in \mathbb{N}\cup\{\infty, -\infty\}$.
\item The generic module $G$.
\end{enumerate}
\end{theorem}

The term \textbf{arc in $(\mathrm{S},\mathrm{M})$} is used to refer to finite arcs, asymptotic arcs or $\beta$.  For a string module $M(\alpha)$, we say that $\alpha$ is the \textbf{associated arc} and $u_\alpha$ is the \textbf{associated string}.  For a band module $M(n, \lambda)$ ($n \in \mathbb{N}\cup\{\infty, -\infty\}, \lambda \in \k^*$) or $G$, we say that $\beta$ is the \textbf{associated arc} and $u_\beta$ is the \textbf{associated string}.

We will take the convention that, for any $\gamma \in \Gamma$, the module $M(\gamma)$ is defined to be the zero module.

%%%%%%%%
%

 \subsubsection{Injective dimension of indecomposable pure-injective modules.}\label{Sec: inj dim}
In this subsection we consider the injective dimension of the indecomposable pure-injective modules introduced in the previous subsection.  

\paragraph{Injective dimension of string modules}
According to Corollary \ref{cor: max rig cosilt}, we must show that certain modules have injective dimension less than or equal to $1$ over a quotient of $A(\Gamma)$.  It turns out that this quotient is still a gentle algebra (see Corollary \ref{cor: ann is gentle}).  In this section we therefore consider the injective dimension of string modules over an arbitrary gentle algebra $A \cong \k Q/I$.  

In \cite{Ringel} Ringel defines finite strings, infinite strings and bands over any string algebra (and hence any gentle algebra).  It is also shown in \cite{Ringel} that in this context infinite- and finite-dimensional string and band modules are always indecomposable pure-injective modules.    

Several authors have computed syzergies and hence projective resolutions of finite-dimensional string and band modules (see, for example, \cite{CPS, WW, BaurSchroll}).  Syzergies of infinite-dimensional string modules have also been considered in \cite{HZS}.  Since we are considering injective dimension rather than projective dimension, we include a full proof of the next lemma.  However everything is completely dual to the papers cited here.

Let $w$ be a finite string. We can write $w = p_0^{-1}q_0p_1^{-1}q_1\dots p_n^{-1}q_n$ where $p_i, q_i$ are paths in $(Q,I)$ for all $0 \leq i \leq n$ and the length of $p_i$ (resp.~$q_j$) is greater than or equal to $1$ when $i >0$ (resp.~$j<n$).  Similarly we can write any $\mathbb{N}$-string as $w = \dots q_1^{-1}p_1q_0^{-1}p_0$.  
Note that this partition into direct and inverse paths is uniquely determined.

\begin{lemma}\label{Lem: id-gentle}
Let $A = \k Q/I$ be a gentle algebra.  Let $w = q_n^{-1}p_n \dots q_1^{-1}p_1q_0^{-1}p_0$ (respectively 
$w = \dots q_1^{-1}p_1q_0^{-1}p_0$) be a finite string (respectively an $\mathbb{N}$-string).

Then $\id{A}{M(w)} \leq 1$ if and only if the following conditions are satisfied. 
\begin{enumerate}
\item 
If there is an arrow $l_0$ such that $p_0l_0$ is a path in $(Q,I)$, then there is no arrow $a$ such that $l_0a \in I$. 
\item 
If $w$ is a finite string and there is an arrow $l_n$ such that $q_nl_n$ is a path in $(Q,I)$, then there is no arrow 
$a$ such that $l_na \in I$.
\end{enumerate}
\end{lemma}
\begin{proof}
We begin by proving the proposition for a finite string ${w}$.  The injective envelope of $M({w})$ is 
$\bigoplus_{i=0}^n M(b_i^{-1}q_i^{-1}p_ia_i)$ where ${a}_i$ is the longest path in $Q$ such that 
${p}_i{a}_i$ is not contained in $I$ and ${b}_i$ is the longest path in $Q$ such that ${q}_i{b}_i$ is not contained in $I$. 
The first cosyzygy of $M(w)$ is isomorphic to $\bigoplus_{i=1}^n M({b}_{i-1}^{-1}a_i) \oplus M({a}_0') \oplus M({b}_n')$ where 
${a}_0 = l_0{a}_0'$ and ${b}_n = l_n{b}_n'$ for arrows $l_0, l_n$ if they exist and zero otherwise. 

The module $M({a}_i^{-1}{b}_{i-1})$ is injective for all $1\leq i\leq n$.  So $\id{A}{M} \leq 1$ if and only if 
$M({a}_0')$ and $M({b}_n')$ are injective if and only if conditions $(1)$ and $(2)$ are satisfied.

Now consider an $\mathbb{N}$-string ${w}$.  If ${w}$ is expanding then the injective envelope of 
$M(w)$ is given by $\prod_{i\geq 0} M(b_i^{-1}q_i^{-1}p_ia_i)$ and if ${w}$ is contracting then the injective envelope of $M({w})$ is given by $\bigoplus_{i\geq 0} M(b_i^{-1}q_i^{-1}p_ia_i)$.  In both cases $a_i$ and $b_i$ are defined as in the finite case.  

Then, as in the finite case, the first cosyzygy is the direct sum of an injective module and a serial module 
$M({a}_0')$ where ${a}_0'$ is either $0$ or ${a}_0 = l_0{a}_0'$ for an arrow $l_0$. Finally we observe that 
$M({a}_0')$ is injective if and only if condition $(1)$ is satisfied as required.
\end{proof}

\paragraph{Injective dimension of band modules}
It suffices for us to consider the injective dimension of band modules over $A(\Gamma)$.  Consider the quotient algebra $B := A(\Gamma) / J$ from Remark~\ref{Rem: idem-ideal} and let $f \colon A(\Gamma) \rightarrow B$ be the corresponding 
ring epimorphism.  

\begin{remark}\label{rem:ideal-image} The image of the restriction functor $f^* \colon \Mod{B} \rightarrow \Mod{A}$ contains all the band modules and is closed under extensions because $J$ is an idempotent ideal (see, for example, \cite[Cor.~2.2]{Jans}).  Moreover, it is a bireflective and therefore definable subcategory. \end{remark}

\begin{lemma}\label{Lem: band-Z-id}
If $N$ is a band module over $A(\Gamma)$, then $\id{A(\Gamma)}{N} \leq 1$ and $\pd{A(\Gamma)}{N} \leq 1$.
\end{lemma}
\begin{proof}
By \cite[Cor.~2.13]{CPS} finite-dimensional band modules have projective dimension equal to $1$.  It follows from \cite{GR} that the algebra is $1$-Iwanga-Gorenstein, and so the finite-dimensional band modules also have injective dimension equal to $1$ (see \cite[Thm.~5]{Iwanaga}).

Recall that, by Baer's criterion, the full subcategory with objects in the class $\mathcal{F}$ of 
modules of injective dimension less than or equal to $1$ is definable.  Similarly, the full 
subcategory with objects in the class $\mathcal{P}$ of modules of projective dimension 
less than or equal to $1$ is definable (see, for example, \cite[Lem.~5]{KraFin}).  Let $\mathcal{D}$ be the image of the restriction functor $f^* \colon \Mod{B} \rightarrow \Mod{A(\Gamma)}$. Thus $\mathcal{F} \cap \mathcal{D}$ (resp. 
$\mathcal{P} \cap \mathcal{D}$) is a definable subcategory and it contains all the 
finite-dimensional band modules.  By \cite[Thm.~A]{RinZg}, the remaining band modules are in 
also in $\mathcal{F} \cap \mathcal{D}$ (resp. $\mathcal{P} \cap \mathcal{D}$).
\end{proof}

%%%%%%%%%%%%%%%%%%%
\subsubsection{Extensions between indecomposable pure-injective modules.}\label{Sec: Extensions in A}

In this section, specifically in Theorem \ref{Thm: extensions}, we prove that there is a very explicit description of the extensions between the indecomposable pure-injective $A(\Gamma)$-modules in terms of crossings of the corresponding associated arcs.  

\begin{definition}\label{def:cross-in-3-cycle}
Let $N$ and $M$ be string or band modules for $A(\Gamma)$
and let $\alpha$, $\delta$ be their respective associated arcs in $(\mathrm{S},\mathrm{M})$. 
Then we say that $\alpha$ and $\delta$ \textbf{cross in a 3-cycle} if we have 
one of the following local configuration with $\gamma_i,\gamma_j,\gamma_k$ bridging 
arcs in the triangulation. 
\[
\includegraphics[width=7cm]{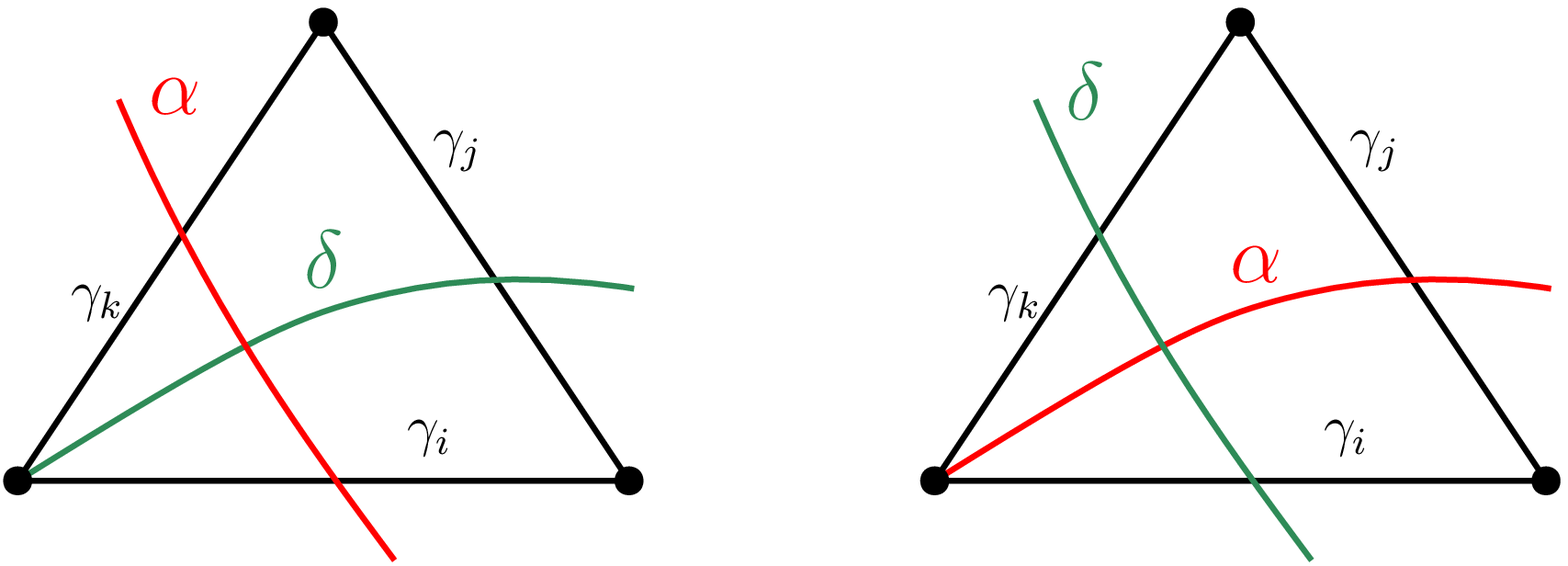}
\]
Similarly, we will say that \textbf{$\alpha$ and $\delta$ do not cross outside of 3-cycles} 
if the only crossings between $\alpha$ and $\delta$ occur as in the above diagrams.
\end{definition}

The left-hand picture in Definition~\ref{def:cross-in-3-cycle} 
may be expressed in terms of the associated strings by saying we have 
that $\alpha$ and $\delta$ cross in a 3-cycle if there exists a 3-cycle 
$\xymatrix@R=4mm@C=6mm{i \ar[r]^d & j \ar[d]^c \\ & k \ar[ul]^a }$ 
(so necessarily $cd, ac, da \in I(\Gamma)$) 
such that $u_\alpha$ contains the letter $a$ or $a^{-1}$ and we have that $s(u_\delta) = j$ 
(assuming without loss of generality that $\delta(0)$ is shown in the diagram). 
An analogous description can be given for the the right-hand picture.

In Theorem \ref{Thm: extensions} we will show that all crossings between arcs that are not in a 3-cycle indicate 
the existence of extensions between the associated modules.

\begin{theorem}\label{Thm: extensions}
Let $M$ and $N$ be string or band modules over $A(\Gamma)$ and let $\alpha$ and $\delta$ be 
their respective associated arcs.  
If $M$ and $N$ are not both band modules, then 
$\Ext^1_{A(\Gamma)}(M,N) = 0 = \Ext^1_{A(\Gamma)}(N,M)$ if and only if $\alpha$ and $\delta$ 
do not cross outside of 3-cycles.
\end{theorem}
\begin{proof}
The case where $M$ and $N$ are both finite-dimensional string modules is covered 
in \cite{CS}. 
We consider the case where $M$ or $N$ is the generic module; all other cases are 
addressed in Appendix \ref{App: Ext} since the arguments require the derived category.  
Throughout the proof we will denote the group $\Ext^1_{A(\Gamma)}(X, Y)$ by $\Ext^1(X, Y)$ 
for any pair $X,Y$ of $A(\Gamma)$-modules.

First suppose that $M$ is the generic module and $N$ is a string module.  Then the associated 
arc of $M$ is $\alpha = \beta$.  By \cite[Prop.~4]{RinZg}, there exists an endomorphism $\phi$ 
of an adic module $M(\lambda, -\infty)$ such that 
$M(\lambda, -\infty) \overset{\phi}{\to} M(\lambda, -\infty) \overset{\phi}{\to}  \dots$ is a 
direct system of monomorphisms and the direct limit $L$ is isomorphic to $M^{(J)}$ for 
some set $J$ (where $M^{(J)}$ denotes the direct sum of copies of $M$ indexed by $J$). 
This induces an inverse system 
\[
\Ext^1(M(\lambda, -\infty), N) \overset{\Ext^1(\phi,N)}{\longleftarrow}\Ext^1(M(\lambda, -\infty), N) \overset{\Ext^1(\phi,N)}{\longleftarrow}\dots
\] 
Denote its inverse limit by $L'$. 
It then follows that $\Ext^1(M,N)=0$ if and only if 
$0= \Ext^1(M,N)^J \cong \Ext^1(M^{(J)},N) \cong \Ext^1(L,N) \cong L'$ 
where the final 
isomorphism follows from \cite[Prop.~I.10.1]{Auslander}.  By Appendix \ref{App: Ext}, we 
have that $\Ext^1(M(\lambda, -\infty), N) =0$ and  if and only if $\beta$ and $\delta$ do 
not cross outside of 3-cycles.  Thus, it suffices to show that $L' = 0$ if and only if 
$\Ext^1(M(\lambda, -\infty), N) =0$.

Clearly, if $\Ext^1(M(\lambda, -\infty), N) =0$, then $L'=0$. 
We will therefore prove the claim by showing that 
$\dim{\Ext^1(M(\lambda, -\infty), N)} \leq \dim{L'}$. The inverse system consists of 
epimorphisms because $\mathrm{coker}{\phi}$ is a band module and so 
$\pd{A(\Gamma)}{\mathrm{coker}{\phi}} \leq 1$ by Lemma \ref{Lem: band-Z-id}. 
Moreover, we can express the universal map $L' \to \Ext^1(M(\lambda, -\infty), N)$ as an 
inverse limit of epimorphisms of the form 
\[
\Ext^1(M(\lambda, -\infty), N) \overset{\Ext^1(\phi^n,N)}{\longrightarrow}\Ext^1(M(\lambda, -\infty), N)
\] 
and so it is an epimorphism (see, for example, \cite[Lem.~2.2]{Ext}). This proves that indeed 
$\dim{\Ext^1(M(\lambda, -\infty), N)} \leq \dim{L'}$.

Now suppose that $N$ is the generic module and $M$ is a string module. 
Then the associated arc of $N$ is $\delta = \beta$.  By \cite[Prop.~3]{RinZg}, $N$ is a direct 
summand of $M(\lambda, \infty)^K$ for any Pr\"ufer module $M(\lambda, \infty)$ and any 
infinite set $\k$.  If $\alpha$ and $\beta$ only cross in 3-cycles, then, by Appendix \ref{App: Ext}, 
we have that $0 = \Ext^1(M, M(\lambda, \infty))^K \cong \Ext^1(M, M(\lambda, \infty)^K)$ and 
so $\Ext^1(M,N) = 0$.  

Conversely, suppose $\alpha$ and $\beta$ cross outside of a 3-cycle.  Then $\alpha$ must be 
a bridging arc and $\Ext(M,X)\neq0$ for any finite-dimensional band module $X$. 
By \cite[Prop.~5]{RinZg}, for any infinite set $\mathcal{X}$ of finite-dimensional band 
modules from pairwise different homogeneous tubes, there exists a set $H$ such that 
$0 \to \bigoplus_{X\in \mathcal{X}}X \to \prod_{X\in \mathcal{X}}X \to N^{(H)} \to 0$ 
is an exact sequence.  
In fact, this is a pure-exact sequence (see, for example, \cite[Lem.~2.1.20]{PSL}) and so, 
combining this with the fact that, by Lemma \ref{Lem: band-Z-id}, we have 
$\id{A(\Gamma)}{\bigoplus_{X\in \mathcal{X}}X}\leq 1$, we have an exact sequence 
\[ 
0 \to \Ext^1(M, \bigoplus_{X\in \mathcal{X}}X) \to \Ext^1(M, \prod_{X\in \mathcal{X}}X) 
\to \Ext^1(M, N^{(H)}) \to 0.
\] 
Finally, since $M$ is finite-dimensional, $\Ext^1(M,-)$ commutes with direct sums and direct 
products.  It therefore follows that $\Ext^1(M, N^{(H)}) \cong \Ext^1(M, N)^{(H)} \neq0$ and 
hence $\Ext^1(M, N)$ is non-zero.
\end{proof}

Next we consider the case where both indecomposable pure-injective modules are band modules.  Since finite-dimensional band modules have self extensions, we only need to consider extensions between non-isomorphic infinite-dimensional band modules.

\begin{proposition}
If $M$ and $N$ are both band modules over $A(\Gamma)$, then the following statements hold. \begin{enumerate}
\item Suppose $M$ and $N$ are both infinite dimensional.  Then 
$\Ext^1_{A(\Gamma)}(M,N) \neq 0$ if and only if $M \cong M(\lambda,\infty)$ and 
$N \cong M(\lambda,-\infty)$ for $\lambda \in \k^*$.
\item If $M \cong M(\lambda,n)$ for $n \in \mathbb{N}$, then $\Ext^1_{A(\Gamma)}(M, M) \neq 0$.
\end{enumerate}
\end{proposition}

\begin{proof}
By  Remark \ref{rem:ideal-image}, 
we have $\Ext^1_{A(\Gamma)}(M,N) \cong \Ext^1_{B}(M,N)$ for any pair of indecomposable 
band modules where $B\cong \k \tilde{A}_{s-1}$.
The first statement then follows from 
\cite[Lemma 2.7]{BuanKrause} and the second statement is well-known.
\end{proof}

\paragraph{Middle terms of extensions}

In the case of string and 1-dimensional band modules, it is possible to identify the middle terms 
of the extensions in $\Mod{A(\Gamma)}$.  In fact, Proposition~\ref{prop:list-extensions} gives
a complete list of the extensions between string and 1-dimensional band modules that correspond to crossings of the associated arcs in $(\mathrm{S}, \mathrm{M})$. 

\begin{definition}
Consider the unique (periodic) $\mathbb{Z}$-string $u_\beta$. Consider the module 
$M^{\Pi}(\beta)$ that is defined exactly as $M^\oplus(\beta)$ but taking a direct product of 
copies of $\k$ instead of a direct sum.  In particular, we have 
\[
M^{\Pi}(\beta)\cong\prod_{i\in \mathbb{Z}}\k_i
\]
\noindent 
(as vector spaces) where $\k_i \cong \k$ for each $i\in \mathbb{Z}$.  Define $M^-(\beta)$ to be 
the submodule of $M^{\Pi}(\beta)$ such that $(\lambda_i)_{i\in\mathbb{Z}}$ is in 
$M^-(\beta) \subseteq M^{\Pi}(\beta)$ if and only if $\lambda_i=0$ for $i\ll 0$. Define 
$M^+(\beta)$ to be the submodule of $M^{\Pi}(\beta)$ such that $(\lambda_i)_{i\in\mathbb{Z}}$ 
is in $M^+(\beta) \subseteq M^{\Pi}(\beta)$ if and only if $\lambda_i=0$ for $i\gg 0$.  
\end{definition}

\begin{definition} 
Let $A(\Gamma)$ be the algebra associated to a triangulated $\Gamma$ of the annulus 
$(\mathrm{S},\mathrm{M})$ and let $\alpha$ and $\delta$ be arcs in 
$(\mathrm{S},\mathrm{M})$. 

In the 
statements of Proposition~\ref{prop:list-extensions}, 
$v_Le$, $d^{-1}v_R$, $w_Lc^{-1}$ and $aw_R$ are (possibly trivial) strings such that $a,c,d,e$ 
are direct letters whenever they are contained in a non-trival string.  Moreover $m$ is a 
(possibly trivial) finite string.  We define 
\[
*_\delta = \begin{cases} \prod & \text{if } \delta \text{ is an expanding asymptotic arc};\\
\bigoplus & \text{if } \delta \text{ is a contracting asymptotic arc};\\
\emptyset &\text{if } \delta \text{ is a finite arc}.\end{cases}
\]

\[
*_\alpha = \begin{cases} \prod & \text{if } \alpha \text{ is an expanding asymptotic arc};\\
\bigoplus & \text{if } \alpha \text{ is a contracting asymptotic arc};\\
\emptyset &\text{if } \alpha \text{ is a finite arc}.\end{cases}
\] 
and the notation $M^\emptyset(\alpha)$ will mean $M(\alpha)$ 
for any finite arc $\alpha$.
\end{definition}

\begin{proposition}\label{prop:list-extensions} 
Let $A(\Gamma)$ be the algebra associated to a triangulated $\Gamma$ of the annulus 
$(\mathrm{S},\mathrm{M})$ and let $\alpha$ and $\delta$ be arcs in 
$(\mathrm{S},\mathrm{M})$. 
\begin{enumerate}
\item 
Suppose $u_\delta = v_Lemd^{-1}v_R$ and $u_\alpha = w_Lc^{-1}maw_R$ are such that 
$v_Le$ and $w_Lc^{-1}$ are not both trivial and $d^{-1}v_R$ and $aw_R$ are not both trivial. 
Then there exists a non-split extension 
\[
0 \to M(\delta) \to M^{*_\delta}(\varepsilon_1)\oplus M^{*_\alpha}(\varepsilon_2) \to M(\alpha) \to 0
\] 
where  $u_{\varepsilon_1} = v_Lemaw_R$ and $u_{\varepsilon_2} = w_Lc^{-1}md^{-1}v_R$. 
We call such an extension an \textbf{overlap extension} between string modules.

\item 
Suppose there exists $l\geq 1$ such that $u_\alpha = v_Lemd^{-1}v_R$ and 
$b^l = w_Lc^{-1}maw_R$.  If we choose $l$ to be minimal, then there exists a non-split 
overlap extension 
\[
0 \to M(\alpha) \to M(\varepsilon) \to M(\lambda, 1) \to 0
\] 
for all $\lambda \in \k^*$ where 
\begin{enumerate}
\item[(i)] if $l=1$ then $u_\varepsilon = v_Lemaw_Rw_Lc^{-1}md^{-1}v_R$; and
\item[(ii)] if $l>1$ then $u_\varepsilon = v_Lem_2bm_1d^{-1}v_R$ such that $m = m_2m_1$ and $b = m_1aw_R$.
\end{enumerate} 
We call such an extension an \textbf{overlap extension} from a band module to a string module.

\item 
Suppose there exists $l \geq 1$ such that $b^l = v_Lemd^{-1}v_R$ and 
$u_\delta = w_Lc^{-1}maw_R$.  If we choose $l$ to be minimal, then there exists a non-split 
overlap extension 
\[
0 \to M(\lambda, 1)  \to M(\varepsilon) \to M(\delta) \to 0
\] 
for all $\lambda \in \k^*$ where 
\begin{enumerate}
\item[(i)] if $l=1$ then $u_\varepsilon = w_Lc^{-1}md^{-1}v_Rv_Lemaw_R$; and
\item[(ii)] if $l>1$ then $u_\varepsilon = w_Lc^{-1}m_2bm_1aw_R$ such that $m = m_2m_1$ 
and $b = m_1d^{-1}v_R$.
\end{enumerate} 
We call such an extension an \textbf{overlap extension} from a string module to a band module.

\item 
Suppose there exists an arrow $a\in Q_1$ such that either $u_\alpha a^{-1} u_\delta^{-1}$ or 
$u_\delta au_\alpha^{-1}$ is a finite string, an $\mathbb{N}$-string or $u_\beta$ and denote 
this string by $u_\varepsilon$.  Then there exists a non-split extension 
\[ 
0 \to M(\delta) \to M^*(\varepsilon) \to M(\alpha) \to 0
\] 
where 
\[
* = \begin{cases} + & \text{if } u_\varepsilon = u_\alpha a^{-1} u_\delta^{-1} = u_\beta;\\
- & \text{if } u_\varepsilon = u_\delta a u_\alpha^{-1} = u_\beta;\\
*_\alpha & \text{if } u_\varepsilon = u_\alpha a^{-1} u_\delta^{-1} \neq u_\beta;\\
*_\delta & \text{if } u_\varepsilon = u_\delta a u_\alpha^{-1} \neq u_\beta. 
 \end{cases}
\]  

We call such an extension an \textbf{arrow extension}.
\end{enumerate}
We refer to arrow and overlap extensions as \textbf{standard extensions}.
\end{proposition}

\begin{proof}
In \cite{CPS}, 
the authors prove this result for finite-dimensional modules. The extensions between infinite-dimensional modules are completely analogous.  
\end{proof}

\begin{remark}\begin{enumerate}
\item In cases \textbf{(2)} and \textbf{(3)} the arcs $\alpha$ and $\delta$ must be finite arcs 
since otherwise the above conditions on the strings $u_\alpha$ and $u_\delta$ cannot be 
satisfied.
\item In case \textbf{(4)} $\alpha$ must be contracting and $\delta$ must be expanding.  
\end{enumerate} 
\end{remark}

\begin{example}
The first three pictures illustrate case 1 of Proposition~\ref{prop:list-extensions}.  
\[
\includegraphics[width=15cm]{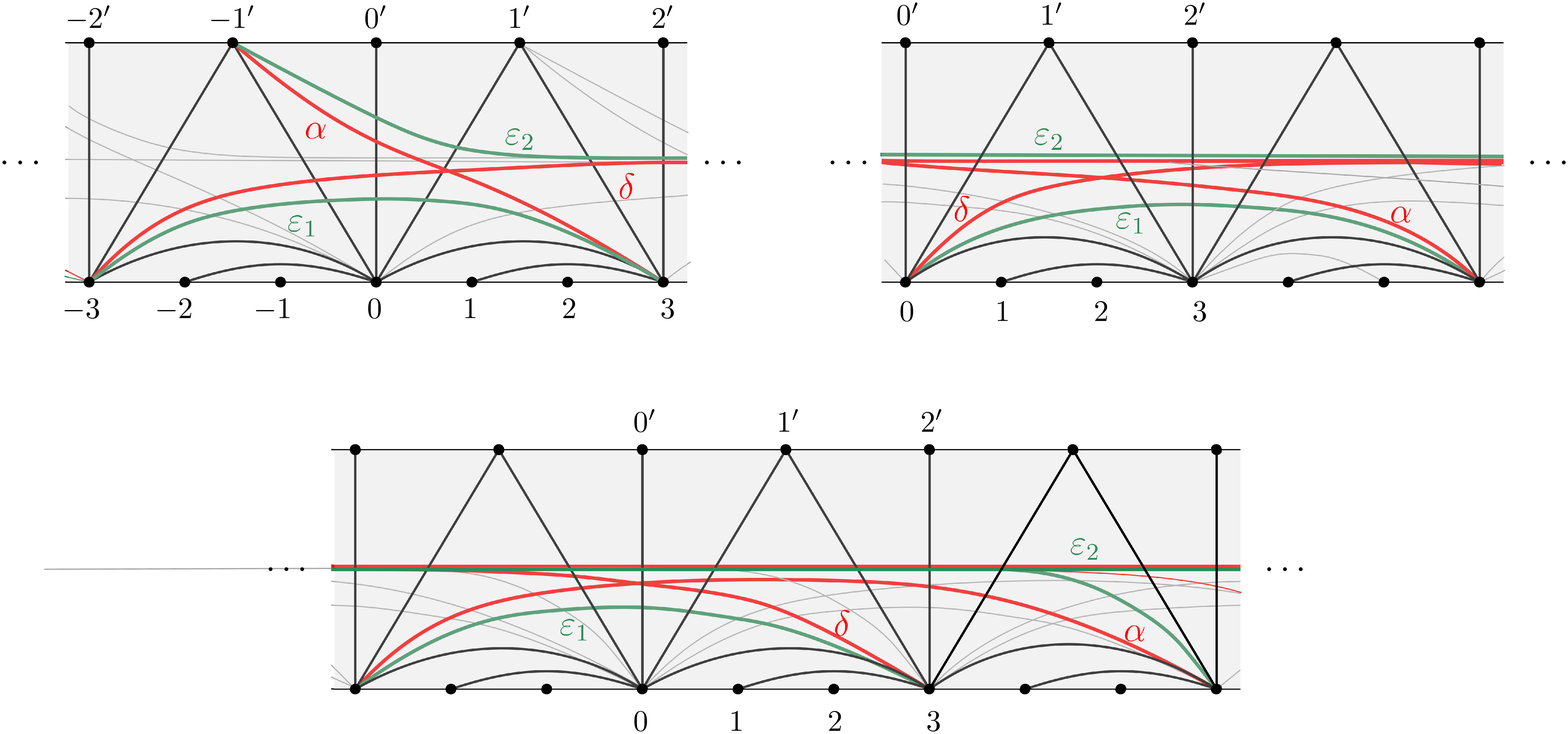}
\]
The second three pictures illustrate case 4. 
\[
\includegraphics[width=15cm]{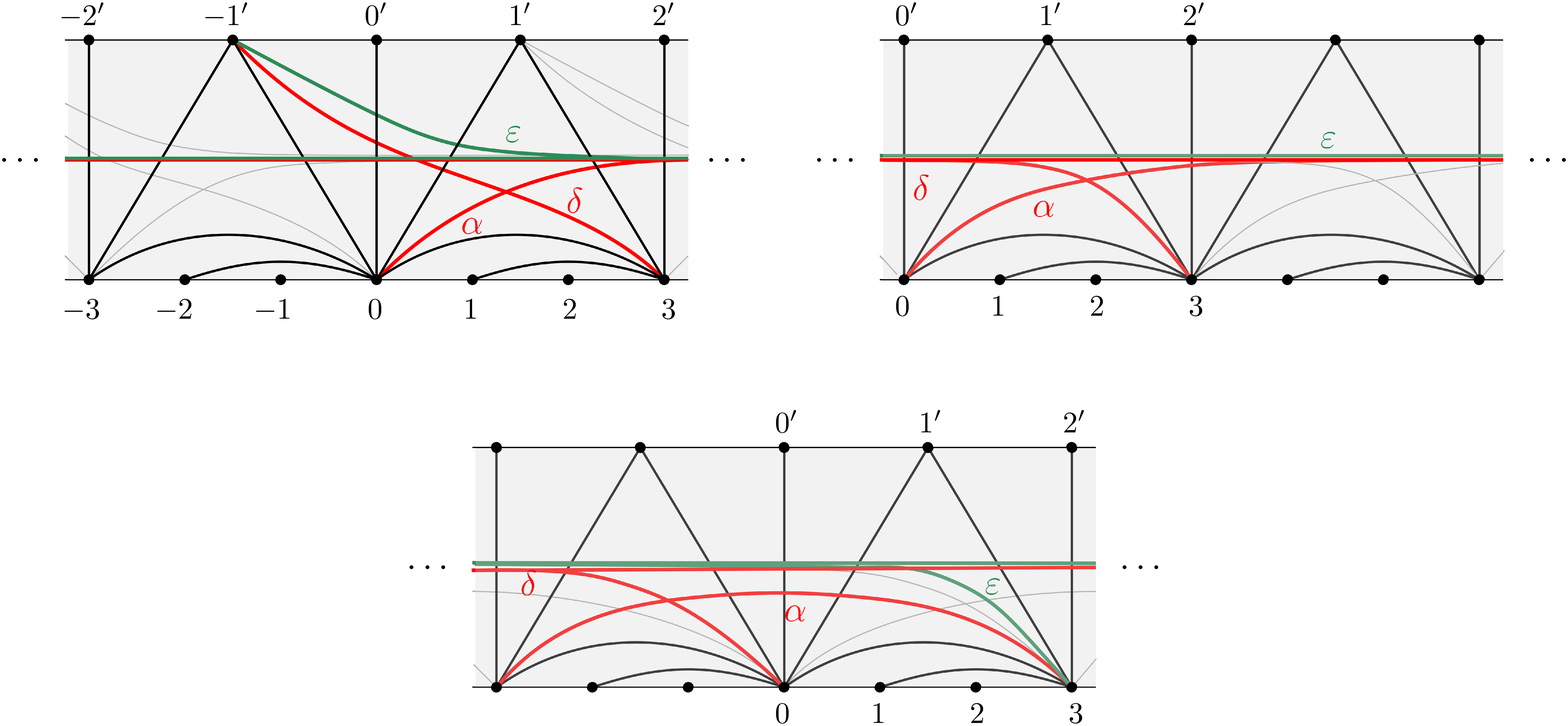}
\]
\end{example}

%%%%%%%
%
\subsection{The classification of cosilting modules.}\label{Sec: Asymp_tri}

In this section will classify the cosilting modules over $A(\Gamma)$ for a triangulation $\Gamma$ 
of an annulus $(\mathrm{S},\mathrm{M})$, with quiver $Q=Q(\Gamma)$ and 
$A(\Gamma)=\k Q/I(\Gamma)$. 
We begin by introducing the $A(\Gamma)$-modules $M(t)$ that we will show are cosilting.

Let $T$ be a collection of non-crossing arcs in $(\mathrm{S},\mathrm{M})$. We will call $T$ \textbf{strictly asymptotic} if it contains an asymptotic arc. 
If either 
\begin{enumerate}
\item[(i)] $t = (T, P_1, P_2)$ where $T$ is strictly asymptotic and $P_1$ and $P_2$ 
are disjoint subsets of $\k^*$ or
\item[(ii)] $t = T$ where $T$ is not strictly asymptotic,
\end{enumerate} 
then we call $t$ a \textbf{partial asymptotic triangulation}. 
If $T$ is a maximal collection of non-crossing asymptotic arcs and 
$\k^* = P_1\cup P_2$ (when $t$ is strictly asymptotic), 
then $t$ is called an \textbf{asymptotic triangulation}.

 \begin{remark}\label{Rem: parametrising sets}
Note that the strictly asymptotic triangulations are in bijection with the set $\mathcal{A}(\mathrm{S}, \mathrm{M}) \times \mathcal{P}(\k^*)$ defined in the Main Theorem in the introduction.  Similarly the asymptotic triangulations that are not strictly asymptotic are in bijection with the set $\mathcal{T}(\mathrm{S}, \mathrm{M})$.   
\end{remark}

Let $t$ be a partial asymptotic triangulation, then we define a set of indecomposable pure-injective objects as follows.
If $t$ is strictly asymptotic, then let
\[ 
\mathcal{N}_t :=  \{ M(\alpha) \mid \alpha \in T\} \cup \{M(\lambda, \infty) \mid \lambda \in P_1\} \cup 
\{M(\lambda, -\infty) \mid \lambda \in P_2\}\cup\{G\}.
\] 
If $t$ is not strictly asymptotic, then let \[\mathcal{N}_t :=  \{ M(\alpha) \mid \alpha \in T\}.\]

\noindent 
For each partial asymptotic triangulation $t$, let \[M(t) := \prod_{N \in \mathcal{N}_t} N.\]

%%%%%%%%
%
\subsubsection{Annihilators of asymptotic triangulations.}\label{Sec: annihilators}

We will consider partial asymptotic triangulations $t$ and consider the associated $A(\Gamma)$-module $M(t)$.  In order to show that $M(t)$ is a cosilting module over $A(\Gamma)$ whenever $t$ is an asymptotic triangulation, we will make use of the characterisation of cosilting modules over left artinian rings given in Corollary \ref{cor: max rig cosilt}.  The first step will be to understand the annihilator ideal $\ann{M(t)}$.  The following observation holds for an arbitrary gentle algebra; see \cite{Ringel} for definitions of (possibly infinite-dimensional) string and band modules.

\begin{observation}\label{Obs: ann}
Let $B = \k Q/I$ be an arbitrary gentle algebra.  
Then the following hold.  
\begin{enumerate}
\item If $M$ is a string module corresponding to a string $u$, then $\ann{M}$ 
has a $\k$-linear basis consisting of the paths $p$ in $(Q,I)$ such that neither 
$p$ nor $p^{-1}$ is contained in $u$. 
\item 
If $N$ is a band module corresponding to a band $b$, then $\ann{N}$ has a $\k$-linear basis 
consisting of the paths $q$ in $(Q, I)$ such that neither $q$ nor $q^{-1}$ is contained in 
$\mathbb{Z}$-string ${}^\infty b^\infty$.
\end{enumerate}
\end{observation}

For the rest of the section, let $Q=Q(\Gamma)$ and $I=I(\Gamma)$ for $\Gamma$ a triangulation of 
an annulus, with string algebra $A(\Gamma)=\k Q/I$ and let $t$ be a (partial) asymptotic triangulation. 

\begin{lemma}\label{Lem: ann-of-ind}
The annihilator ideal $\ann{M(t)}$ in $A(\Gamma)$ has a $\k$-linear basis consisting of all paths $p$ satisfying the following conditions.\begin{enumerate}
\item Neither $p$ nor $p^{-1}$ occur as a substring of $u_\alpha$ for any $\alpha\in T$. 
\item If $t$ is strictly asymptotic, then neither $p$ nor $p^{-1}$ occur as a substring of $u_\beta$.\end{enumerate}
\end{lemma}
\begin{proof}
The lemma follows immediately from the above observation since \[\ann{M(t)} = \bigcap_{M\in \mathcal{N}_t} \ann{M}.\]
\end{proof}

Given Lemma \ref{Lem: ann-of-ind}, we will use $\ann{t}$ to denote the annihilator ideal associated to the module $M(t)$.   

\begin{lemma}\label{Lem: paths arrows}
Let $t$ be a partial asymptotic triangulation of $(\mathrm{S},\mathrm{M})$ and let $p = p_n\dots p_0$ be a path of length $n+1$ in $(Q, I(\Gamma))$.  Then $p \in \ann{t}$ if and only if $p_i \in \ann{t}$ for some $i\in \{0,\dots,n\}$.
\end{lemma}
\begin{proof}
By definition, either $t= T$ or $t = (T,P_1, P_2)$.  Suppose $p$ is contained in $\ann{t}$ but $p_0$ is not.  By Lemma \ref{Lem: ann-of-ind} there is some $\alpha \in T$ such that $u_\alpha$ contains $p_0$ or $p_0^{-1}$ as a substring or $T$ is strictly asymptotic and $u_\beta$ contains $p_0$ or $p_0^{-1}$ as a substring.  Any arc $\gamma$ such that $u_\gamma$ contains $p_n \dots p_1$ or $p_1^{-1}\dots p_n^{-1}$ as a substring but does not contain $p$ or $p^{-1}$ as a substring will necessarily cross $\alpha$ and so cannot be contained in $T$.  Thus $p_n \dots p_1 \in \ann{t}$ by Lemma \ref{Lem: ann-of-ind}.  By induction on the length of $p$ we have that $p \in \ann{t}$ implies $p_i \in \ann{t}$ for some $0 \leq i \leq n$.  The converse is immediate.
\end{proof}

\begin{corollary}\label{cor: ann is gentle}
The algebra $A(\Gamma)/\ann{t}$ is a gentle algebra.
\end{corollary}
\begin{proof}
By Lemma \ref{Lem: paths arrows}, the algebra $A(\Gamma)/\ann{t}$ is obtained from $A(\Gamma)$ by removing arrows and vertices and so the axioms of a gentle algebra will still be satisfied by $A(\Gamma)/\ann{t}$.  
\end{proof}
 
Next we consider the annihilator of an asymptotic triangulation $t$.  In this case we can describe the annihilator in more detail.  In particular, $\ann{t}$ will contain arrows satisfying the following definition with respect to some $\alpha \in T$.

\begin{definition}
Let $a$ be an arrow in $Q$ with $s(a) = i$ and $t(a) = j$.  We say that an arc $\alpha$ \textbf{crosses 
$a$ in a 3-cycle} if $i$ and $j$ are edges of an internal triangle in $\Gamma$ and either $s(u_\alpha) = k$ 
or $t(u_\alpha) =k$ where $k$ is the remaining edge in the triangle, see figure below. 
\[
\includegraphics[width=3cm]{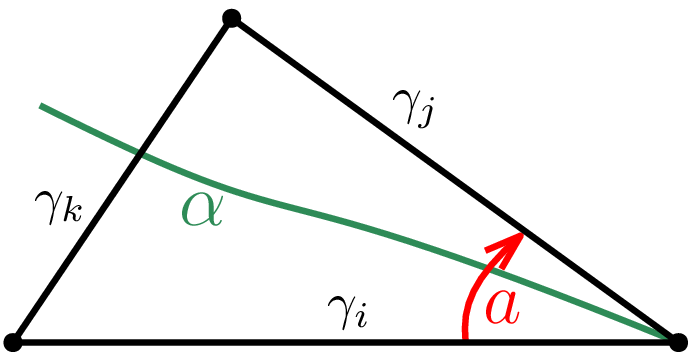}
\]
\end{definition}
Note that this means that $\alpha$ and the arc corresponding to $a$ cross in a 3-cycle, 
cf. Definition~\ref{def:cross-in-3-cycle}. \\

\begin{lemma}\label{Lem: vert-trans} 
Let $t$ be an asymptotic triangulation with corresponding $T$. Then the following statements hold.
\begin{enumerate}
\item Let $a:i\to j$ be an arrow in $Q$. 
Then $a$ is contained in $\ann{t}$ if and only if one of the following conditions hold. 
\begin{enumerate}
\item One of $\gamma_i$ or $\gamma_j$ is contained in $T$.
\item There exists some $\alpha \in T$ such that $\alpha$ crosses $a$ in a 3-cycle.
\end{enumerate}
\item Let $i$ be a vertex in $Q$.  Then $e_i$ is contained in $\ann{t}$ if and only if $\gamma_i\in T\cap\Gamma$.
\end{enumerate}
\end{lemma}
\begin{proof}
\textbf{(1)} If either of the conditions hold then, it follows from Lemma \ref{Lem: ann-of-ind}, that $a \in \ann{t}$.  We will therefore show the converse implication.

So let $a \in \ann{t}$.  First suppose that $\gamma_i$ and $\gamma_j$ are edges of a triangle such 
that the remaining edge is a boundary component and suppose $\gamma_i$ is not contained in $T$.  
As $T$ is a triangulation, there must be some arc $\alpha \in T$ that crosses $\gamma_i$.  
By Lemma \ref{Lem: ann-of-ind}, neither $a$ nor $a^{-1}$ occur as a substring of $u_\alpha$ and so 
$s(u_\alpha) = i$ or $t(u_\alpha) = i$.  Any arc $\gamma$ that crosses $\gamma_j$ such that $u_\gamma$ does 
not contain $a$ or $a^{-1}$ as a substring will cross $\alpha$ and so cannot be contained in $T$. 
So $\gamma_j$ is compatible with all arcs in $T$ and we must 
therefore have that $\gamma_j \in T$.  

Next suppose that $\gamma_i$ and $\gamma_j$ are edges of an internal triangle in $\Gamma$ with remaining 
edge $\gamma_k$.  If $\gamma_i$ is not contained in $T$, then there exists an arc $\alpha \in T$ that crosses 
$\gamma_i$.  Since $u_\alpha$ cannot contain $a$ or $a^{-1}$ as a substring, either $s(u_\alpha) = i$ or 
$t(u_\alpha) = i$ or $\alpha$ crosses $\gamma_k$ immediately after $\gamma_i$.  In the first situation, it follows 
that $\gamma_j$ must be contained in $\ann{t}$. In the second situation, there must be an arc $\gamma$ in $T$ 
that crosses $a$ in a 3-cycle and then agrees with $\alpha$ until its endpoint.  As before, this arc is compatible with all of $T$ and so is in $T$. 

\textbf{(2)}  
The statement follows immediately from the fact that $\gamma_i$ is crossed by an arc in $T$ if and only if 
$\gamma_i$ is not contained in $T$.  
\end{proof}

\noindent We now associate to any asymptotic triangulation $t$ a new quiver $Q_t$ as follows: 
\begin{enumerate} 
\item[I.] 
If $\gamma_i$ is contained in $\Gamma \cap T$, remove the  vertex $i$ from $Q$ and all 
arrows incident with $i$. 
\item[II.] 
If $a$ is an arrow in $Q$ that is crossed in a 3-cycle by some $\alpha \in T$, then remove $a$ from $Q$. 
\end{enumerate}
Let $I_t$ be the admissible ideal in $\k Q_t$
generated by the generating relations in $I(\Gamma)$ that are supported on 
$Q_t$.

\begin{corollary}\label{Cor: new-quiver}
Let $t$ be an asymptotic triangulation with quiver $Q_t$ and ideal $I_t$. 
Then $A(\Gamma)/ \ann{t}$ is isomorphic to the algebra $\k Q_t/I_t$. 
\end{corollary}

%%%%%%%%
%
\subsubsection{Cosilting modules are partial asymptotic triangulations.}

Recall from Lemma \ref{Lem: no sup dec} and Proposition \ref{Prop: Cosilt is unfaithful cotilt} that every cosilting module $C$ is equivalent to a 
cosilting module with no superdecomposable part.  As in Remark \ref{Rem: ess unique}, we will use the notation 
$\mathcal{N}_C$ to denote the set of indecomposable pure-injective summands of $C$ up to isomorphism.  

\begin{proposition}\label{Prop: cosilt are partial}
Let $C$ be a cosilting module over $A(\Gamma)$.  There exists a partial asymptotic triangulation 
$t$ such that $\Prod{\mathcal{N}_C} = \Prod{\mathcal{N}_t}$. 
\end{proposition}
\begin{proof}
By Theorem \ref{Thm: cotilt vs rigid} and Proposition \ref{Prop: Cosilt is unfaithful cotilt}, the set $\mathcal{N}_C$ is a maximal rigid system in the category $\Mod{A(\Gamma)/\ann{C}}$.  We consider $\Mod{A(\Gamma)/\ann{C}}$ as a full subcategory of $\Mod{A(\Gamma)}$.  

Let $T := \{ \alpha \mid M(\alpha) \in \mathcal{N}_C\}$.  We will show that $T$ is a collection of non-crossing arcs in $(\mathrm{S},\mathrm{M})$.  Let $\alpha$ and $\delta$ be contained in $T$.  By Proposition \ref{Prop: Cosilt is unfaithful cotilt}, we have that $\Cogen{C}\subseteq \Mod{A(\Gamma)/\ann{C}}$ is closed under extensions. It follows that $\Ext_{A(\Gamma)/\ann{C}}^1(M(\alpha), M(\delta)) = 0$ if and only if $\Ext_{A(\Gamma)}^1(M(\alpha), M(\delta)) = 0$.  Similarly $\Ext_{A(\Gamma)/\ann{C}}^1(M(\delta), M(\alpha))=0$ if and only if $\Ext_{A(\Gamma)}^1(M(\delta), M(\alpha)) = 0$.  By Theorem \ref{Thm: extensions} the arcs $\alpha$ and $\delta$ cannot cross outside of 3-cycles or else this would contradict that $\mathcal{N}_C$ is a maximal rigid system in $\Mod{A(\Gamma)/\ann{C}}$.

Suppose that $\alpha$ and $\delta$ cross in a 3-cycle.  Without loss of generality, 
suppose we have the following local configuration. 
\[
\includegraphics[width=2.5cm]{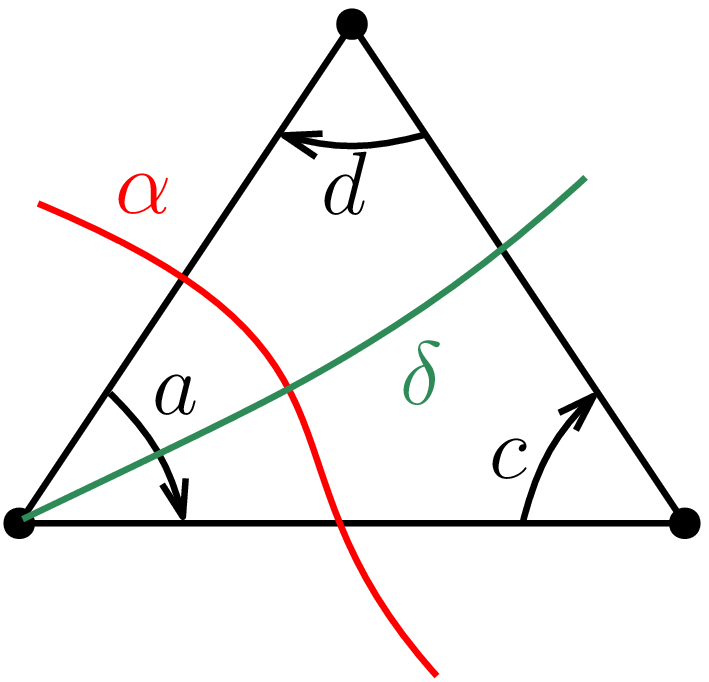}
\]
Then we must have that $c \in \ann{C}$ because otherwise $\id{A(\Gamma)/\ann{C}}{M(\delta)} > 1$ by Lemma \ref{Lem: id-gentle}, which contradicts the fact that $\mathcal{N}_C$ is a maximal rigid system in $\Mod{A(\Gamma)/\ann{C}}$.  Let $\varepsilon$ be the arc that crosses $d$ in a 3-cycle and follows $\alpha$ until its endpoint (if possible) or until 
$\varepsilon$ meets an inverse letter.  That is, we have the following local configuration.
\[
\includegraphics[width=6cm]{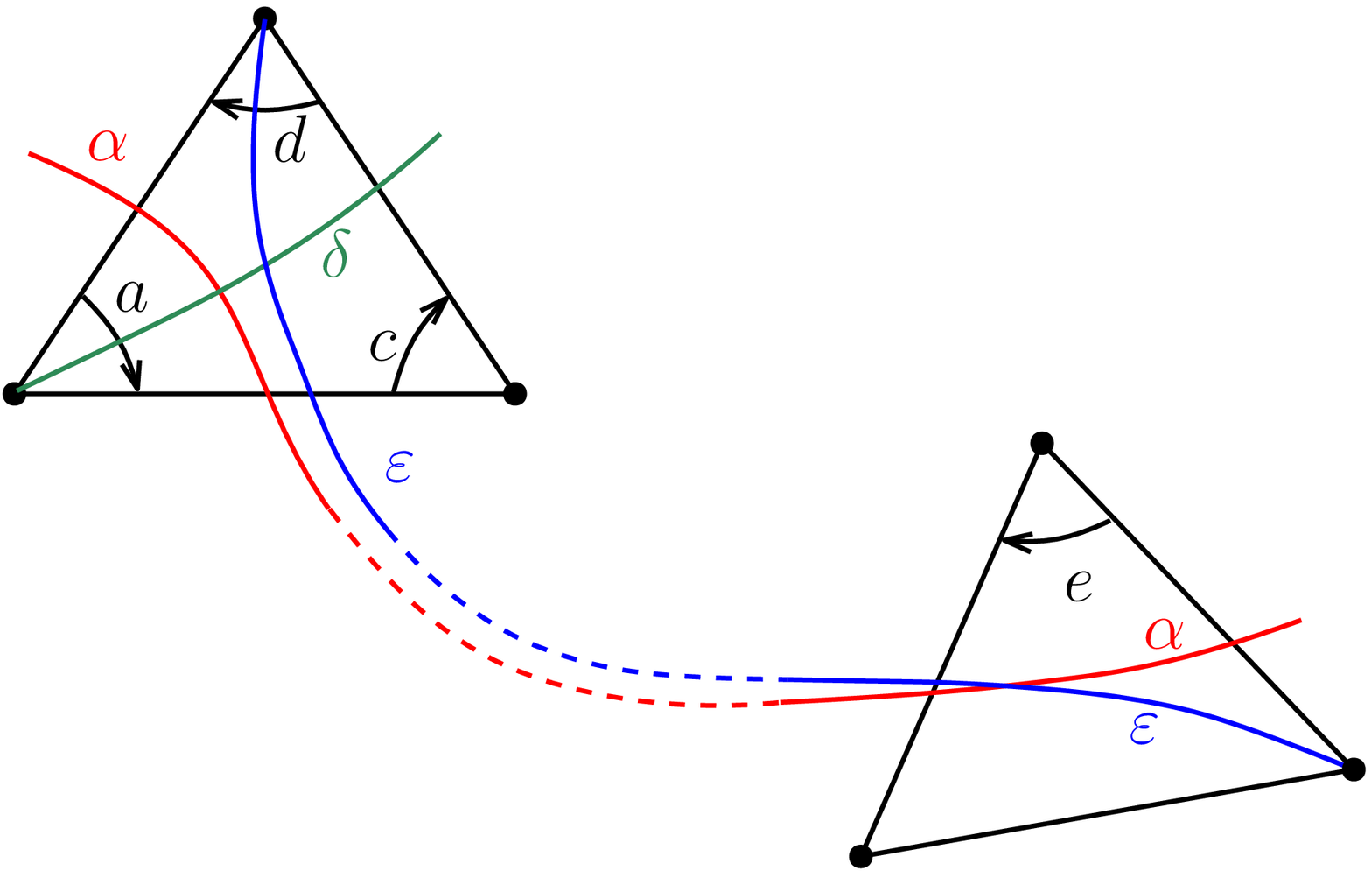}
\]

Then $M(\varepsilon)$ is a finite-dimensional submodule of $M(\alpha)$ and hence $M(\varepsilon)\in \Cogen{C}$.  However, there is an arrow extension $0 \to M(\delta) \to M(\nu) \to M(\varepsilon) \to 0$ where $u_\nu = u_\delta c u_\varepsilon$ in $\Mod{A(\Gamma)}$ and $M(\nu)$ is not contained in $\Cogen{C} \subseteq \Mod{A(\Gamma)/\ann{C}}$ because $c\in \ann{C}$.  This contradicts that $\Cogen{C}$ is closed under extensions and so $\alpha$ and $\delta$ do not cross.

If $T$ is strictly asymptotic, then define $t:= (T, P_1, P_2)$ where $P_1 := \{ \lambda \in \k^*\mid M(\lambda, \infty) \in \mathcal{N}_C\}$ and $P_2 := \{ \lambda \in \k^*\mid M(\lambda, -\infty) \in \mathcal{N}_C\}$.  Note that $P_1 \cap P_2 = \emptyset$ by Theorem \ref{Thm: extensions} and the same reasoning as above.  If $T$ is not strictly asymptotic, then define $t := T$.  We have shown that $t$ is a partial asymptotic triangulation.  Clearly $\mathcal{N}_C$ is a rigid system in $\Mod{A(\Gamma)/\ann{C}}$ and $\mathcal{N}_C \subseteq \mathcal{N}_t$ (with a proper inclusion exactly when $G$ is not contained in $\mathcal{N}_C$).  Since $\mathcal{N}_C$ is maximal, we may conclude that $\Prod{\mathcal{N}_C} = \Prod{\mathcal{N}_t}$.
\end{proof}

%%%%%%%%
%
\subsubsection{Asymptotic triangulations are cosilting modules.}

Let $t$ be an asymptotic triangulation of $(\mathrm{S},\mathrm{M})$.  In this section we show that $M(t)$ is a cosilting module over $A(\Gamma)$.  By Corollary \ref{cor: max rig cosilt}, it is enough to show that $\mathcal{N}_t$ is a maximal rigid system in $\Mod{A(\Gamma)/\ann{t}}$ and that $\Cogen{M(t)}$ is a torsion-free class.

\begin{proposition}\label{Prop: asym max rig}
Let $t$ be an asymptotic triangulation.  The set $\mathcal{N}_t$ is a maximal rigid system in $\Mod{A(\Gamma)/\ann{t}}$ and $M(t)$ is a cotilting $A(\Gamma)/\ann{t}$-module. 
\end{proposition}
\begin{proof}
First note that, by Theorem \ref{Thm: extensions}, we have that $\Ext^1_{A(\Gamma)}(M,N) = 0$ for all $N,M \in \mathcal{N}_t$ and so it follows that also $\Ext^1_{A(\Gamma)/\ann{t}}(M,N) = 0$ for all $N,M \in \mathcal{N}_t$.

Next we show that $\id{A(\Gamma)/\ann{t}}{M} \leq 1$ for all $M\in \mathcal{N}_t$.  If $M$ is a band module then $\id{A(\Gamma)/\ann{t}}{M} \leq \id{A(\Gamma)}{M} \leq 1$ by Lemma \ref{Lem: band-Z-id}.  Suppose $M \cong M(\alpha)$ for some arc $\alpha$.  We show that the conditions of Lemma \ref{Lem: id-gentle} holds for $u_\alpha$.  Suppose that the endpoint of $\alpha$ is contained in a triangle with edges $\gamma_i$, $\mu$ and $\nu$ such that $\nu$ follows $\mu$ in the 
anticlockwise direction.

\[
\includegraphics[width=7cm]{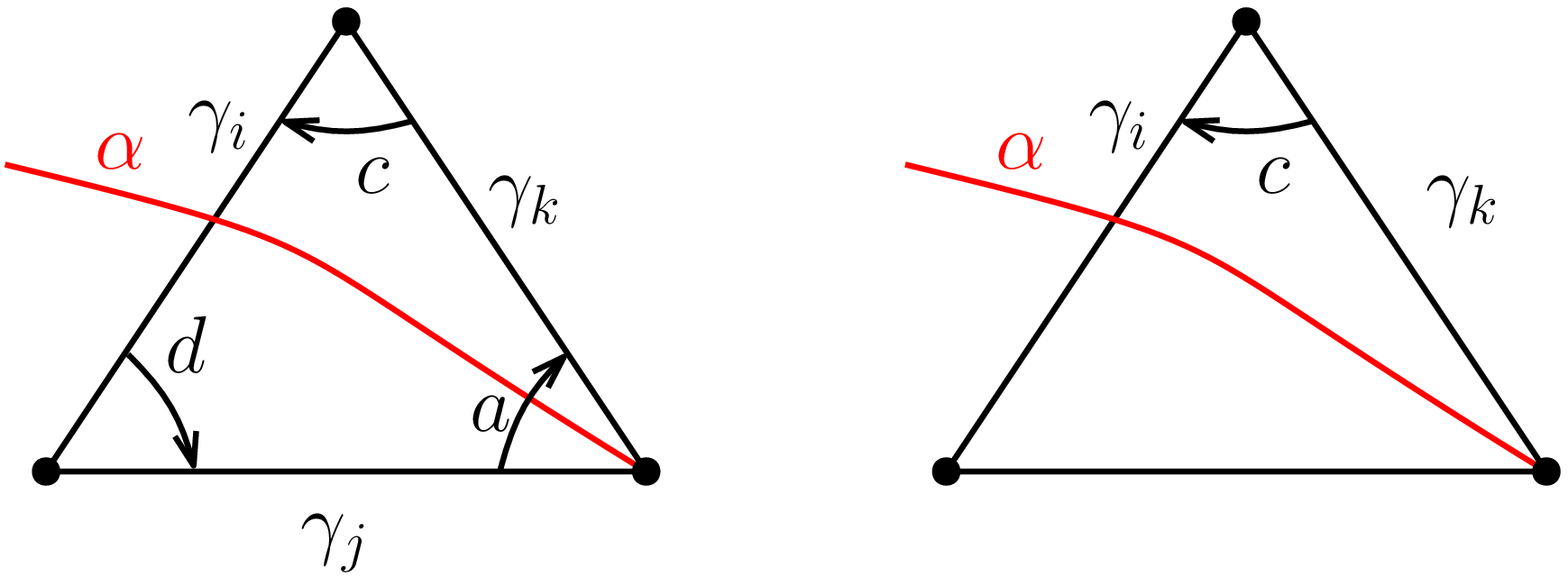}
\]
Without loss of generality assume that $s(u_\alpha) =i$.  If all of $\gamma_i$, $\mu$ 
and $\nu$ are arcs in $\Gamma$, say $\mu = \gamma_j$ and $\nu = \gamma_k$, 
then there are arrows $a \colon j \to k$, $c \colon k \to i$ and $d \colon i \to j$ in $Q$.  
Moreover, the arrow $c$ is the unique arrow such that $u_\alpha c$ is a string 
and the arrow $a$ is the unique arrow such that $ca \in I(\Gamma)$. 
However, the arc $\alpha$ crosses $a$ in a 3-cycle so by 
Corollary~\ref{Cor: new-quiver} there is no arrow d in $Q_t$ such that $cd \in I_t$. 
Thus condition \textbf{(i)} of Lemma \ref{Lem: id-gentle} holds. 
In a similar way, if either or both of $\mu$ and $\nu$ are boundary components, 
then, even if there exists an arrow $c$ such that $u_\alpha c$ is a string, there is no arrow $a$ such that $ca \in I_t$. 
If $\alpha$ is finite, we may run the same argument at the other endpoint of 
$\alpha$ and we have that $\id{A(\Gamma)/\ann{t}}{M} \leq 1$ by 
Proposition~\ref{Lem: id-gentle}.

We have shown that $\mathcal{N}_t$ is a rigid system in $\Mod{A(\Gamma)/\ann{t}}$.  Finally we must show that it is maximal.  Let $\mathcal{L}$ be a rigid system in $\Mod{A/\ann{t}}$ such that $\Prod{\mathcal{N}_t} \subseteq \Prod{\mathcal{L}}$.  Without loss of generality, we may assume that $M(t)$ is a direct summand of the direct product $L$ of the objects in $\mathcal{L}$ i.e.~we have that $ L \cong M(t) \oplus X$ for some module $X$.  Now, by Lemma \ref{Lem: self-orth}, we must have that $\Ext^1_{A(\Gamma)/\ann{t}}( L,  L) =0$.  It follows that $\Ext^1_{A(\Gamma)/\ann{t}}(Y, M(t)) = 0 = \Ext^1_{A(\Gamma)/\ann{t}}(M(t), Y)$ for all indecomposable summands $Y$ of $X$.  Suppose there is an indecomposable summand $Y$ of $X$ such that $\Ext^1_{A(\Gamma)}(Y, M(t)) \neq 0$ or $\Ext^1_{A(\Gamma)}(M(t), Y) \neq 0$.  By Lemma \ref{Lem: standard maps and standard extensions}, there exists an arrow or an overlap extension between $Y$ and some $N \in \mathcal{N}_t$.  By Lemma \ref{Lem: vert-trans}, the middle term of any overlap extension between $Y$ and $M$ will be contained in $\Mod{A(\Gamma)/\ann{t}}$.   Therefore the only possibility is the following.  There is some $\delta \in T$ and a arc $\varepsilon$ such that $Y \cong M(\varepsilon)$ such that there is an arrow extension between $M(\delta)$ and $M(\varepsilon)$ with the corresponding arrow $a$ contained in $\ann{t}$.  However, by Lemma \ref{Lem: vert-trans}, this implies that there is a arc $\alpha \in T$ such that $\alpha$ crosses $a$ in a 3-cycle.  But then $\alpha$ crosses $\delta$, which is a contradiction.

So then all indecomposable summands of $X$ must be contained in $\mathcal{N}_t$ because any other indecomposable pure-injective module will have a non-trivial extension with $M(t)$.  By \cite[Thm.~5.6]{PP} there are no superdecomposable pure-injective modules over $A/\ann{t}$ and so, by Remark \ref{Rem: Prod = Prod}, we have shown that $\Prod{\mathcal{L}} = \Prod{\mathcal{N}_t}$.

By Theorem \ref{Thm: cotilt vs rigid}, we also have that $M(t)$ is a cotilting module over $A(\Gamma) /\ann{t}$.
\end{proof}

\begin{proposition}\label{Prop: asymp is cosilting}
If $t$ is an asymptotic triangulation, then $M(t)$ is a cosilting $A(\Gamma)$-module.
\end{proposition}
\begin{proof}
By Corollary \ref{cor: max rig cosilt} and Proposition \ref{Prop: asym max rig}, it remains to show that $\Cogen{M(t)}$ is a torsion-free class in $\Mod{A(\Gamma)}$.  We know that $\Cogen{M(t)}$ is a torsion-free class in $\Mod{A(\Gamma)/\ann{t}}$ because $M(t)$ is a $A(\Gamma)/\ann{t}$-cotilting module by Proposition \ref{Prop: asym max rig}. 
In particular, if $M,N$ are in $\Cogen{M(t)}$ such that $0 \to M \to L \to N \to 0$ is 
a short exact sequence in $\Mod{A(\Gamma)/\ann{t}}$, then $L\in \Cogen{M(t)}$. 
As we saw in the proof of Proposition \ref{Prop: asym max rig}, any overlap 
extension in $\Mod{A(\Gamma)/\ann{t}}$ is also an overlap extension in 
$\Mod{A(\Gamma)}$.  We must rule out the possibility that there is an arrow 
extension $0 \to M \to L \to N \to 0$ in $\Mod{A(\Gamma)}$ with 
$L \notin \Cogen{M(t)}$.

Suppose there is an arrow extension between $M$ and $N$.  Without loss of generality we may assume that $M\cong M(\alpha)$ and $N \cong M(\delta)$ for arcs $\alpha$ and $\delta$ and that there is an arrow $a \in \ann{t}$ 
such that $u_\delta a^{-1}u_\alpha$ is a string.  By Lemma \ref{Lem: vert-trans}, 
this means that there is $\varepsilon\in T$ such that $\varepsilon$ crosses $a$ 
in a 3-cycle.  We must therefore have the following local configuration.  

\[
\includegraphics[width=3cm]{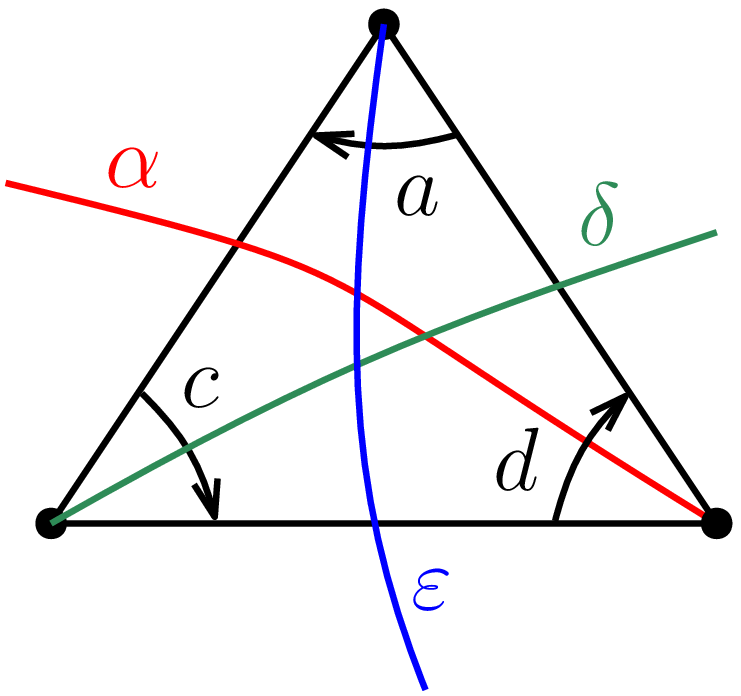}
\]

The arrows $c$ and $d$ cannot be in $\ann{t}$ because any arc crossing them 
in a 3-cycle would cross $\varepsilon$. 
Thus $\Ext^1_{A(\Gamma)/\ann{t}}(M(\alpha), M(\varepsilon)) \neq 0$ since 
$u_\alpha c^{-1}u_\varepsilon$ is a string in $(Q_t, I_t)$.  But this contradicts that 
$\Cogen{M(t)} = {}^{\perp}M(t)$ in $\Mod{A(\Gamma)/\ann{t}}$.
\end{proof}

%%%%%%%%
%
\subsubsection{The classification.}

\begin{lemma}\label{Lem: add arcs}
If $t$ is a partial asymptotic triangulation, with corresponding $T$, 
then the following statements hold. 
\begin{enumerate}
\item 
For $i\in Q$, $e_i \in \ann{t}$ if and only if $\gamma_i \in \Gamma$ 
does not cross any arc in $T$.
\item 
Let $a$ be an arrow in $Q$ such that $e_{s(a)}$ and $e_{t(a)}$ are not contained 
in $\ann{t}$.  Then $a\in \ann{t}$ if and only if one of the following statements hold.
\begin{enumerate}
\item[(i)] There is an arc $\alpha \in T$ that crosses $a$ in a 3-cycle.
\item[(ii)] There is an arc $\alpha$ in $(\mathrm{S},\mathrm{M})$
that crosses $a$ in a 3-cycle such that 
$\ann{t} = \ann{t\cup\{\alpha\}}$. 
Moreover, the arc $\alpha$ does not cross any arc in $T$.  
\end{enumerate}
\end{enumerate}
\end{lemma}
\begin{proof}
\textbf{(1)} This is clear from Observation \ref{Obs: ann}.

\textbf{(2)} $\mathbf{(\Leftarrow)}$ In both case $(i)$ and $(ii)$, we must have $a\in \ann{t}$ because there cannot be any arc $\delta\in T$ such that $u_\delta$ contains $a$ or $a^{-1}$ as a substring since such a $\delta$ would cross $\alpha$.  

$\mathbf{(\Rightarrow)}$ Suppose $a\in \ann{t}$ and there is no arc in $T$ that crosses $a$ in a 3-cycle. First we will argue that there is no arc $\delta$ in $T$ such that $s(u_\delta) = s(a)$ or $s(u_\delta) = t(a)$.  If there is a $\delta \in T$ such that $s(u_\delta) = t(a)$, then $e_{s(a)}\in \ann{t}$ because any other arc $\eta$ that crosses $\gamma_{s(a)}$ must either cross $\delta$ or $u_\eta$ must contain $a$ or $a^{-1}$ as a substring.  Such an $\eta$ cannot be contained in $T$. A similar argument yields that, if there is a $\delta \in T$ such that $s(u_\delta) = s(a)$, then $e_{t(a)} \in \ann{t}$.  We have assumed that $e_{s(a)}$ and $e_{t(a)}$ are not contained in $\ann{t}$ so there is no arc $\delta$ in $T$ such that $s(u_\delta) = s(a)$ or $s(u_\delta) = t(a)$.

It follows that there exists an arc $\alpha$ that crosses $a$ in a 3-cycle and does not cross any arcs in $T$.  Finally we must show that we can choose $\alpha$ such that $\ann{t} = \ann{t\cup\{\alpha\}}$.  We have that $\ann{t\cup\{\alpha\}} \subset \ann{t}$, so we must show that $u_\alpha$ does not contain any (trivial or non-trivial) letters that are not already contained in $u_\delta$ for some $\delta\in T$.  Let $\gamma_k \in \Gamma$ be such that $\gamma_{s(a)}, \gamma_{t(a)}, \gamma_k$ form a triangle.   The conclusion of the first part of this proof tells us that we have some arcs $\delta_1$ and $\delta_2$ in $T$ such that $\delta_1$ crosses $\gamma_{s(a)}$ followed by $\gamma_k$ and $\delta_2$ crosses $\gamma_{t(a)}$ followed by $\gamma_k$.  Choose $\alpha$ to cross $a$ in a 3-cycle and then follow $\delta_1$ to its endpoint.  
\[
\includegraphics[width=4cm]{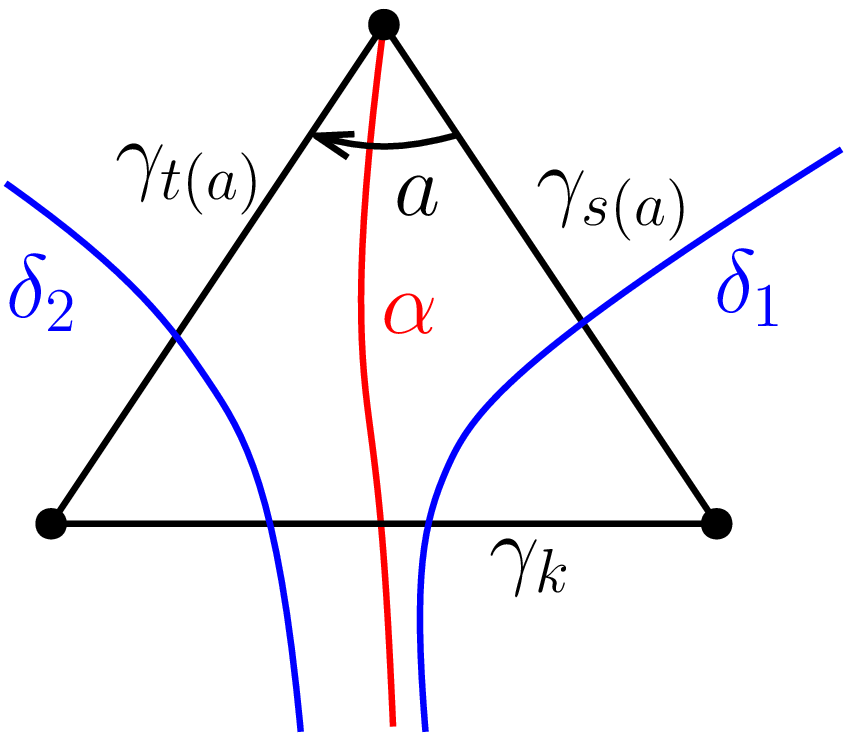}
\]
\end{proof}

\begin{lemma}\label{Lem: complete}
Let $t$ be a partial asymptotic triangulation.  There exists an asymptotic triangulation $s$ such that $\mathcal{N}_t \subseteq \mathcal{N}_s$ and $\ann{t} = \ann{s}$.
\end{lemma}
\begin{proof}
By definition, we have that either $t=T$ or $t = (T,P_1,P_2)$. 
For the second case, let $Q_1 : = P_1$ and let $Q_2 := \k^*\setminus P_1$. 
In the following proof we will define sets $S_0, \dots,S_{n+1}$ of arcs in $(\mathrm{S},\mathrm{M})$ for some
$n\in \mathbb{N}$, all containing $T$, from which will will construct $s$. 

We fix the following notation:  for each $0 \leq i\leq n+1$, 
let 
\[
s_i := \begin{cases} S_i & \text{if } t = T; \\ (S_i, Q_1, Q_2) & \text{if } t=(T, P_1, P_2)\end{cases}
\]

\noindent where the sets $S_i$ are defined as follows.  Let $S_0 := \{ \gamma_i \in \Gamma \mid \gamma_i \text{ does not cross any } \alpha\in T\} \cup T$.  Clearly $\ann{s_0} = \ann{t}$.  Let $a_1, \dots, a_n$ be the arrows in $\ann{t}$ that are not crossed in a 3-cycle by any arc in $T$ and such that $s(a_i)$, $t(a_i)$ are not contained in $\ann{t}$.  Then, for $1\leq i\leq n$, let $S_i := S_0 \cup \{\alpha_1,\dots,\alpha_i\}$ where $\alpha_i$ is an arc such that $\alpha_i$ crosses $a_i$ in a 3-cycle and $\ann{s_i} = \ann{s_{i-1}}$.  Each $\alpha_i$ exists by Lemma \ref{Lem: add arcs}.

Let $S:= S_{n+1}$ be any completion of $S_n$ to a full asymptotic triangulation of 
$(\mathrm{S},\mathrm{M})$ and let $s:= s_{n+1}$. 
Clearly $\mathcal{N}_t \subseteq \mathcal{N}_s$. 
We finish the proof by arguing that $\ann{s} = \ann{t}$. 
Since $M(t)$ is a direct summand of $M(s)$, we have that $\ann{s} \subseteq \ann{t}$. 
By Lemma \ref{Lem: paths arrows}, in order to show that $\ann{t} \subseteq \ann{s}$, 
it is sufficient to show that every vertex and every arrow contained in $\ann{t}$ is 
also contained in $\ann{s}$.  

Suppose $e_i$ is a vertex in $\ann{t}$.  Then either 
$\gamma_i \in T\cap\Gamma \subseteq S\cap\Gamma$ or $\gamma_i$ is not crossed 
by any arc in $T$. 
In the latter case we have that $\gamma_i \in S_0 \subseteq S\cap\Gamma$. 
By Lemma \ref{Lem: vert-trans}, both cases imply that $e_i \in \ann{s}$.

Suppose $a$ is an arrow contained in $\ann{t}$.  By the above argument, we may 
assume that both $e_{s(a)}$ and $e_{t(a)}$ are not in $\ann{t}$.  Then either $a$ is 
crossed in a 3-cycle by some $\alpha \in T \subseteq S$ or $a =a_j$ for some 
$1\leq j\leq n$.  In the latter case, we have that $a$ is crossed in a 3-cycle by 
$\alpha_j\in S_j\subseteq S$.  By Lemma \ref{Lem: vert-trans}, we have that 
$a\in \ann{s}$.

\end{proof}

\noindent  By Remark \ref{Rem: parametrising sets}, the following theorem yields the Main Theorem of the introduction.

\begin{theorem}\label{Thm: main theorem}
Every cosilting $A(\Gamma)$-module $C$ is equivalent to a module $M(t)$ where $t$ is a uniquely determined asymptotic triangulation.
\end{theorem}
\begin{proof}
Suppose $C$ is a cosilting module then $C$ is a cotilting module over $A(\Gamma)/\ann{C}$ and by Theorem \ref{Thm: cotilt vs rigid}, we have that $\mathcal{N}_C$ is a maximal rigid system in $\Mod{A(\Gamma)/\ann{C}}$.  Then, by Proposition \ref{Prop: cosilt are partial}, there is a partial asymptotic triangulation such that $\mathcal{N}_C = \mathcal{N}_t$ and so $\ann{C} = \ann{t}$.  By Lemma \ref{Lem: complete}, there exists an asymptotic triangulation $s$ such that $\ann{s} = \ann{t}$ and $\mathcal{N}_C =\mathcal{N}_t \subseteq \mathcal{N}_s$.  By Proposition \ref{Prop: asymp is cosilting}, we have that $M(s)$ is cosilting and so $\mathcal{N}_s$ is a maximal rigid system in $\Mod{A(\Gamma)/\ann{C}}$ such that $\Prod{\mathcal{N}_C} \subseteq \Prod{\mathcal{N}_s}$.  Therefore, by the maximality of $\mathcal{N}_C$, we have that $\Prod{\mathcal{N}_C} = \Prod{\mathcal{N}_s}$ and so $C$ is equivalent to $M(s)$.  

Suppose there is some other asymptotic triangulation $s$ such that $M(s)$ is equivalent to $C$.  Then $\ann{s} = \ann{t}$ and $\Prod{\mathcal{N}_s} = \Prod{\mathcal{N}_t}$. 
As $M(s) \in \Prod{\mathcal{N}_t}$, there exists a module $X$ such that 
$M(s)\oplus X \cong \prod_{i\in I} M_i$ with $M_i \in \mathcal{N}_t$. 
Without loss of generality, assume that a copy of every module in $\mathcal{N}_t$ 
occurs as a summand of $X\oplus M(s)$. By Lemma \ref{Lem: self-orth}, we have 
that $\Ext^1_{A(\Gamma)/\ann{C}}(\prod_{i\in I} M_i, \prod_{i\in I} M_i) = 0$ and 
hence, for every $N\in \mathcal{N}_t$, we have 
$\Ext^1_{A(\Gamma)/\ann{C}}(N, M) = 0 = \Ext^1_{A(\Gamma)/\ann{C}}(M, N)$ for all 
$M \in \mathcal{N}_s$.  Let $\delta$ be the arc corresponding to $N$ and first suppose 
$N \cong M(\delta)$.  As we argued in Proposition \ref{Prop: asym max rig}, it follows 
that, for every arc $\alpha\in S$, if $\delta$ and $\alpha$ cross, then they cross in 
a 3-cycle or in and arrow $a\in \ann{s} = \ann{t}$.  If $\alpha$ and $\delta$ 
cross in a 3-cycle, then $\ann{s} \neq \ann{t}$, by Lemma \ref{Lem: vert-trans}, 
which is a contradiction so $\alpha$ and $\delta$ can not cross in a 3-cycle. 

If $\alpha$ and $\delta$ cross in an arrow $a\in \ann{t}=\ann{s}$, then there are arcs 
$\varepsilon, \varepsilon'$ (with $\varepsilon \in S$ and $\varepsilon'\in T$) that 
cross $a$ in a 3-cycle, by Lemma~\ref{Lem: vert-trans}. 
But then $\varepsilon$ crosses $\alpha$ and $\varepsilon'$ crosses $\delta$, which 
is not possible. 
Thus $\delta$ does not cross any arc $\alpha \in S$ and so it must be that $\delta$ is in 
$S$.  If $N$ is a band module then the above argument implies that $s$ is 
strictly asymptotic, since otherwise there would be a bridging arc that crosses 
$\beta$.  Again $\Ext$-orthogonality implies $P_1 = Q_1$ and $P_2 = Q_2$.  
We may argue in the same, exchanging the roles of $s$ and $t$.  Therefore we have 
shown $t=s$.
\end{proof}

\begin{remark}\label{Rem: recovery}
From this theorem we recover the classification of cotilting modules over $\k \tilde{A_n}$ given in \cite{BuanKrause} by choosing $\Gamma$ to consist only of bridging arcs.  The cotilting modules are then given by the asymptotic triangulations $t$ such that $\Gamma \cap T = \emptyset$.  

Such a choice of $\Gamma$ also yields the description of the torsion classes in a tube category given in \cite{BBM}.  To see this, one should note that the arcs involving only one boundary component gives rise to a geometric model of one of the two (possibly non-homogeneous) tubes in $\mod{A(\Gamma)}$ that contain string modules. \end{remark}

%%%%%%%

\begin{appendices}
\section{Extensions and asymptotic arcs.}\label{App: Ext}

Throughout this appendix we will follow the definitions and notation set up in Section \ref{Sec: cluster-tilted}.  In particular, we fix $S$ to be the annulus and $\mathrm{M}$ to be a set of marked points in the boundary of $\mathrm{S}$.  We will then consider the $\k$-algebra $A(\Gamma)$ (defined as in Subsection \ref{Sec: type A}) where $\Gamma$ is a triangulation of $(\mathrm{S},\mathrm{M})$.  As in Section \ref{Sec: cluster-tilted}, we assume that $\k$ is algebraically closed.  

The main result will be the following.

\begin{theorem}\label{Thm: extensions app}
Let $M$ and $N$ be string or band modules in $\Mod{A(\Gamma)}$ and let $\alpha$ and $\delta$ be their respective associated arcs. If $M$ and $N$ are not both band modules, then the following statements are equivalent. \begin{enumerate}
\item The $\k$-vector spaces $\Ext^1_{A(\Gamma)}(M,N)$ and $\Ext^1_{A(\Gamma)}(N,M)$ are both zero.
\item The arcs $\alpha$ and $\delta$ do not cross outside of 3-cycles.
\end{enumerate}
\end{theorem}

We will approach this theorem in stages.  First we will describe explicitly the form of extensions 
between indecomposable pure-injective $A(\Gamma)$-modules.  Our approach is similar to \cite{ALP, CPS, CPS2}. 
We will then connect these extensions to crossings of arcs in $(\mathrm{S},\mathrm{M})$ (outside of 3-cycles). 
In this first section our approach is similar to \cite{CS}.

%%%%%%%
%
\subsection{Extensions in $\Mod{A(\Gamma)}$ and morphisms in the homotopy category.}\label{Sec: string and band complexes}

Let $M$ and $N$ be string or band modules over $A(\Gamma)$ and let $\alpha$ and $\delta$ be their respective associated arcs. Assume that $M$ and $N$ are not both band modules.  We will identify $M$ and $N$ with their projective resolutions, denoted $P_M^\bullet$ and $P_N^\bullet$ respectively, and make use of the following well-known isomorphisms (see, for example, \cite[Cor.~10.4.7, Cor.~10.7.5]{Weibel}). \begin{align*} \Ext^1_{A(\Gamma)}(M,N) & \cong \Hom_{\mathrm{D}(A(\Gamma))}(M[0], N[1])\\
& \cong \Hom_{\mathrm{K}(A(\Gamma))}(P_M^\bullet, P_N^\bullet[1])   .\end{align*}
where $\mathrm{D}(A(\Gamma))$ is the unbounded derived category of $\Mod{A(\Gamma)}$ and $\mathrm{K}(A(\Gamma))$ is the homotopy category of (co)chain complexes in $\Mod{A(\Gamma)}$.  

In \cite{BM} it is shown that every indecomposable complex in $\mathrm{D^b}(\mod{A(\Gamma)})$ is isomorphic to a complex of finite-dimensional projective modules that is either a string complex or a band complex associated to a homotopy string or a homotopy band.  We will see that $P_N^\bullet$ and $P_M^\bullet$ also have the form of a string or a band complex, however the projective modules may be infinite-dimensional.  

Next we define what we mean by a homotopy string in this setting.  Observe that any finite string may be uniquely expressed as a sequence $p_nq_n^{-1}\dots p_1q_1^{-1}p_0q_0^{-1}$ where $p_i$ and $q_j$ are paths of length $\geq 1$ for $0\leq i< n$ and $0 < j \leq n$ and $p_n$ and $q_0$ are possibly trivial paths.  Similarly, any $\mathbb{N}$-string may be uniquely expressed as a sequence $\dots p_1q_1^{-1}p_0q_0^{-1}$.  In the following definition, all strings will be represented in this way.

\begin{definition}\label{Def: homotopy string}
\begin{enumerate}
\item A sequence of arrows $\dots a_3a_2a_1a_0$ is called a \textbf{antipath} if $a_{i+1}a_i \in I(\Gamma)$ for all $i\geq 0$.

\item Let $u_\alpha = \dots p_1q_1^{-1}p_0q_0^{-1}$ be an $\mathbb{N}$-string (respectively let $u_\alpha = p_nq_n^{-1}\dots p_1q_1^{-1}p_0q_0^{-1}$ be a finite string).  \begin{enumerate}
\item If there exists an arrow $a_0$ such that $u_\alpha a_0^{-1}$ is a string, then define $h(u_\alpha)$ to be the sequence $u_\alpha a_0^{-1}a_1^{-1}a_2^{-1}\dots$ where $\dots a_2a_1a_0$ is the longest possible antipath starting at $a_0$.
\item If there is no such arrow $a_0$, then define $h(u_\alpha)$ to be the sequence $\dots p_1q_1^{-1}p_0$ (respectively $p_nq_n^{-1}\dots p_1q_1^{-1}p_0$).
\end{enumerate}

\item Let $u_\alpha$ be an $\mathbb{N}$-string, then the \textbf{homotopy string $\tau_\alpha$ associated to $u_\alpha$} is defined to be $h(u_\alpha)$.

\item Let $u_\alpha$ be a finite string, then the \textbf{homotopy string $\tau_\alpha$ associated to $u_\alpha$} is defined to be $h(h(u_\alpha^{-1})^{-1})$.
\end{enumerate}
\end{definition}

From this data we can define string complexes, which are essentially infinite-dimensional versions of the string complexes defined in \cite{BM}.  For a vertex $i$ in $Q$, we will denote the corresponding indecomposable projective module by $P(i)$.  Moreover, we will make use of the $\k$-basis of $\Hom_{A(\Gamma)}(P(i), P(j))$ given by paths $p \colon j \to i$.  We will denote these basis elements by the paths $p$.

As with string infinite-dimensional string modules, we must consider the expanding and contracting cases separately.

\begin{definition}
\begin{enumerate}
\item Suppose $u_\alpha = \dots p_1q_1^{-1}p_0q_0^{-1}$ is a contracting $\mathbb{N}$-string and let $\tau_\alpha$ be the homotopy string associated to $u_\alpha$.  If $h(u_\alpha)$ fell into case \textbf{(b)}, then we take $P(t(a_k))$, $a_0q_0$ and $a_{k+1}$ to be zero for $k \geq 0$.  

Define the \textbf{string complex} $P_\alpha^\bullet$ as follows. \[ P_\alpha^0 := \bigoplus_{i\geq 0} P(s(p_i)), \hspace{7pt} P_\alpha^{-1} := P(t(a_0)) \oplus (\bigoplus_{i\geq0} P(t(p_i))), \hspace{7pt} P_\alpha^{-k} := P(t(a_{k-1})), \hspace{7pt} P_\alpha^{-l} := 0 \] for $k >1$ and $l < 0$.  Also let \[ d_\alpha^{-1} := (A_{rs})_{r>0, s>0}, \hspace{8pt} d_\alpha^{-k} := a_{k-1}, \hspace{8pt} d_\alpha^{-l} := 0 \] for $k>1$ and $l< 1$ where \[A_{rs}:= \begin{cases} a_0q_0 & \text{if } r=1, s=1, \\
q_i & \text{if } r = i+1, s =i+1 \text{ and } i>0, \\
p_i & \text{if } r = i+1, s = i+2 \text{ and } i\geq0,\\
0 & \text{otherwise}.
\end{cases}\]

\item Suppose $u_\alpha = \dots p_1q_1^{-1}p_0q_0^{-1}$ is a expanding $\mathbb{N}$-string and $\tau_\alpha$ is the homotopy string associated to $u_\alpha$.  Define the \textbf{string complex} $P_\alpha^\bullet$ in the same way as in the contracting case but replacing any direct sums with direct products. 

\item Suppose $u_\alpha$ is a finite string and let $\tau_\alpha$ be the associated homotopy string.  If $h(u_\alpha^{-1})$ fell into case \textbf{(a)}, then denote the antipath by $\dots a_2a_1a_0$.  If $h(h(u_\alpha^{-1})^{-1})$ fell into case \textbf{(a)}, then denote the antipath by $\dots c_2c_1c_0$. If $h(u_\alpha^{-1})$ fell into case \textbf{(b)}, then we take $P(t(a_k))$, $a_0p_0$ and $a_{k+1}$ to be zero for $k \geq 0$.  If $h(h(u_\alpha^{-1})^{-1})$ fell into case \textbf{(b)}, then we take $P(t(c_k))$, $c_0p_n$ and $c_{k+1}$ to be zero for $k \geq 0$.

Define the \textbf{string complex} $P_\alpha^\bullet$ as follows. \[ P_\alpha^0 := \bigoplus_{i=0}^n P(s(p_i)), \hspace{10pt} P_\alpha^{-1} := P(t(a_0)) \oplus (\bigoplus_{i=0}^n P(t(p_i))) \oplus P(t(c_0)),\] \[ P_\alpha^{-k} := P(t(a_{k-1})) \oplus P(t(c_{k-1})), \hspace{10pt} P_\alpha^{-l} := 0 \] for $k >1$ and $l < 0$.  Also let \[ d_\alpha^{-1} := (A_{rs})_{0<r\leq n+1,\: 0<s\leq n+2}, \hspace{8pt} d_\alpha^{-k} := \left(\begin{smallmatrix} a_{k-1} & 0 \\ 0 & c_{k-1} \end{smallmatrix}\right), \hspace{8pt} d_\alpha^{-l} := 0 \] for $k>1$ and $l< 1$ where \[A_{rs}:= \begin{cases} a_0q_0 & \text{if } r=1, s=1, \\
q_i & \text{if } r = i+1, s =i+1 \text{ and } 0<i\leq n-1, \\
p_i & \text{if } r = i+1, s = i+2 \text{ and } 0<i\leq n-1,\\
c_0 p_n & \text{if } r=n+1, s=n+2,\\
0 & \text{otherwise}.
\end{cases}\]
\end{enumerate}
\end{definition}

\noindent We can view the definition of $P_\alpha^\bullet$ for an $\mathbb{N}$-string schematically as follows.
\[ \xymatrix@R=5mm@C=14mm{ & & P(t(p_1)) \ar[r]_{q_2}^{\vdots} \ar[dr]_<<<<<<<{p_1} & P(s(p_2)) \\
& & P(t(p_0)) \ar[r]^{q_1} \ar[dr]^{p_0} & P(s(p_1)) \\
\dots \ar@{.>}[r]^{a_2} & P(t(a_1)) \ar@{.>}[r]^{a_1} & P(t(a_0)) \ar@{.>}[r]^{a_0q_0} & P(s(q_0))
}\]

\noindent Similarly, we can view the definition of $P_\alpha^\bullet$ for a finite string schematically as follows.

\[ \xymatrix@R=5mm@C=14mm{ \dots \ar@{.>}[r]^{c_2} & P(t(c_2)) \ar@{.>}[r]^{c_1} & P(t(c_0)) \ar@{.>}[dr]^{c_0p_n} &  \\
& & P(t(p_{n-1}))  \ar[r]^{q_n}  & P(s(p_n)) \\
& & \ar@{}[ur]|\vdots P(t(p_1)) \ar[dr]^{p_1} & \\
& & P(t(p_0)) \ar[r]^{q_1} \ar[dr]^{p_0} & P(s(p_1)) \\
\dots \ar@{.>}[r]^{a_2} & P(t(a_1)) \ar@{.>}[r]^{a_1} & P(t(a_0)) \ar@{.>}[r]^{a_0q_0} & P(s(q_0))
}\]

As in \cite{ALP, CPS, CPS2}, we will make use of unfolded diagrams i.e.~we represent the complexes $P_\alpha^\bullet$ in the following ways.

\[
\xymatrix@R=1mm@C=13mm{
 & 0 & -1 & 0 & -1 & 0 & -1 & -2 & \\
\dots \ar[r]^{ p_2} & \bullet  & \bullet \ar[r]^{p_1} \ar[l]_{q_2} & \bullet & \bullet \ar[l]_{q_1} \ar[r]^{p_0} & \bullet & \bullet \ar@{.>}[l]_{a_0q_0} & \bullet \ar@{.>}[l]_{a_1} & \dots \ar@{.>}[l]_{a_2}
}
\]

\[
\xymatrix@R=1mm@C=13mm{
-2 & -1 & 0 & -1 & 0 &  -1 & 0 & -1 & -2 \\
\dots \ar@{.>}[r]^{ c_1} & \bullet \ar@{.>}[r]^{c_0p_n} & \bullet & \bullet \ar[l]_{q_n} \ar[r]^{p_{n-1}}& \bullet \ar@{}[r]|{\dots\dots} & \bullet \ar[r]^{p_0} &  \bullet & \bullet \ar@{.>}[l]_{a_0q_0} & \dots \ar@{.>}[l]_{a_1}
}
\]

Recall that $u_\beta$ is the $\mathbb{Z}$-string associated to the closed arc $\beta$.  We will refer to $u_\beta$ as the \textbf{homotopy string associate to $\beta$} and we will denote $u_\beta$ by $\tau_\beta$ when we are considering it as a homotopy string.  Next we will use a string complex $P_{u_\beta}^\bullet$ to define generalised band complexes.    Let $R := \k[\Phi, \Phi^{-1}]$ and recall the classification of indecomposable pure-injective modules given in Remark \ref{Rem: K[T,T^{1}]}.    

\begin{definition}
Let $u_\beta = \dots p_1q_1^{-1}p_0q_0^{-1}p_{-1}q_{-1}^{-1}\dots$ be the periodic $\mathbb{Z}$-string. Define the complex $P_{u_\beta}^\bullet$ to be \[ P_{u_\beta}^{-1} := \bigoplus_{i\in \mathbb{Z}} P(t(p_i)), \hspace{15pt} P_{u_\beta}^0 := \bigoplus_{i\in \mathbb{Z}} P(s(p_i)), \hspace{15pt} P_{u_\beta}^k := 0\] for all $k \neq 0,-1$.  Also let 
$d_{u_\beta}^{-1}:= (A_{rs})_{r,s\in \mathbb{Z}}$ and $d_{u_\beta}^k := 0$ for all $k \neq -1$ where \[A_{rs}:= \begin{cases} q_i & \text{if } r = i+1, s =i+1, \\
p_i & \text{if } r = i+1, s = i+2,\\
0 & \text{otherwise}.
\end{cases}\]

We can view $P_{u_\beta}^\bullet$ as a complex of $(A(\Gamma), R)$-bimodules (see, for example, \cite[Def.~1.3.45]{raph}).  For each $\lambda \in \k^*$ and each $n\in \mathbb{N}\cup\{\infty, -\infty\}$, define the \textbf{band complex} $B_{\lambda,n}^\bullet$ to be the complex obtained by applying $-\otimes_RN_{\lambda,n}$ to $P_{u_\beta}^\bullet$. 
\end{definition}

\begin{remark}\begin{enumerate}\item The complex $B_{\lambda,n}^\bullet$ is a complex of projective $A(\Gamma)$-modules (see, for example, \cite[Lem.~1.3.47]{raph}). 

\item Since $\k$ is algebraically closed, the complex $B_{\lambda,n}^\bullet$ is a band complex as defined in \cite{BM} when $n\in \mathbb{N}$. 
\end{enumerate} 
\end{remark}

We also associate an unfolded diagram to the complexes $B_{\lambda,1}^\bullet$.  The unfolded diagrams of the other band complexes will be considered later (see ...).    Since $u_\beta$ is periodic, we have that there exist non-trivial paths $m_0,\dots m_t$ and $n_0,\dots n_t$ such that $u_\beta = \dots m_t^{-1}n_0m_0^{-1}n_1\dots n_tm_t^{-1}n_0\dots$.  The complex $B_{\lambda,1}^\bullet$ is then represented as follows.

\[ \xymatrix@R=1mm@C=13mm{
 &0&-1&0&-1&0 & -1 & 0 & \\
\dots \ar[r]^{n_{t-2}}  & \bullet & \bullet \ar[r]^{n_{t-1}} \ar[l]_{m_{t-1}} & \bullet  & \bullet \ar[r]^{n_0} \ar[l]_{\lambda m_t} & \bullet & \bullet \ar[l]_{m_0} \ar[r]^{n_1} & \bullet & \dots \ar[l]_{m_1}
}\]

The aim of the preceding definitions was to identify the projective resolutions of the indecomposable pure-injective modules (excluding $G$).  The following lemma confirms that this is the case.

\begin{lemma}\label{Lem: proj res}
\begin{enumerate}
\item Let $u_\alpha$ be a finite or an $\mathbb{N}$-string and let $\alpha$ be the associated homotopy string.  The complex $P_\alpha^\bullet$ is quasi-isomorphic to $M(\alpha)[0]$.

\item Let $\lambda \in \k^*$ and let $n\in \mathbb{N}\cup\{\infty, -\infty\}$.  The complex $B_{\lambda,n}$ is quasi-isomorphic to $M(\lambda, n)[0]$.
\end{enumerate}
\end{lemma}
\begin{proof}
\textbf{(1)} Compute cohomology as in \cite[Cor.~2.12]{CPS}.

\textbf{(2)} By computing cohomology of $P_{u_\beta}^\bullet$, we have that the following sequence is a projective resolution of $M^\oplus(\beta)$.
\[  0 \to P_{u_\beta}^{-1} \to P_{u_\beta}^{0} \to M^\oplus(\beta)\to 0.\]  Moreover, the functor $-\otimes_RN_{\lambda,n}$ is right exact, so \[ 0 \to B_{\lambda,n}^{-1} \to B_{\lambda,n}^0 \to M(\lambda, n) \to 0\] is right exact.  The module $M^\oplus(\beta)$ is free as an $R$-module, so $\mathrm{Tor}^R_1(M^\oplus(\beta), N_{\lambda,n}) = 0$, therefore the sequence is exact.  We have shown that this sequence is a projective resolution of $M(\lambda, n)$ and so in particular, the complex $B_{\lambda,n}^\bullet$ is quasi-isomorphic to $M(\lambda,n)[0]$.
\end{proof}

Let $M$ and $N$ be string modules or band modules with $n=1$, such that $M$ and $N$ are not both band modules.  Let $\alpha$ and $\delta$ be the associated arcs in $(\mathrm{S},\mathrm{M})$.  Also, denote the projective resolutions of $M$ and $N$ by $P_M^\bullet$ and $P_N^\bullet$ respectively; note that these are given in Lemma \ref{Lem: proj res}.  Next we will establish when there exists a non-zero morphism $P_M^\bullet \to P_N^\bullet[1]$ by comparing the homotopy strings $\tau_\alpha$ and $\tau_\delta$.  

We refer to \cite[Sec.~3]{ALP} for the definition of \textbf{(singleton) single}, \textbf{double}, \textbf{graph} and \textbf{quasi-graph maps} between $P_M^\bullet$ and $P_N^\bullet[1]$.  These definitions extend easily to include projective resolutions of infinite-dimensional string modules.

\begin{remark}
In \cite{ALP} a quasi-graph map ``from $P_M^\bullet$ to $P_N^\bullet$" determines a non-trival homotopy class of morphisms $P_M^\bullet \to P_N^\bullet[1]$ that does not contain zero.  We use the notation $P_M^\bullet \rightsquigarrow P_N^\bullet[1]$ to indicate that there is such a quasi-graph map.
\end{remark}

For there to be a graph map $P_M^\bullet \to P_N^\bullet[1]$ or quasi-graph map $P_M^\bullet \rightsquigarrow P_N^\bullet[1]$, there must be a (maximal) common substring shared by $\tau_\alpha$ and $\tau_\delta$; we stress that this common substring $\rho$ must be finite unless $\rho$ contains an infinite length antipath.  We will call quasi-graph maps $P_M^\bullet \rightsquigarrow P_N^\bullet[1]$, graph maps, singleton single and singleton single maps $P_M^\bullet \to P_N^\bullet[1]$ the \textbf{standard maps} from $P_M^\bullet$ to $P_N^\bullet[1]$.

If we apply the main theorem of \cite{ALP} to our setting, then we obtain the following proposition.  Since the proof of Proposition \ref{Prop: standard iff non-zero} will extend the arguments given in \cite{ALP}, we include a step-by-step sketch of the proof in order to refer to it later.

\begin{proposition}[{\cite[Thm.~3.15]{ALP}}]\label{Prop: f.d. basis}
Let $M$ and $N$ be finite-dimensional string modules or band modules with $n=1$, such that $M$ and $N$ are not both band modules, and let $P_M^\bullet$ and $P_N^\bullet$ be the projective resolutions of $M$ and $N$ respectively.  The standard maps $P_M^\bullet \to P_N^\bullet[1]$ form a $\k$-basis for $\Hom_{\mathrm{K}(A(\Gamma))}(P_M^\bullet, P_N^\bullet[1])$.
\end{proposition}
\begin{proof}
\begin{description}
\item[Step 1:] First it is shown that the graph, single and double maps form a basis for $\Hom_{\mathrm{C}(A(\Gamma))}(P_M^\bullet, P_N^\bullet[1])$ where $\mathrm{C}(A(\Gamma))$ is the category of chain complexes in $\Mod{A(\Gamma)}$.  This is done by showing that the morphisms are linearly independent (\cite[Lem.~4.3, Cor.~4.4]{ALP}) and that any non-zero component of a morphism $f \colon P_M^\bullet \to P_N^\bullet[1]$ determines a graph, single or double map that is a summand of $f$ (\cite[Lem.~4.2]{ALP}).  Since $P_M^\bullet$ and $P_N^\bullet$ are bounded complexes of finite-dimensional modules, this is enough to prove that the graph, single and double maps also span $\Hom_{\mathrm{C}(A(\Gamma))}(P_M^\bullet, P_N^\bullet[1])$ (see \cite[Prop.~4.1]{ALP}).

\item[Step 2:] In order to consider morphisms in $\mathrm{K}(A(\Gamma))$, the next step is to consider the class $\mathcal{H}(f)$ of single, double or graph maps $g$ such that $f$ is homotopy equivalent to $\mu g$ for some $0 \neq \mu \in \k$ (see \cite[Def.~4.6]{ALP}).  It is shown in \cite[Prop.~4.8]{ALP} that if $f\colon  P_M^\bullet \to P_N^\bullet[1]$ is a single or double map such that $\mathcal{H}(f) \neq \{f\}$, then there is a quasi-graph map $P_M^\bullet \rightsquigarrow P_N^\bullet[1]$.  The idea of the proof is to produce an algorithm that is determined by a given single or double map and produces the corresponding quasi-graph map.  Since $P_M^\bullet$ and $P_N^\bullet$ are bounded complexes of finite-dimensional modules, the algorithm terminates after finitely many steps.

\item[Step 3:]  The final step is to show that if $f$ is a graph map or a singleton single or double map, then $\mathcal{H}(f) = \{f\}$.  This is done in \cite[Lem.~4.12, Lem.~4.12]{ALP}.
\end{description}
\end{proof}

Next we allow $M$ and $N$ to be infinite-dimensional string modules.  The standard maps no longer necessarily form a basis, however we do obtain the following proposition.

\begin{proposition}\label{Prop: standard iff non-zero}
Let $M$ and $N$ be string modules or band modules with $n=1$, such that $M$ and $N$ are not both band modules, and let $P_M^\bullet$ and $P_N^\bullet$ be the projective resolutions of $M$ and $N$ respectively.  The following statements are equivalent.
\begin{enumerate}
\item The $\k$-vector space $\Hom_{\mathrm{K}(A(\Gamma))}(P_M^\bullet, \: P_N^\bullet[1])$ is zero.
\item There exists a standard map $P_M^\bullet \to P_N^\bullet[1]$.
\end{enumerate}
\end{proposition}
\begin{proof} 
In the following proof we will denote the arcs associated to $M$ and $N$ by $\alpha$ and $\delta$ respectively.

The idea of the proof is to the extend to arguments of \cite{ALP} to include projective resolutions of infinite-dimensional string modules.  The proof is divided into three steps to reflect the three steps used to prove \cite[Thm.~3.15]{ALP} (see Proposition \ref{Prop: f.d. basis}).  For each step it is necessary to apply the corresponding arguments in \cite{ALP} and then extend them in the way that is described below.
\begin{description}

\item[Step 1:]  First we show that the Proposition \ref{Prop: standard iff non-zero} holds in the category $\mathrm{C}(A(\Gamma))$. The arguments used to obtain \cite[Lem.~4.2, Lem.~4.3, Cor.~4.4]{ALP} may also be applied to projective resolutions of infinite-dimensional string modules.  This yields that the graph, single and double maps $P_M^\bullet \to P_N^\bullet[1]$ are linearly independent in $\mathrm{C}(A(\Gamma))$.  By applying the same argument as the proof of \cite[Prop.~4.1]{ALP}, we achieve Step 1 by ruling out the case where $g$ is a `graph map' induced by an infinite common substring that is not an antipath.  Such a situation cannot arise in our setup since the degrees of any infinite common substring would not be compatible.  Thus the proposition holds in the category $\mathrm{C}(A(\Gamma))$. 

\item[Step 2:] If neither of $M$ and $N$ are finite-dimensional string modules and $f \colon P_M^\bullet \to P_N^\bullet[1]$ is a single or double map such that $\mathcal{H}(f) \neq \{f\}$, then the algorithm described in \cite[Prop.~4.8]{ALP} works in the same way but may not terminate.  If the algorithm stops, then there is a quasi-graph map $g\colon P_N^\bullet \rightsquigarrow P_M^\bullet[1]$.   

We show that if the algorithm does not terminate, then the morphism $f$ is null-homotopic. By the same argument as \cite[Lem.~3.5]{CPS}, the maximal common substring produced by the algorithm is supported in cohomological degrees $-1$ and $0$ only.  Then the algorithm described in the proof of \cite[Prop.~4.8]{ALP} yields the following two extra cases: 

\noindent \textbf{Case $1$:}
\[ \xymatrix@R=1mm@C=11mm{
  & 0 & -1 & 0 & -1 & 0 & -1 &0 & \\
  \ar@{}[r]|{\cdots\cdots} & \bullet \ar@{}[ddl]|\dots \ar@{=}[dd] & \bullet \ar@{=}[dd] \ar[l]_{q_{n+2}} \ar[r]^{p_{n+1}}  & \bullet \ar@{=}[dd] & \bullet \ar@{=}[dd]\ar[l]_{q_{n+1}} \ar[r]^{p_n} & \bullet \ar@{=}[dd] & \bullet \ar[l]_{q_n} \ar[r]^{p_{n-1}}  & \bullet  \ar@{}[r]|{\dots\dots}  &\\
&&&&&&&&\\
  \ar@{}[r]|{\cdots\cdots} & \bullet &  \bullet \ar[l]^{q'_{m+2}} \ar[r]_{p'_{m+1}}  & \bullet & \bullet \ar[l]^{q'_{m+1}} \ar[r]_{p'_{m}} & \bullet  & \bullet \ar@{.>}[l]^{q'_{m}} \ar@{.>}[r]_{p'_{m-1}} & \bullet \ar@{}[r]|{\dots\dots} & \\
& -1 & -2 & -1 & -2 & -1 & -2 & -1 &
}\] where the top unfolded diagram is the unfolded diagram of $\tau_\alpha$ and the bottom one is the unfolded diagram of $\tau_\alpha$.  Moreover, there is no path $p$ from $t(q_m')$ to $t(q_n)$ such that $q_n = pq_m'$.  This is only possible if either $\delta$ is an asymptotic arc with $m=0$ and $q_0'$ is trivial or there is a non-trivial path $q$ such that $qq_n = q_m'$.

\noindent \textbf{Case $2$:}
\[ \xymatrix@R=1mm@C=11mm{
  & 0 & -1 & 0 & -1 & 0 & -1 &0 & \\
  \ar@{}[r]|{\cdots\cdots} & \bullet \ar@{}[ddl]|\dots \ar@{=}[dd] & \bullet \ar[l]_{q_{n+1}} \ar[r]^{p_n} \ar@{=}[dd]  & \bullet \ar@{=}[dd]  & \bullet \ar@{=}[dd]\ar[l]_{q_n} \ar@{.>}[r]^{p_{n-1}} & \bullet  & \bullet \ar@{.>}[l]_{q_{n-1}} \ar@{.>}[r]^{p_{n-2}}  & \bullet \ar@{}[r]|{\dots\dots}  &\\
&&&&&&&&\\
  \ar@{}[r]|{\cdots\cdots} & \bullet &  \bullet \ar[l]^{q'_{m+1}} \ar[r]_{p'_m}  & \bullet & \bullet \ar[l]^{q'_m} \ar[r]_{p'_{m-1}} & \bullet  & \bullet \ar[l]^{q'_{m-1}} \ar[r]_{p'_{m-2}} & \bullet \ar@{}[r]|{\dots\dots} & \\
& -1 & -2 & -1 & -2 & -1 & -2 & -1 &
}\] where, again, the top unfolded diagram is the unfolded diagram of $\tau_\alpha$ and the bottom one is the unfolded diagram of $\tau_\delta$, and there is no path $p$ from $s(p_{m-1}')$ to $s(p_{n-1})$ such that $p_{n-1}p = p_{m-1}'$.  This is only possible if $\alpha$ is an asymptotic arc and $n=0$ or there exists a non-trivial path $q$ such that $p_{n-1} = p_{m-1}'q$

If $\alpha$ and $\delta$ are both expanding asymptotic arcs or both contracting asymptotic arcs, then both cases $1$ and $2$ yield classes of null-homotopic maps.

Suppose that $\alpha$ and $\delta$ are not both expanding asymptotic arcs nor both contracting asymptotic arcs.  We will show that we must have that $\alpha$ is a contracting asymptotic arc or $\delta$ is an expanding asymptotic arc; either scenario yields a class of null-homotopic maps.  Note that the maximal common substring of $\tau_\delta$ and $\tau_\alpha$ must be the maximal periodic part of $\tau_\alpha$ or $\tau_\delta$, otherwise they would be both contracting or both expanding asymptotic arcs.  We will deal with case $1$, case $2$ is similar.

First we assume we have case $1$ and suppose that $m\neq 0$.  The maximal common substring of $\tau_\delta$ and $\tau_\alpha$ is $w:= \dots q_{n+2}^{-1}p_{n+1}q_{n+1}^{-1}p_nq_n^{-1} = \dots (q_{m+2}')^{-1}p_{m+1}' (q_{m+1}')^{-1} p_m'q_n^{-1}$.  If $w$ is the maximal periodic part of $\tau_\alpha$, then it is also the maximal periodic part of $u_\alpha$ and $\alpha$ is a contracting asymptotic arc.  If $w$ is the maximal periodic part of $\tau_\delta$, then it is also the maximal periodic part of $u_\delta$ and $\delta$ is an expanding asymptotic arc.  Now suppose we are in case $1$ with $\delta$ an asymptotic arc and $m = 0$.  We cannot have that $u_\delta = \dots p_{1}'(q_{1}')^{-1}p_0'$ since, by Definition \ref{Def: homotopy string}, the existence of the last letter of $q_n$ would imply that $\tau_\delta \neq u_\delta$.  It must therefore be the case that $u_\delta = \dots p_{1}'(q_1')^{-1}p_0'(q_0')^{-1}$ where $q_0'$ is a path and there is no arrow $a_0$ such that $a_0q_0'$ is a string.  That is, we must have that $q_0' = rq_n$ for some (possibly trivial) path $r$.  It follows that the maximal common substring of $u_\alpha$ and $u_\delta$ is $w := \dots p_{1}'(q_{1}')^{-1}p_0'q_n^{-1} = \dots p_{n+1}q_{n+1}^{-1}p_nq_n^{-1}$.  If $r$ is non-trivial then we have already dealt with this case.  If $r$ is trivial, then $u_\delta$ is the maximal periodic substring of itself and must be expanding.

\item[Step 3:] The arguments given in \cite[Lem.~4.12, Lem.~4.12]{ALP} also yield that if $f$ is a graph map or a singleton single or double map, then $\mathcal{H}(f) = \{f\}$ if $M$ and/or $N$ is an infinite-dimensional string module.
\end{description}
\end{proof}

%%%%%%%
%
\subsection{Extensions in $\Mod{A(\Gamma)}$ and standard maps.}

Let $M$ and $N$ be string or band modules with $n=1$ over $A(\Gamma)$ and let $\alpha$ and $\delta$ be their respective associated arcs. Assume that $M$ and $N$ are not both band modules.  Let \[\Phi \colon \Hom_{\mathrm{K}(A(\Gamma))}(P_M^\bullet, P_N^\bullet[1]) \to \Ext_{A(\Gamma)}(M,N)\] be the isomorphism described at the beginning of the appendix.

In \cite{CPS}, the authors show that if $M$ and $N$ are finite-dimensional, then $\Phi$ induces a bijection between the standard maps in $\Hom_{\mathrm{K}(A(\Gamma))}(P_M^\bullet, P_N^\bullet[1])$ and the standard extensions in $\Ext_{A(\Gamma)}(M,N)$, which are defined in Proposition \ref{prop:list-extensions}.  The arguments used in \cite{CPS, CPS2} do not depend on $M$ and $N$ being finite-dimensional and so we directly obtain the following lemma.  

\begin{lemma}\label{Lem: standard maps and standard extensions}
Let $M$ and $N$ be string or band modules with $n=1$ over $A(\Gamma)$ and let $\alpha$ and $\delta$ be their respective associated arcs. Assume that $M$ and $N$ are not both band modules.  Then $\Phi$ induces a bijection as follows.
\[\left\{ \text{Standard maps in } \Hom_{\mathrm{K}(A(\Gamma))}(P_M^\bullet, P_N^\bullet[1]) \right\} \overset{\text{1-1}}{\longleftrightarrow} \left\{\text{Standard extensions in } \Ext^1_{A(\Gamma)}(M,N)\right\} \]

\noindent In particular, there exists a standard extension $0 \to N \to E \to M \to 0$ in $\Mod{A(\Gamma)}$ if and only if $\Ext^1_{A(\Gamma)}(M,N) \neq 0$.
\end{lemma} 

%%%%%%%
%
\subsection{Extensions in $\Mod{A(\Gamma)}$ and crossings of arcs.}

As before, let $M$ and $N$ be string or band modules with $n=1$ such that $M$ and $N$ are not both band modules.  We will denote the arc associated to $M$ by $\alpha$ and the arc associated to $N$ by $\delta$.  The next proposition extends some of the results of \cite{CS} to include infinite-dimensional string modules and band modules with $n=1$.  

The following observations are similar to the proof of \cite[Prop.~3.3]{CS}.  They essentially follow from the definitions.  There exists an overlap extension $0\to N \to E\to M \to 0$ if and only if we have to following local configuration. 

\[
\includegraphics[width=4cm]{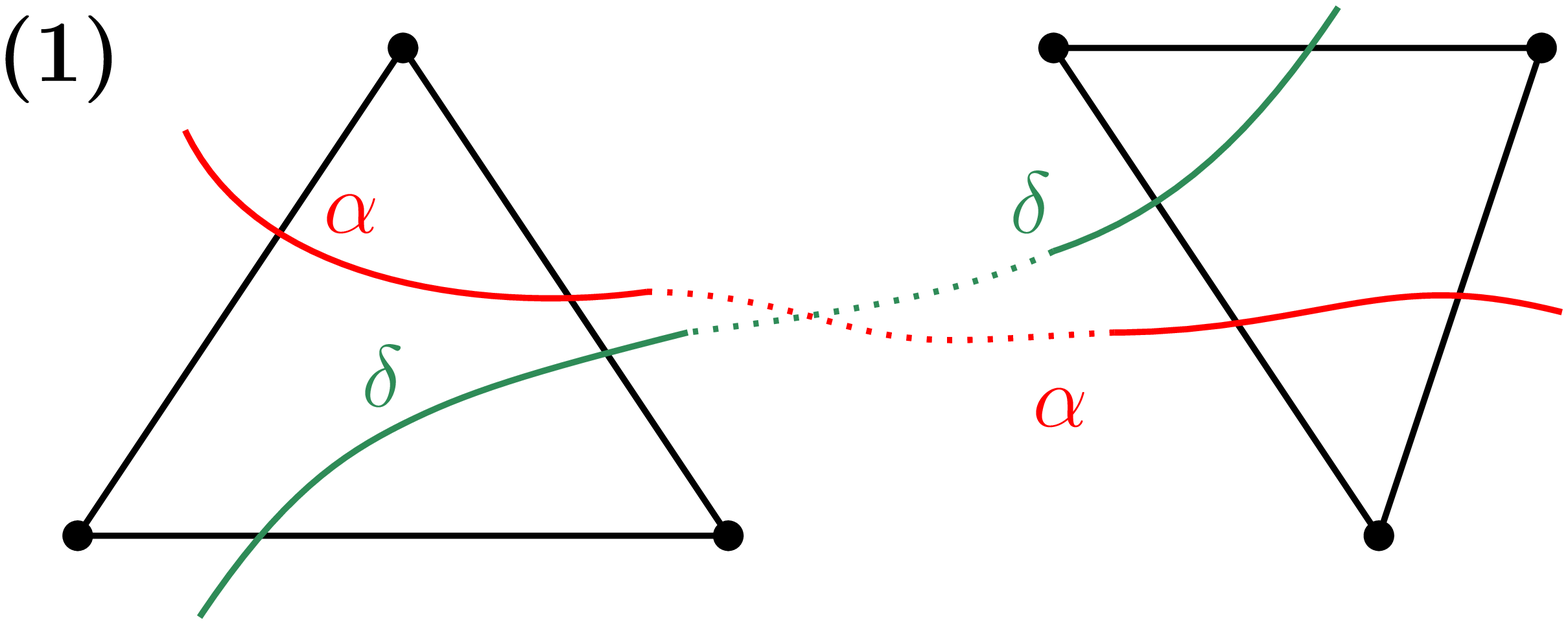}
\]

  There is an arrow extension $0 \to N \to E \to M \to 0$ if and only if we have the following local configuration.
  
  \[
\includegraphics[width=2cm]{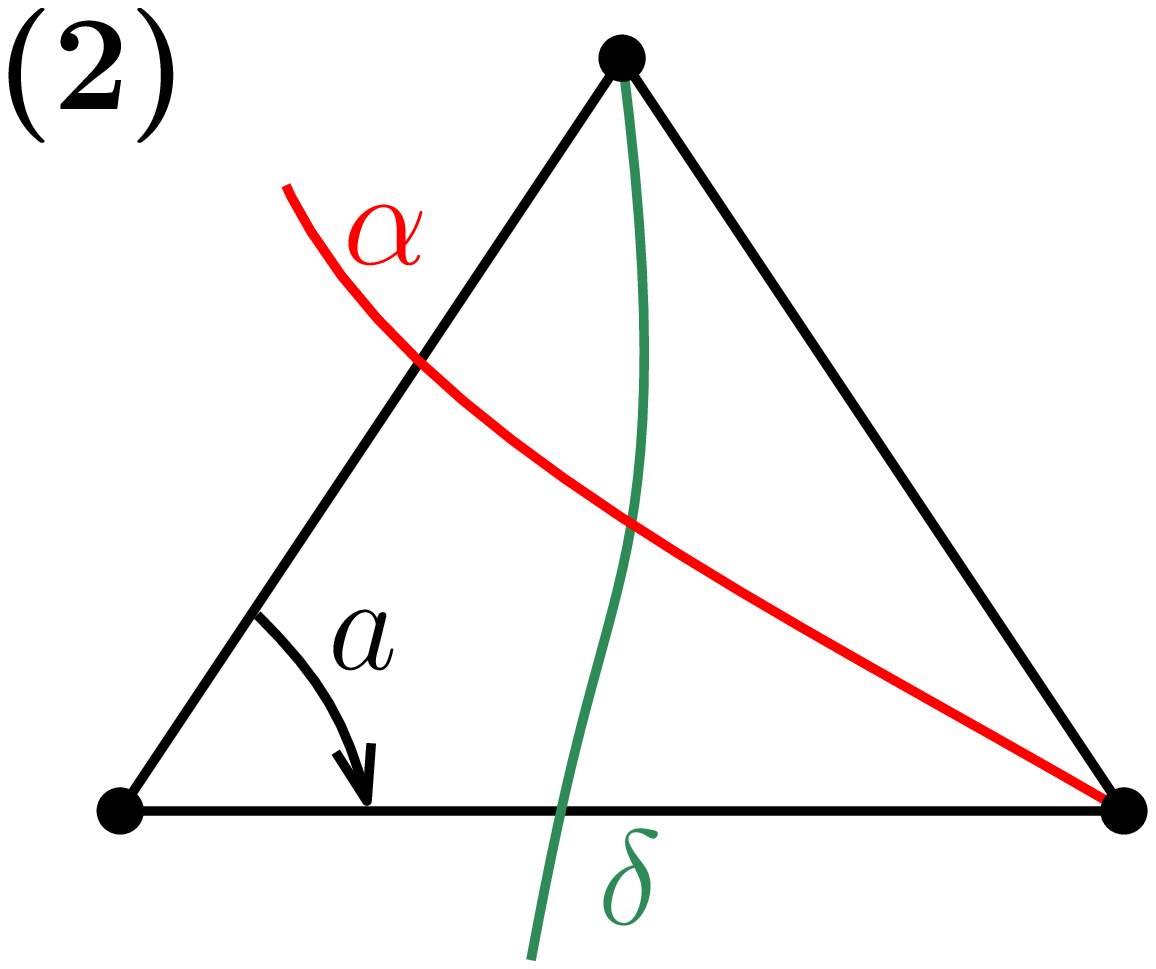}
\]

    The only remaining possible crossing between $\alpha$ and $\delta$ (up to homotopy and reordering of $\alpha$ and $\delta$) corresponds to $\alpha$ and $\delta$ crossing in a 3-cycle.

\begin{lemma}\label{Lem: main theorem for n=1}
Let $M$ and $N$ be string or band modules with $n=1$ such that $M$ and $N$ are not both band modules.  Denote the arcs associated to $M$ and $N$ by $\alpha$ and $\delta$ respectively.  The following statements are equivalent. \begin{enumerate}
\item The $\k$-vector spaces $\Ext^1_{A(\Gamma)}(M,N)$ and $\Ext^1_{A(\Gamma)}(N,M)$ are both zero.
\item The arcs $\alpha$ and $\delta$ do not cross outside of 3-cycles.
\end{enumerate}
\end{lemma}
\begin{proof}
By Lemma \ref{Lem: standard maps and standard extensions}, we have that the groups $\Ext^1_{A(\Gamma)}(M,N)$ and $\Ext^1_{A(\Gamma)}(N,M)$ are both trivial if and only if there are no standard extensions in $\Ext^1_{A(\Gamma)}(M,N)$ or $\Ext^1_{A(\Gamma)}(N,M)$.  By the observations preceding this lemma, this arises exactly when $\alpha$ and $\delta$ do not cross outside of 3-cycles.
\end{proof}

%%%%%%%
%
\subsection{Band modules with $n \neq 1$.}\label{Sec: higher bands}

In this section we will show that, for a finite or asymptotic arc $\alpha$, a scalar $\lambda \in \k^*$ and $n\in \mathbb{N}\cup\{\infty, -\infty\}$, we have that the homomorphisms in $\mathrm{K}(A(\Gamma))$ between the string complex $P_\alpha^\bullet$ and a band complex $B_{\lambda, n}^\bullet$ are determined by the homomorphisms between $P_\alpha^\bullet$ and $B_{\lambda, 1}^\bullet$.  In order to do this we will follow the arguments of \cite[Sec.~5]{ALP}.

Since $\k$ is algebraically closed, we associate the following unfolded diagram to the complexes $B_{\lambda,n}^\bullet$ with $n\in \mathbb{N}$ as in \cite[Sec.~5]{ALP}.

\[ \xymatrix@R=1mm@C=13mm{
 &-1&0&-1&0& & \\
\dots & \bullet \ar[l]_{\lambda m_t} \ar[r]^{n_0} \ar[dl]^<<{m_t} & \bullet & \bullet \ar[l]_{m_0} \ar[r]^{n_1} & \bullet & \dots \ar[l]_{m_1} & \text{layer n}\\
\dots & \bullet  \ar@{}[d]|<<{\vdots} \ar@{}[d]|<<{\vdots} \ar[l]_<<<<<{\lambda m_t} \ar[r]^{n_0}  & \bullet \ar@{}[d]|<<{\vdots} \ar@{}[d]|<<{\vdots} & \bullet \ar@{}[d]|<<{\vdots} \ar[l]_{m_0} \ar[r]^{n_1} & \bullet \ar@{}[d]|<<{\vdots} & \dots \ar[l]_{m_1} & \text{layer n-1} \ar@{}[d]|<<{\vdots}\\
&&&&&&{}\\
\dots & \bullet \ar[l]_<<<<<{\lambda m_t} \ar[r]^{n_0} \ar[dl]^<<{m_t} & \bullet & \bullet \ar[l]_{m_0} \ar[r]^{n_1} & \bullet & \dots \ar[l]_{m_1} & \text{layer 2}\\
\dots & \bullet \ar[l]_<<<<<{\lambda m_t} \ar[r]^{n_0} & \bullet & \bullet \ar[l]_{m_0} \ar[r]^{n_1} & \bullet & \dots \ar[l]_{m_1} & \text{layer 1}
}\]

\noindent In the same way we associate the following unfolded diagrams to $B_{\lambda,\infty}^\bullet$ and $B_{\lambda, -\infty}^\bullet$ respectively.

\[ \xymatrix@R=1mm@C=13mm{
 & \ar@{}[d]|<<{\vdots}& \ar@{}[d]|<<{\vdots}&  \ar@{}[d]|<<{\vdots} &  \ar@{}[d]|<<{\vdots} & &  \ar@{}[d]|<<{\vdots}\\
&&&&&&\\
\dots & \bullet \ar[l]_<<<<<{\lambda m_t} \ar[r]^{n_0} \ar[dl]^<<{m_t}  & \bullet & \bullet \ar[l]_{m_0} \ar[r]^{n_1} & \bullet & \dots \ar[l]_{m_1} & \text{layer 3} \\
\dots & \bullet \ar[l]_<<<<<{\lambda m_t} \ar[r]^{n_0} \ar[dl]^<<{m_t} & \bullet & \bullet \ar[l]_{m_0} \ar[r]^{n_1} & \bullet & \dots \ar[l]_{m_1} & \text{layer 2}\\
\dots & \bullet \ar[l]_<<<<<{\lambda m_t} \ar[r]^{n_0} & \bullet & \bullet \ar[l]_{m_0} \ar[r]^{n_1} & \bullet & \dots \ar[l]_{m_1} & \text{layer 1}\\
 &-1&0 &-1&0& & 
}\]

\[ \xymatrix@R=1mm@C=13mm{
 &-1&0 &-1&0& & \\
&&&&&&\\
\dots & \bullet \ar[l]_<<<<<{\lambda m_t} \ar[r]^{n_0} \ar[dl]^<<{m_t}  & \bullet & \bullet \ar[l]_{m_0} \ar[r]^{n_1} & \bullet & \dots \ar[l]_{m_1} & \text{layer 1} \\
\dots & \bullet \ar[l]_<<<<<{\lambda m_t} \ar[r]^{n_0} \ar[dl]^<<{m_t} & \bullet & \bullet \ar[l]_{m_0} \ar[r]^{n_1} & \bullet & \dots \ar[l]_{m_1} & \text{layer 2}\\
\dots & \bullet\ar@{}[d]|<<{\vdots} \ar[l]_<<<<<{\lambda m_t} \ar[r]^{n_0}  & \bullet  \ar@{}[d]|<<{\vdots} & \bullet\ar@{}[d]|<<{\vdots} \ar[l]_{m_0} \ar[r]^{n_1} & \bullet\ar@{}[d]|<<{\vdots} & \dots \ar[l]_{m_1} & \text{layer 3} \ar@{}[d]|<<{\vdots}\\
&&&&&&
}\]

\begin{definition}[cf.~{\cite[Def.~5.2]{ALP}}]
Consider a finite or asymptotic arc $\alpha$ in $(\mathrm{S}, \mathrm{M})$ and let $\lambda \in \k^*$, $n\in \mathbb{N}\cup\{\infty, -\infty\}$.   

A morphism $f$ in $\Hom_{\mathrm{K}(A(\Gamma))}(P_\alpha^\bullet, B_{\lambda,n}^\bullet[1])$ is said to be \textbf{lifted from the morphism $g$ in $\Hom_{\mathrm{K}(A(\Gamma))}(P_\alpha^\bullet, B_{\lambda,1}^\bullet[1])$} if there exists an $m\in\mathbb{N}$ such that the components of $f$ that have codomain in layer $m$ of $B_{\lambda,n}^\bullet[1]$ are exactly the same as the components of $g$ and $f$ is the minimal such morphism.

A morphism $f$ in $\Hom_{\mathrm{K}(A(\Gamma))}(B_{\lambda,n}^\bullet, P_\alpha^\bullet[1])$ is defined to be \textbf{lifted from the morphism $g$ in $\Hom_{\mathrm{K}(A(\Gamma))}(B_{\lambda,1}^\bullet, P_\alpha^\bullet[1])$} analogously.
\end{definition}

The following lemma is adapted from \cite[Lem.~5.3]{ALP} and the argument applies easily to the infinite-dimensional case.

\begin{lemma}\label{Lem: can take 1-dim bands}
Let $\alpha$ be a finite or asymptotic arc in $(\mathrm{S}, \mathrm{M})$ and let $\lambda\in \k^*$ and $n\in \mathbb{N}\cup\{\infty, -\infty\}$.  Then the following statements hold.
\begin{enumerate}
\item If $f\neq 0$ is a morphism in $\Hom_{\mathrm{K}(A(\Gamma))}(P_\alpha^\bullet, \: B_{\lambda,n}^\bullet[1])$, then there exists a morphism $g$ in $\Hom_{\mathrm{K}(A(\Gamma))}(P_\alpha^\bullet,\: B_{\lambda,n}^\bullet[1])$ and $\mu \in \k^*$ such that $f = \mu g + f'$ and $g$ is lifted from a non-zero standard map in $\Hom_{\mathrm{K}(A(\Gamma))}(P_\alpha^\bullet, \: B_{\lambda,1}^\bullet[1])$.

\item If $f\neq 0$ is a morphism in $\Hom_{\mathrm{K}(A(\Gamma))}(B_{\lambda,n}^\bullet, \: P_\alpha^\bullet[1])$, then there exists a morphism $g$ in $\Hom_{\mathrm{K}(A(\Gamma))}(B_{\lambda,n}^\bullet, \: P_\alpha^\bullet[1])$ and $\mu \in \k^*$ such that $f = \mu g + f'$ and $g$ is lifted from a non-zero standard map in $\Hom_{\mathrm{K}(A(\Gamma))}(B_{\lambda,1}^\bullet, \: P_\alpha^\bullet[1])$.
\end{enumerate}
\end{lemma}

\begin{proposition}\label{Prop: can take 1-dim bands}
Let $\alpha$ be a finite or asymptotic arc in $(\mathrm{S}, \mathrm{M})$ and let $\lambda\in \k^*$ and $n\in \mathbb{N}\cup\{\infty, -\infty\}$.  Then the following statements hold. \begin{enumerate}
\item $\Ext^1_{A(\Gamma)}(M(\alpha), \: M_{\lambda,n}) = 0$ if and only if $\Ext^1_{A(\Gamma)}(M\alpha), \: M_{\lambda,1}) = 0$.
\item $\Ext^1_{A(\Gamma)}(M_{\lambda,n},\: M(\alpha)) = 0$ if and only if $\Ext^1_{A(\Gamma)}(M_{\lambda,1}, \: M(\alpha)) = 0$.
\end{enumerate}
\end{proposition}
\begin{proof}
We prove \textbf{(1)}, the proof of \textbf{(2)} is similar.  By Proposition \ref{Prop: standard iff non-zero}, we have that the $\k$-vector space $\Hom_{\mathrm{K}(A(\Gamma))}(P_\alpha^\bullet, \: B_{\lambda,1}^\bullet[1])$ is zero if and only if there are no standard maps in from $P_\alpha^\bullet$ to $B_{\lambda,1}^\bullet[1]$.  By Lemma \ref{Lem: can take 1-dim bands}, we also have that the $\k$-vector space $\Hom_{\mathrm{K}(A(\Gamma))}(P_\alpha^\bullet, \: B_{\lambda,n}^\bullet[1])$ is zero if and only if there are no standard maps from $P_\alpha^\bullet$ to $B_{\lambda,1}^\bullet[1]$.  The proposition follows immediately after observing the isomorphisms $\Ext^1_{A(\Gamma)}(M,N) \cong \Hom_{\mathrm{K}(A(\Gamma))}(P_M^\bullet, P_N^\bullet[1])$ and $\Ext^1_{A(\Gamma)}(N, M) \cong \Hom_{\mathrm{K}(A(\Gamma))}(P_N^\bullet, P_M^\bullet[1])$ given in Section \ref{Sec: string and band complexes}.
\end{proof}

%%%%%%%
%
\subsection{Proof of Theorem \ref{Thm: extensions app}.}

Following Sections \ref{Sec: string and band complexes} to \ref{Sec: higher bands}, we are now ready to give a proof of the main theorem of the appendix.

\begin{proof}[Proof of Theorem \ref{Thm: extensions app}]
If neither of $M$ and $N$ is isomorphic to $M_{\lambda,n}$ with $n\neq 1$, then the theorem holds by Lemma \ref{Lem: main theorem for n=1}.  Without loss of generality, assume that $M \cong M_{\lambda, n}$ for $n\neq 1$.  By Proposition \ref{Prop: can take 1-dim bands}, we have that $\Ext^1_{A(\Gamma)}(M,N) = 0 = \Ext^1_{A(\Gamma)}(N, M)$ exactly when $\Ext^1_{A(\Gamma)}(M_{\lambda,1},\: N) = 0 = \Ext^1_{A(\Gamma)}(N,\: M_{\lambda, 1})$.  Therefore we may apply Lemma \ref{Lem: main theorem for n=1} again to conclude that this is the case exactly when $\alpha = \beta$ and $\delta$ do not cross outside of 3-cycles.
\end{proof}

\end{appendices}

%%%%%%%%%

\bibliographystyle{plain}
\bibliography{bib}

\end{document}